\documentclass{article}
\usepackage{amsfonts}
\usepackage{amsmath}
\usepackage{amssymb}
\usepackage{amsthm}
\usepackage{multirow}
\usepackage[margin = 1.5in, top = 0.8in]{geometry}
\usepackage{hyperref}
\usepackage{natbib}
\hypersetup{
    colorlinks=true,
    linkcolor=blue,
    filecolor=magenta,      
    urlcolor=blue,
    citecolor= red
}
\usepackage{mathrsfs}

\def\abs#1{\left|#1\right|}

\def\argmin{{\arg\min}}

\def\bA{\mathbf{A}}

\def\bP{\mathbf{P}}

\def\bR{\mathbf{R}}

\def\bU{\mathbf{U}}
\def\bV{\mathbf{V}}

\def\bX{\mathbf{X}}

\def\bZ{\mathbf{Z}}

\def\bbB{\mathbb{B}}

\def\bbE{\mathbb{E}}

\def\bbI{\mathbb{I}}

\def\bbP{\mathbb{P}}

\def\bbR{\mathbb{R}}

\def\bDelta{{\boldsymbol\Delta}}

\def\bLambda{{\boldsymbol\Lambda}}
\def\bepsilon{{\boldsymbol\epsilon}}
\def\bvarepsilon{{\boldsymbol\varepsilon}}

\def\bPi{{\boldsymbol\Pi}}
\def\bSigma{{\boldsymbol\Sigma}}

\def\bsr{{\boldsymbol{r}}}

\def\ba{\mathbf{a}}
\def\bb{\mathbf{b}}

\def\be{\mathbf{e}}

\def\bt{\mathbf{t}}
\def\bu{\mathbf{u}}
\def\bv{\mathbf{v}}
\def\bw{\mathbf{w}}
\def\bx{\mathbf{x}}
\def\by{\mathbf{y}}

\def\balpha{{\boldsymbol\alpha}}
\def\bbeta{{\boldsymbol\beta}}
\def\bgamma{{\boldsymbol\gamma}}
\def\bdelta{{\boldsymbol\delta}}

\def\bfeta{{\boldsymbol\eta}}

\def\bmu{{\boldsymbol\mu}}

\def\bxi{{\boldsymbol\xi}}

\def\brho{{\boldsymbol\rho}}

\def\sfD{{\sf D}}

\def\sfN{{\sf N}}

\def\Ber{{\sf Ber}}

\def\sfd{{\sf d}}

\def\cA{\mathcal{A}}

\def\cD{\mathcal{D}}
\def\cE{\mathcal{E}}

\def\cG{\mathcal{G}}

\def\cI{\mathcal{I}}
\def\cJ{\mathcal{J}}
\def\cK{\mathcal{K}}
\def\cL{\mathcal{L}}
\def\cM{\mathcal{M}}
\def\cN{\mathcal{N}}

\def\cP{\mathcal{P}}

\def\cS{\mathcal{S}}
\def\cT{\mathcal{T}}
\def\cU{\mathcal{U}}

\def\ccA{\mathscr{A}}

\def\ccC{\mathscr{C}}

\def\ccE{\mathscr{E}}
\def\ccEstar{{\mathscr{E}^{*}}}

\def\KL{{\sf KL}}

\def\norm#1{\left\|#1\right\|}

\def\pr{{\bbP}}

\def\d{{\rm d}}

\def\tr{{\sf tr}}

\def\var{{\sf var}}

\def\op{{\rm op}}
\def\diam{{\sf diam}}

\def\emptyset{\varnothing}
\def\tol{\varepsilon_{n,p}}

\def\Cov{\text{cov}}
\def\Ber{{\sf {Ber}}}

\def\sfN{{\sf N}}

\def\pen{{\rm pen}}

\def\col{\text{col}}
\def\dprime{{\prime\prime}}
\def\tprime{{\dprime\prime}}
\def\kl{{\rm kl}}

\def\param{{\rm par}}
\def\glm{{\rm glm}}
\def\best{{\rm best}}
\def\ctT{{\widetilde{\cT}}}
\def\ctG{{\widetilde{\cG}}}
\def\bone{{\boldsymbol{1}}}

\def\R{{\mathbb{R}}}

\usepackage{bbm}
\def\ind{{\mathbbm{1}}}

\newcommand{\floor}[1]{\left\lfloor#1\right\rfloor}
\mathchardef\mhyphen="2D
\usepackage{times}
\usepackage{graphicx, wrapfig}
\usepackage{subcaption}
\usepackage{enumitem}

\newtheorem{theorem}{Theorem}
\newtheorem{lemma}{Lemma}
\newtheorem{definition}{Definition}
\newtheorem{corollary}{Corollary}
\newtheorem{assumption}{Assumption}
\newtheorem{remark}{Remark}
\newtheorem{example}{Example}
\newtheorem{condition}{Condition}
\theoremstyle{plain}
\theoremstyle{remark}
\newtheorem{proposition}{Proposition}

\newcount\Comments  
\usepackage{color}
\newcommand{\kibitz}[2]{\ifnum\Comments=1\textcolor{#1}{#2}\fi}
\usepackage{color}
\definecolor{darkgreen}{rgb}{0,0.5,0}
\definecolor{purple}{rgb}{1,0,1}

\newcommand{\sapta}[1]{\kibitz{purple}     {[SR: #1]}}

\begin{document}


\title{Understanding Best Subset Selection: A Tale of Two C(omplex)ities}

\author{Saptarshi Roy $^\star$\quad Ambuj Tewari $^\star$\quad Ziwei Zhu $^\dagger$\\
$^\star$ University of Michigan, Ann Arbor, USA \\ $^\dagger$ Radix Trading, Chicago, USA}
\maketitle
\begin{abstract}
We consider the problem of best subset selection (BSS) under high-dimensional sparse linear regression model. 
Recently, \cite{guo2020best} showed that the model selection performance of BSS depends on a certain \emph{identifiability margin}, a measure that captures the model discriminative power of BSS under a general correlation structure that is robust to the design dependence, unlike its computational surrogates such as LASSO, SCAD, MCP, etc. Expanding on this, we further broaden the theoretical understanding of best subset selection in this paper and 
show that the complexities of the \emph{residualized signals}, the portion of the signals orthogonal to the true active features, and \emph{spurious projections}, describing the projection operators associated with the irrelevant features,  also play fundamental roles in characterizing the margin condition for model consistency of BSS. In particular, we establish both necessary and sufficient margin conditions depending only on the identifiability margin and the two complexity measures. We also partially extend our sufficiency result to the case of high-dimensional sparse generalized linear models (GLMs).
\end{abstract}






\section{Introduction}
\label{sec: Introduction}
\sapta{Introduction}
Variable selection in high-dimensional sparse regression has been one of the central topics in statistical research over the past few decades. 
Consider $n$ observations $\{(\bx_i,y_i)\}_{i=1}^n$ following the linear model:
\begin{equation}
\label{eq: base model}
y_i = \bx_i^\top \bbeta +  \varepsilon_i, \quad i \in \{1, \ldots,n\},
\end{equation}
where $\{\bx_i\}_{i \in [n]}$ are \textit{fixed} $p$-dimensional feature vectors, $\{\varepsilon_i\}_{i \in [n]}$ are i.i.d. \textit{mean-zero} noise, and the signal vector $\bbeta \in \bbR^p$ is unknown but is assumed to have a sparse support. 
In matrix notation, the observations can be represented as
\begin{equation*}
    \by = \bX\bbeta + \bvarepsilon,
\end{equation*}
 where $\by = (y_1, \ldots, y_n)^\top$, $\bX = (\bx_1,\ldots, \bx_n)^\top$, and $\bvarepsilon = (\varepsilon_1, \ldots, \varepsilon_n)^\top$. We consider the standard \emph{high-dimensional sparse} setup where $n< p$, and possibly $n \ll p$, and
the vector $\bbeta$ is sparse in the sense that $\norm{\bbeta}_0 := \sum_{j =1}^p \ind(\beta_j \neq 0) = s$, which is much smaller than $p$. 
In this paper, we focus on the variable selection problem, i.e., identifying the active set $\cS:= \{j : \beta_j \neq 0\}$. We primarily use the 0-1 loss, i.e.,  $\pr(\widehat
{\cS}\neq \cS)$, to assess the quality of the selected model $\widehat{\cS}$.

One of the well-studied methods for variable selection in high-dimensional sparse regression is to penalize the empirical risk by model complexity, thereby encouraging sparse solutions. Specifically, consider
\[
\widehat{\bbeta}^{\rm pen}:= \argmin_{\bbeta \in \bbR^p} \, \cL(\bbeta) + \pen_\lambda(\bbeta),
\]
were $\cL(\bbeta)$ is a loss function and $\pen_\lambda(\bbeta)$ is the penalization term that controls the model complexity. Classical methods such as AIC \citep{akaike1974new, akaike1998information}, BIC \citep{schwarz1978estimating}, Mallow's $C_p$ \citep{mallows2000some} use model complexity as penalty term, i.e., $\ell_0$-norm of the regression coefficient, to penalize the negative log-likelihood. Although these methods enjoy nice sampling properties \citep{barron1999risk, zhang2012general}, such $\ell_0$ regularized methods are known to suffer from huge computational bottleneck \citep{foster2015variable}. This motivated a whole generation of statisticians to develop alternative penalization methods such as LASSO \citep{tibshirani1996regression}, SCAD \citep{fan2001variable}, MC+ \citep{zhang2010nearly}, and many others that have both strong statistical guarantees and computational expediency. 

However, after recent computational advancements in solving BSS \citep{bertsimas2016best, bertsimas2020sparse, zhu2020polynomial}, there has been growing acknowledgment that BSS enjoys significant statistical superiority over its computational surrogates and has inevitably motivated statisticians to investigate the properties of BSS. For example, through extensive simulations, 
\cite{hastie2020best} shows that
BSS performs better than LASSO in high signal-to-noise ratio regime in terms of the prediction risk. \cite{Jain2014iterative} showed
that a wide family of iterative hard thresholding (IHT) algorithms can approximately solve the BSS
problem, in the sense that they can achieve similar goodness of fit with the best subset with slight
violation of the sparsity constraint. \cite{liu2020between} studied the optimal thresholding operator for such iterative thresholding algorithms, which manages to exploit fewer variables
than IHT to achieve the same goodness fit as BSS. Recently, \cite{she2023slowkill} proposed an algorithmic framework based on quantile-thresholding that iteratively optimizes $\ell_2$-penalized BSS objective function and can achieve model consistency under certain regularity conditions on the design.
On the theoretical side, \cite{guo2020best} showed that the model selection behavior of BSS does not explicitly depend on the restricted eigenvalue condition for the design \citep{bickel2009simultaneous, van2009conditions}, a condition which appears unavoidable (assuming a standard computational complexity conjecture) for any polynomial-time method 
\citep{zhang2014lower}. Specifically, they show that  BSS is robust to design collinearity. 
Under a particular asymptotic regime and independent design, \cite{roy2022high} further established information-theoretic optimality of BSS in terms of precise constants for the signal strength parameter under weak and heterogeneous signal regimes.



In this paper, we also study the variable selection property of BSS and identify novel quantities that are fundamental to understanding the model consistency of BSS. 
 Specifically, we take the geometric alignment of the feature vectors  $\{\bX_j\}_{j \in [p]}$ into consideration to produce a more refined analysis of BSS, and show that on top of a certain identifiability margin introduced in \cite{guo2020best}, the following two geometric quantities also control the model selection performance of BSS: (a) Geometric complexity of the space of \emph{residualized signals}, and (b) Geometric complexity of \emph{spurious projections}. We show the explicit dependence of these two complexity measures in our main results and demonstrate the interplay between the margin condition and the underlying geometric structure of the features through some illustrative examples. In the process, we also point out the existence of a design that is more favorable to BSS than the orthogonal design, which is commonly believed to be the easiest case for model selection. To the best of our knowledge, this is the first work that identifies the underlying geometric complexity of the feature space as a governing force behind the performance of BSS. 

 The rest of the paper is organized as follows. In Section \ref{sec: intro BSS} we discuss the preliminaries of BSS. Section \ref{sec: two complexities} is devoted to the discussion of the key quantities, i.e., identifiability margin and the two complexities. In particular, Section \ref{sec: identifiability margin}-\ref{sec: space of projections} carefully introduce the notion of identifiability margin discussed in \cite{guo2020best} and the two novel complexity measures. In Section \ref{sec: corr and complexities}, we build intuition for understanding the effect of these two complexities with varying correlation. In Section \ref{sec: main results}, we present both sufficient (Section \ref{sec: sufficient condition}) and necessary (Section \ref{sec: necessary condition for BSS}) conditions for model consistency of BSS. We also partially extend our result to GLMs and present a similar sufficiency result for model consistency in Section S2 of the supplementary material.

\paragraph{Notation.} Let $\R$ denote the set of real numbers. Denote by $\R^p$ the $p$-dimensional Euclidean space and by $\R^{p\times q}$ the space of real matrices of order $p\times q$. For a positive integer $K$, denote by $[K]$ the set $\{1, 2, \ldots, K\}$. 

Regarding vectors and matrices, for a vector $v \in \R^p$, we denote by $\norm{v}_2$ the $\ell_2$-norm of $v$. We use $\bbI_{p}\in \R^{p \times p}$ to denote the $p$-dimensional identity matrix, and $\mathbf{1}_p\in \bbR^p$ denotes the $p$-dimensional vector with all entries equal to 1.
For a matrix $\bA \in \R^{p\times p}$, we denote by $\bA_j$ and $\ba_j$ the $j$th column and the transposed $j$th row of $\bA$ respectively. We use $\col(\bA)$ and $\col(\bA)^\perp$ to denote the columnspace of $\bA$ and its orthogonal complement respectively.

Let $(M, d)$ be a metric space where $M$ is a set endowed with the metric $d$. For a subset $T\subseteq M$, we denote by $\cN(T, d, \varepsilon)$ the $\varepsilon$-covering number of $T$. Similarly, we denote by $\cM(T, d, \varepsilon)$ the $\varepsilon$-packing number of $T$.

Throughout the paper, let $O(\cdot)$ (respectively $\Omega(\cdot)$) denote the standard big-O (respectively big-Omega) notation, i.e., we say $a_n = O(b_n)$ if there exists a universal constant $C>0$, such that $a_n \leq C b_n$ (respectively $a_n\geq C b_n$) for all $n \in \mathbb{N}$. Sometimes for notational convenience, we write $a_n \lesssim b_n$ in place of $a_n = O(b_n)$ and $a_n \gtrsim b_n$ in place of $a_n = \Omega(b_n)$. We write $a_n \asymp b_n$ if $a_n = O(b_n)$ and $a_n = \Omega(b_n)$. Finally, we write $a_n \sim b_n$ if $\lim_{n \to \infty}a_n/b_n  = 1$, and $a_n = o(b_n)$ if $\lim_{n \to \infty} a_n/b_n = 0$.

\renewcommand{\arraystretch}{2}
\setlength{\tabcolsep}{5pt}
\begin{table}[h]
\scriptsize
    \centering
    \caption{Notations for the geometric sets and corresponding complexity measures.}
    \vspace{0.1in}
    \begin{tabular}{ |c|c|c| }
    \hline
    Geometric Sets  & Residualized signals (scaled) & Spurious projections\\
    \hline
    \hline
     Notation & $\widehat{\bgamma}_{\cD}$ for all $\cD \in  \ccA_{\widehat{s}}$ (Eq. \eqref{eq: T_Is}) & $\bP_\cD - \bP_{\cD \cap \cS}$ for all $\cD \in  \ccA_{\widehat{s}}$ (Eq. \eqref{eq: ortho proj} and \eqref{eq: G_Is})\\ 
    Collection sets & $\cT_\cI^{(\widehat{s})}$ such that $\cI = \cS \cap \cD$ (Eq. \eqref{eq: T_Is}) & $\cG_\cI^{(\widehat{s})}$ such that $\cI = \cS \cap \cD$ (Eq. \eqref{eq: T_Is})\\
    \hline
    Scaled complexities& $\ccE_{\cT_\cI^{(\widehat{s})}}$ & 
    $ \ccE_{\cG_\cI^{(\widehat{s})}}$ 
    \\
    {Diameter} & 
    $\sfD_{\cT_\cI^{(\widehat{s})}} $ & $\sfD_{\cG_\cI^{(\widehat{s})}} $\\
    {Minimal separation}& $\sfd_{\cT_\cI^{(\widehat{s})}} $ & $\sfd_{\cG_\cI^{(\widehat{s})}} $\\
    \hline
    \end{tabular}
    \label{tab: notations}
\end{table}

\paragraph{Brief summary of main contributions.}
Before diving into the mathematical details, we first lay out a brief summary of the main results of the paper. As mentioned before, the main contribution of the paper is to identify the role of two quantities related to the complexities of the residualized signals and the spurious projections in the model recovery performance of BSS. To elaborate more on this, we revisit a specific identifiability margin $\tau_*(s)$ (introduced in \cite{guo2020best} and Equation \eqref{eq: identifiability margin}) that essentially captures the joint effect of the signal strength and the collinearity among the true and spurious features arising from mutual correlation between them. If the minimum signal strength $\beta_{\min}:= \min \{\abs{\beta_j} : j \in \cS\}$ is large and the correlation between true and spurios features are small, then $\tau_*(s)$ is large, which makes it easier for BSS to identify the true support $\cS$. Next, we introduce the  two complexity measures of the set of residualized signals $\cT_\cI^{(s)}$ and the set of spurious projections $\cG_\cI^{(s)}$ (see Table \ref{tab: notations}):
\begin{enumerate}
    \item  \textit{Complexity of residualized signals}: In Section \ref{sec: space of unexplained signals}, we introduce the complexity measure for the class of residualized signals $\bgamma_\cD$ (see Equation \eqref{eq: residualized signals}) originating from the part of the true signal $\bX_\cS \bbeta_\cS$ that can not be \textit{linearly explained} by a model $\cD$ with $\cD \cap \cS = \cI$. To be prices, we consider a scaled log-entropy integral of the space of resulting unit vectors $\widehat{\bgamma}_\cD$ which we denote by $\ccE_{\cT_\cI^{(s)}}$ (see Equation \eqref{eq: upper complexity cT}).

    \item \textit{Complexity of spurious projections}: Section \ref{sec: space of projections} introduces the complexity measure for the the spurious projection operators $\bP_\cD - \bP_\cI$ (see Equation \eqref{eq: ortho proj} and \eqref{eq: G_Is}) which are essentially the orthogonal projection matrices onto the subspaces generated by the part of the column spaces of $\bX_\cD$ that is orthogonal to the column space of $\bX_\cS$ for all $\cD$ with $\cD \cap \cS = \cI$. Similar to the previous case, the proposed complexity measure is a scaled version of the log-entropy (under operator norm) integral of the space of spurious projection (see Equation \eqref{eq: upper complexity cG}) which si denoted by $\ccE_{\cG_\cI^{(s)}}$.
\end{enumerate}
Both of this complexities can be linked to the diameter of the set of residualized signals $\cT_\cI^{(s)}$ and the spurious projections $\cG_\cI^{(s)}$ respectively under proper choices of metrics. The important thing to notice is that the intrinsic complexity of the sets might be much smaller compared to the vanilla complexity measure that is often associated with just the cardinality of the sets depending on the structural properties of the design matrix $\bX$. We elaborate through the following example

\begin{example}
  Let us consider the case when the dimension $p = 3$, i.e, $$\bX = [\be_1, (\be_1 + \delta \be_2)/(\sqrt{1 + \delta^2}), (\be_1 + \delta \be_3)/(\sqrt{1 + \delta^2})]\in \bbR^{n \times 3}$$ where $\be_1, \be_2, \be_3$ are the first three canonical basis of  $\bbR^n$, and $\delta \approx 0$. Also assume that the true support is $\cS = \{1\}$. For any other candidate model $\cD$ of unit size, we have $\cD \cap \cS = \emptyset$. Therefore, we have the set of spurious projections to be $\cG_\emptyset^{(1)}:=\{\bP_{\{2\}}, \bP_{\{3\}}\}$. For small enough $\delta$, we have $\norm{\bP_{\{2\}} - \bP_{\{3\}}}_\op = \delta/\sqrt{2}$, and hence $\log \cN(\cG_\emptyset^{(1)}, \norm{\cdot}_\op, \delta^\prime) = 0 $ for all $\delta^\prime \ge \delta$. Therefore, overall complexity of the spurious projections, which is integral of the log-entropy, is much smaller compared to the log-cardinality of the set which is $\log 2$ in this case. A more detailed study can be found in Section \ref{sec: equi-corr design example}.
\end{example}
Therefore, considerations of these quantities will result into sharper results on the margin conditions required for $\tau_*(s)$ (see Theorem \ref{thm: sufficiency of BSS}) in order to have model consistency of BSS. Informally, one of our our main results Theorem \ref{thm: sufficiency of BSS} shows that 
\[
\frac{\tau_*(s)}{\sigma^2} \gtrsim \max \left\{\max_{\cI \subset \cS}\ccE^2_{\cT_{\cI}^{(s)}}, \max_{\cI \subset \cS}\ccE^2_{\cG_{\cI}^{(s)}}\right\} \frac{\log p}{n}
\]
is sufficient for model consistency of BSS. The above condition reveals an interesting interplay between the two complexity measures that shows that \textit{only the set of dominating complexity measure} characterizes the model recovery performance. Moreover, the above condition allows us to artificially construct an example with a specific correlation structure in the design matrix which is more favorable for BSS compared to the independent Gaussian design (see Section \ref{sec: illustartive examples}) case which is popularly believed to be the most easy setting for model selection. Finally, in Theorem \ref{thm: necessary condition BSS}, we also present a somewhat similar necessary condition that also involves the complexity measures of the residualized signals and spurious projections. In the supplementary material, we also present an extension of Theorem \ref{thm: sufficiency of BSS} under the GLM case, i.e., we also provide similar complexity measures tailored to the GLM models under some regularity assumptions on the design and the link function.

\section{Best subset selection}
\label{sec: intro BSS}
We briefly review the preliminaries of BSS, one of the most classical variable selection approaches. For a given sparsity level $\widehat{s}$, BSS solves for

\[
\hat{\bbeta}_{\rm best} (\widehat{s}) :=  \argmin_{\bbeta \in \bbR^p, \norm{\bbeta}_0 \leq  \widehat{s}} \norm{\by - \bX\bbeta}_2^2.
\]
For model selection purposes, we can choose the best fitting model to be $\widehat{\cS}_\best(\widehat{s}):= \{j :[\hat{\bbeta}_\best(\widehat{s})]_j \neq 0\}$. 
For a subset $\cD \subseteq [p]$, define the matrix $\bX_\cD:= (\bX_j; j \in \cD)$. In addition, we denote by $\bP_\cD$ the orthogonal projection operator onto the column space of $\bX_{\cD}$, i.e., 
\begin{equation}
\label{eq: ortho proj}
\bP_{\cD}:= \bX_{\cD}(\bX_{\cD}^\top \bX_{\cD})^{-1} \bX_{\cD}^\top.
\end{equation}
Next, we define the corresponding residual sum of squares (RSS) for model $\cD$ as 
\[
R_\cD : = \by^\top (\bbI_n - \bP_\cD) \by.
\]
With this notation, the $\widehat{\cS}_\best(\widehat{s})$ can be alternatively written as 
\begin{equation}
\widehat{\cS}_{\rm best}(\widehat{s}) :=  \argmin_{\cD \subseteq [p]: \abs{\cD} \leq \widehat{s}} \, R_\cD.
\label{eq: BSS optimization}
\end{equation}
Given any candidate model $\cD \subset [p]$, we can rewrite the model \eqref{eq: base model} as 

\[
\by = \bX_{\cS}\bbeta_{\cS} + \bvarepsilon = \bP_\cD \bX_\cS \bbeta_\cS + (\bbI_n - \bP_\cD) \bX_{\cS \setminus \cD} \bbeta_{\cS\setminus \cD} + \bvarepsilon.
\]
The term $(\bbI_n - \bP_\cD) \bX_{\cS\setminus \cD} \bbeta_{\cS \setminus \cD}$ is the residual part of the signal that can not be linearly explained by $\bX_\cD$. We refer to this part as the \textit{residualized signals}. We can thus measure the discrimination between the true model $\cS$ and a different candidate model $\cD$ through
the quantity $n^{-1}\Vert(\bbI_n - \bP_\cD) \bX_{\cS \setminus \cD} \bbeta_{\cS\setminus \cD}\Vert_2^2$.

Let $\widehat{\bSigma}:= n^{-1} \bX^\top \bX$ be the sample covariance matrix and for any two sets $\cD_1, \cD_2 \subset [p]$, $\widehat{\bSigma}_{\cD_1, \cD_2}$ denotes the submatrix of $\bSigma$ with row indices in $\cD_1$ and column indices in $\cD_2$. Next, we define the collection $\ccA_{\widehat{s}}:= \{\cD \subset [p] : \cD \neq \cS, \abs{\cD} = \widehat{s}\}$, and for $\cD \in \ccA_{\widehat{s}}$
write 
\[\Gamma(\cD) = \widehat{\bSigma}_{ \cS \setminus \cD, \cS \setminus \cD} -  \widehat{\bSigma}_{ \cS \setminus \cD, \cD}\widehat{\bSigma}_{\cD, \cD}^{-1} \widehat{\bSigma}_{\cD, \cS \setminus \cD}.\]
In the Gaussian design case, the above quantity can be indemnified as the empirical version of the conditional variance-covariance matrix $\Cov(\bX_{\cS \setminus \cD} \mid \bX_\cD)$.  Therefore, $\Gamma(\cD)$ roughly captures the degree correlation between the features in $\bX_{\cS \setminus \cD}$ and $\bX_{\cD}$.
To understand the above quantity more clearly, note that \[\bbeta_{\cS\setminus \cD}^\top\Gamma(\cD)\bbeta_{\cS \setminus \cD} = n^{-1}\Vert(\bbI_n - \bP_\cD) \bX_{\cS \setminus \cD} \bbeta_{\cS\setminus \cD}\Vert_2^2 = n^{-1}\norm{(\bbI_n - \bP_\cD) \bX_{\cS} \bbeta_{\cS}}_2^2.\]
This shows that $\bbeta_{\cS\setminus \cD}^\top\Gamma(\cD)\bbeta_{\cS \setminus \cD}$ captures the goodness-of-fit for model $\cD$.
Intuitively, if $\bbeta_{\cS\setminus \cD}^\top\Gamma(\cD)\bbeta_{\cS \setminus \cD}$ is very close to zero, then there exists $\bb\in \bbR^{\abs{\cD}}$ such that $\bX_\cS \bbeta_\cS \approx \bX_\cD \bb$. Hence, $\cS$ and $\cD$ have similar linear explanatory power, and the true model $\cS$ becomes practically indistinguishable from $\cD$.  
In fact, the following lemma shows that $\bbeta_{\cS \setminus \cD}^\top \Gamma(\cD) \bbeta_{\cS \setminus \cD}$ needs to be at least bounded away from 0 for all $\cD \in \ccA_{\widehat{s}}$ to make $\cS$ identifiable.

\begin{lemma}
    \label{lemma: identifiability}
    For any given $\widehat{s}>0$, if there exists a $\cD \in \ccA_{\widehat{s}}$ such that $\bbeta_{\cS \setminus \cD}^\top \Gamma(\cD) \bbeta_{\cS \setminus \cD} = 0$, then there exists $\bb\in \bbR^{\widehat{s}}$ such that $\bX_\cS \bbeta_\cS = \bX_\cD \bb$. Hence, both $\bX_\cS \bbeta_\cS$ and $\bX_\cD \bb$ generates the same probability distribution for $\by$, and $\cS$ becomes non-identifiable.
\end{lemma}

Now we are ready to introduce the identifiability margin that characterizes the  \emph{model discriminative power} of BSS and the two complexity measures.

\section{Identifiability margin and two complexities}
\label{sec: two complexities}

\subsection{Identifiability margin}
\label{sec: identifiability margin}
The discussion in Section \ref{sec: intro BSS} motivates us to define the following \emph{identifiability margin}:
\begin{equation}
    \label{eq: identifiability margin}
    \tau_*(\widehat{s}) := \min_{\cD \in \ccA_{\widehat{s}}} \frac{ \bbeta_{\cS \setminus \cD}^\top \Gamma(\cD) \bbeta_{\cS \setminus \cD}}{\abs{\cS \setminus \cD}}.
\end{equation}
If we define $\ccA_{\widehat{s},k}: = \{\cD \in \ccA_{\widehat{s}} : \abs{
\cS \setminus \cD} = k\}$, then the above can be rewritten as
\[
 \tau_*(\widehat{s}) = \min_{k \in [\widehat{s}]}\min_{\cD \in \ccA_{\widehat{s},k}} \frac{ \bbeta_{\cS \setminus \cD}^\top \Gamma(\cD) \bbeta_{\cS \setminus \cD}}{k}.
\]
As mentioned earlier, the quantity $\tau_*(\widehat{s})$ captures the model discriminative power of BSS. To add more perspective, note that if the features are highly correlated among themselves then it is expected that $\tau_*(\widehat{s})$ is very close to $0$. Hence, any candidate model $\cD$ is practically indistinguishable from the actual model $\cS$ which in turn makes the problem of exact model recovery harder. On the contrary, if the features are uncorrelated then $\tau_*(\widehat{s})$ becomes bounded away from 0 making the true model $\cS$ easily recoverable.
For example,  \cite{guo2020best} showed that under the condition 
\begin{equation}
    \label{eq: margin condition zz}
    \tau_*(s) \gtrsim \sigma^2 \frac{\log p}{n},
\end{equation}
BSS is able to achieve model consistency. In general, Condition \eqref{eq: margin condition zz} is less restrictive than the well known $\beta$-min condition which demands
\[
a:= \min_{j \in \cS} \abs{\beta_j} \gtrsim \sigma \left(\frac{\log p}{n}\right)^{1/2}.
\]
To see this, let $\hat{\lambda}_m := \min_{\cD \in \ccA_s}\lambda_{\min} \left(\Gamma(\cD)\right)$ where $\lambda_{\min} \left(\Gamma(\cD)\right)$ denotes the minimum eigenvalue of $\Gamma(\cD)$, and note that $\tau_*(s) \ge \hat{\lambda}_m a^2$. Thus a sufficient condition for \eqref{eq: margin condition zz} to hold is $a \gtrsim \sigma \{\log p/(n \hat{\lambda}_m)\}^{1/2}$. In comparison, \cite{zhang2012general} showed that the $\ell_0$-regularized least square estimator is able to achieve model consistency when $a \gtrsim \sigma \{\log p/(n \kappa_-)\}^{1/2}$, where $\kappa_-:= \min_{\cD : \abs{\cD} \le s, \cD \subset [p]} \lambda_{\min}(\widehat{\bSigma}_\cD)$. The latter condition is very sensitive to the feature correlation as $\kappa_-$ can vary drastically depending on the degree of correlation between the features. In contrast, $\hat{\lambda}_m$ is robust against design dependence;  rather, it
reflects how spurious variables can approximate the true model, which implies much less restriction
than that induced by $\kappa_-$. For more details on the identifiability margin, we point the readers to Section 2.1 of \cite{guo2020best}. 
From now on, unless otherwise mentioned,  we will assume that the margin quantity $\tau_*(\widehat{s})>0$ to avoid the non-identifiability issue as pointed out in Lemma \ref{lemma: identifiability}. 
\sapta{More content from Ziwei's paper}

Next, we will shift focus on the underlying geometric structures of two spaces that govern the difficulty of the BSS problem \eqref{eq: BSS optimization}. We essentially identify the complexities of two types of sets that control the hardness of BSS: (i) the set of residualized signals, and (ii) the set of spurious projections. We discuss these two sets, and the associated complexities in detail below and Table \ref{tab: notations} compiles important notations and quantities related to the aforementioned sets. 

\subsection{ Complexity of residualized signals}
\label{sec: space of unexplained signals}
We start with the definition of the residualized signal. For a candidate model $\cD\in \ccA_{\widehat{s}}$, define  
\begin{equation}
\label{eq: residualized signals}
\bgamma_\cD:= n^{-1/2} (\bbI_n - \bP_\cD)\bX_{\cS \setminus \cD} \bbeta_{\cS \setminus \cD},
\end{equation}
and the corresponding unit vector $\widehat{\bgamma}_\cD:= \bgamma_{\cD}/ \norm{\bgamma_\cD}_2$. Note that $\widehat{\bgamma}_\cD$ is well-defined as 
\begin{figure}
  \begin{center}    \includegraphics[width=0.5\textwidth]{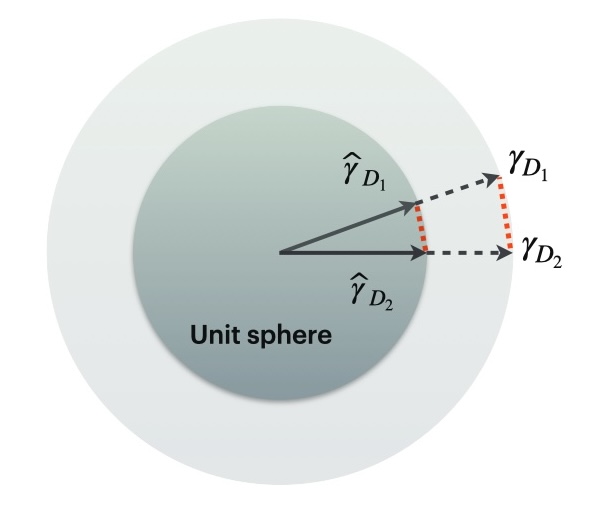}
  \end{center}
  \caption{Distance between $\widehat{\bgamma}_{\cD_1}$ and $\widehat{\bgamma}_{\cD_2}$ correctly captures the angular separation.}
  \label{fig: gamma}
\end{figure}
$\norm{\bgamma_\cD}_2^2\geq \tau_*(\widehat{s})>0$. As mentioned before, $\bgamma_\cD$ represents 
the part of the signal that can not be linearly explained by the features in model $\cD$. Note that the margin condition \eqref{eq: margin condition zz} essentially tells that the vectors $\bgamma_\cD$ are well bounded away from the origin. However, this property does not quite capture the degree of their radial spread in $\bbR^n$. It may happen that despite being well bounded away from the origin,
the vectors are clustered along one common unit direction. For example, in Figure \ref{fig: gamma}, the distance between $\bgamma_{\cD_1}$ and $\bgamma_{\cD_2}$ are large as the vectors are at a larger distance from the origin, although the angular separation between them is small. 
To capture this notion of separation within the vectors $\{\bgamma_\cD\}_{\cD \in \ccA_{\widehat{s}}}$, we also need to capture the spatial alignment of their corresponding unit vectors $\{\widehat{\bgamma}_\cD\}_{\cD \in \ccA_{\widehat{s}}}$.
This motivates us to consider the geometric complexities of this set of unit vectors. Specifically, for a set $\cI \subset \cS$, we define 

\begin{equation}
\label{eq: T_Is}
\cT_\cI^{(\widehat{s})}:= \{ \widehat{\bgamma}_\cD  : \cD \in \ccA_{\widehat{s}}, \cS \cap \cD = \cI  \} \subseteq \bbR^n,
\end{equation}
which is the set of all the normalized forms of the residualized signals corresponding to the models $\cD\in \ccA_{\widehat{s}}$ with $\cI$ as the common part with true model $\cS$. To capture the complexity of these spaces, we look at the scaled entropy integral
\begin{equation}
    \label{eq: upper complexity cT}
    \ccE_{\cT_\cI^{(\widehat{s})}}:= \frac{\int_0^\infty \sqrt{\log \cN(\cT_\cI^{(\widehat{s})}, \norm{\cdot}_{2}, \varepsilon)} \;\d\varepsilon } {\sqrt{\log \vert \cT_\cI^{(\widehat{s})}\vert}} .
\end{equation}
The numerator in the above display is commonly known as entropy integral which captures the topological complexity of $\cT_\cI^{(\widehat{s})}$. In literature, this quantity has a connection to the well-known \textit{Talagrand's complexity} \citep{talagrand2005generic}, which often comes up in controlling the expectation of the supremum of Gaussian processes \citep{krahmer2014suprema, lifshits1995gaussian, adler2007random}. In this paper, we look at the above scaled version of the entropy integral which allows us to compare the quantity with the diameter and minimum pairwise distance between the elements of the set $\cT_\cI^{(\widehat{s})}$. To elaborate on this point, define the diameter and minimum pairwise distance of $\cT_\cI^{(\widehat{s})}$ as follows:
\[
\sfD_{\cT_\cI^{(\widehat{s})}} := \max_{\bu, \bv \in \cT_\cI^{(\widehat{s})}} \norm{\bu - \bv}_2,\quad \text{and} \quad \sfd_{\cT_\cI^{(\widehat{s})}} := \min_{\bu, \bv \in \cT_\cI^{(\widehat{s})}} \norm{\bu - \bv}_2.
\]
Now notice the following two simple facts:
\[
\log \cN(\cT_\cI^{(\widehat{s})}, \norm{\cdot}_2, \sfD_{\cT_{\cI}}) = 0,\quad \log \cN(\cT_\cI^{(\widehat{s})}, \norm{\cdot}_2, \sfd_{\cT_{\cI}}) =  \log \abs{\cT_\cI^{(\widehat{s})}}. 
\]
Noting that $\log \cN(\cT_\cI^{(\widehat{s})}, \norm{\cdot}_2, \delta)$ is a decreasing function over $\delta$, we finally get 
\[
\sfd_{\cT_\cI^{(\widehat{s})}} \le \ccE_{\cT_\cI^{(\widehat{s})}}
\leq \sfD_{\cT_\cI^{(\widehat{s})}}.
\]
This shows that the quantity $\ccE_{\cT_\cI^{(\widehat{s})}}$ roughly captures the average separation of the set $\cT_\cI^{(\widehat{s})}$. In our subsequent discussion, we will show that $\ccE_{\cT_\cI^{(\widehat{s})}}$ heavily influences the margin condition for exact recovery. This is indeed an important observation, as the complexity of the set of residualized signals depends heavily on the association between the features. For example, they can differ vastly for highly correlated designs compared to almost uncorrelated designs. Hence, the effect of $\ccE_{\cT_\cI^{(\widehat{s})}}$ on the exact model recovery also varies significantly across different classes of distributions and leads to sharper margin conditions for exact model recovery.

\subsection{Complexity of spurious projections \sapta{Keep in Mind}} 
\label{sec: space of projections}
\sapta{point out the no dependence on Y}
In this section, we will introduce the space of projection operators that also control the level of difficulty of the true model recovery. Similar to the previous section, for a fixed set $\cI \subset \cS$, we consider the set
\begin{equation}
\label{eq: G_Is}
\cG_\cI^{(\widehat{s})}:= \{\bP_\cD - \bP_{\cI} : \cD\in \ccA_{\widehat{s}}, \cS \cap \cD = \cI\} \subseteq \bbR^{n \times n}.
\end{equation}
It is a well-known fact that every projection operator of 
the form $\bP_\cD - \bP_{\cI} \in \cG_\cI^{(\widehat{s})}$ has a one-to-one correspondence with the subspace $\text{col}(\bX_\cD) \cap \text{col}(\bX_{\cI})^\perp$.
Thus, $\cG_\cI^{(\widehat{s})}$ can be thought of as the collection of all linear subspaces of the form $\text{col}(\bX_\cD) \cap \text{col}(\bX_{\cI})^\perp$, which is essentially the set of spurious features that can not be linearly explained by the set of features in model $\cI \subset \cS$.  To capture the proper measure of complexity of the set $\cG_\cI^{(\widehat{s})}$, it is crucial to induce the space of projection operators with a proper metric. It turns out that Grassmannian distance is the correct distance to consider in this context. Specifically,  for two linear subspaces $ U, V$ we look at their \emph{maximum sin-theta distance}:

\[
\d(U, V) := \norm{\bPi_{U} - \bPi_{V}}_\op,
\]
where $\bPi_{U}, \bPi_{V}$ are the orthogonal projection operators of $U, V$ respectively.
It turns out that $\d(U, V)$ evaluates the trigonometric sine function at the maximum principal angle between the subspaces $U$ and $V$ (see Figure \ref{fig: subspace distance}).
We point the readers to \cite{ye2016schubert} for a more detailed discussion on this topic. 
\begin{figure}
  \begin{center}    \includegraphics[width=0.5\textwidth]{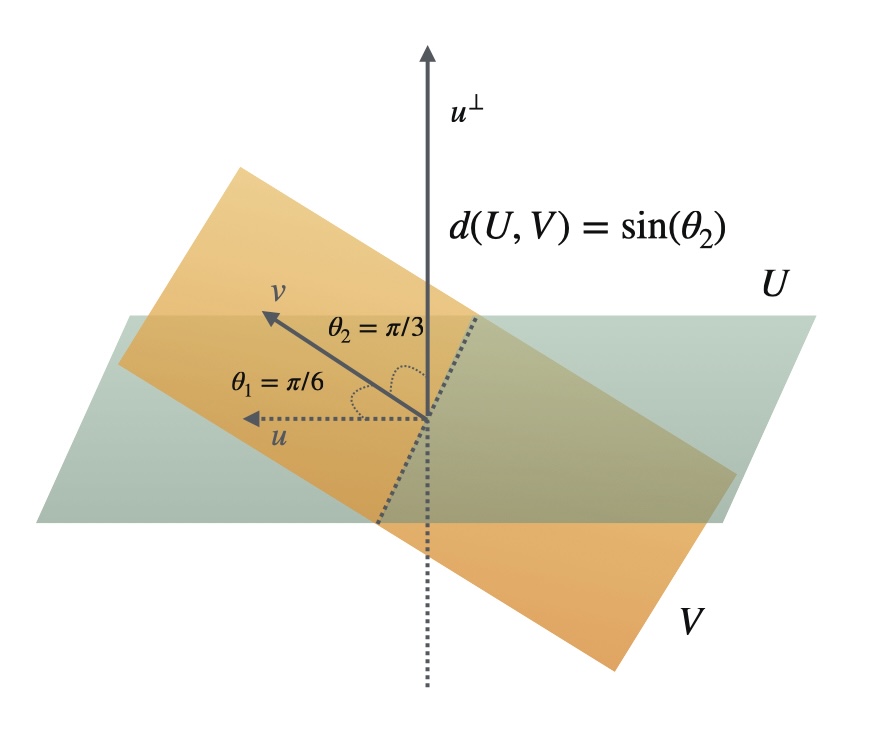}
  \end{center}
  \caption{The figure shows the two principal angles between two subspaces $U$  and $V$. $\{u, u^\perp\}$ are the two orthonormal basis of $U$, and $v$ is an orthonormal basis of $V$. $\theta_2$ is the maximum principal angle between $U$ and $V$.}
  \label{fig: subspace distance}
\end{figure}

Under this distance, we define the scaled entropy integral as
\begin{equation}
    \label{eq: upper complexity cG}
    \ccE_{\cG_\cI^{(\widehat{s})}}:= \frac{\int_0^\infty \sqrt{\log \cN(\cG_\cI^{(\widehat{s})}, \norm{\cdot}_{\op}, \varepsilon)} \;\d\varepsilon}{\sqrt{\log \vert \cG_\cI^{(\widehat{s})}\vert}}.
    \end{equation}
Note that, unlike 
$\ccE_{\cT_\cI^{(\widehat{s})}}$, the complexity measure $\ccE_{\cG_\cI^{(\widehat{s})}}$ has no dependence on $\bbeta$ or the residualized signal. Thus, $\ccE_{\cG_\cI^{(\widehat{s})}}$ roughly captures the geometric complexity of only the spurious features. In fact, via a similar argument as in Section \ref{sec: space of unexplained signals}, it can be shown that 
$
\sfd_{\cG_\cI^{(\widehat{s})}} \le \ccE_{\cG_\cI^{(\widehat{s})}} \leq \sfD_{\cG_\cI^{(\widehat{s})}},
$
where
\[
\sfd_{\cG_\cI^{(\widehat{s})}} := \min_{\bU, \bV \in \cG_\cI^{(\widehat{s})}} \norm{\bU - \bV}_\op, \quad \sfD_{\cG_\cI^{(\widehat{s})}} := \max_{\bU, \bV \in \cG_\cI^{(\widehat{s})}} \norm{\bU - \bV}_\op.
\]
Thus, $\ccE_{\cG_\cI^{(s)}}$ only captures the separability in the set of subspaces generated by the spurious features. The main motivation behind considering such quantity is to capture the influence of the effective size of the set $\{\cG_\cI^{(\widehat{s})}\}_{\cI \subset \cS}$ in the analysis of BSS. A naive union bound only uses $\abs{\cG_\cI^{(\widehat{s})}} = \binom{p-\widehat{s}}{\widehat{s} - \abs{\cI}}$ as a measure of complexity of the set $\cG_\cI^{(\widehat{s})}$. This is rather loose, as the effective complexity of the set is much smaller if $\ccE_{\cG_\cI^{(\widehat{s})}}$ is small. Thus, taking $\ccE_{\cG_\cI^{(\widehat{s})}}$ into account unravels a broader picture of the effect incurred by the underlying geometry of the feature space.


\subsection{Correlation and complexities}
\label{sec: corr and complexities}
From the discussion on the two complexities, it is quite evident that both of the complexity measures heavily rely on the alignment of the feature vectors $\{\bX_j: j \in [p]\}$, which directly depends on the correlation structure among the features in the model. \sapta{maybe give some examples} Below, we discuss how these two types of complexities may vary with correlation among the features. 

\paragraph{Correlation and spurious projection operators:}
We first focus on the set $\cG_\cI^{(\widehat{s})}$, as it is relatively easy to understand its behavior across different correlation structures. Recall that for a fixed choice of $\cI$, the set $\cG_\cI^{(\widehat{s})}$ is the collection of all the projection operators of the form $\bP_\cD - \bP_{\cI}$ for all $\cD \in \ccA_{\widehat{s}}$, which can be thought of as the collection of different subspaces generated by the linear combination of the spurious features. If the spurious features are highly correlated then these subspaces may be essentially indistinguishable from each other, i.e., the mutual distance between the projection operators $\{\bP_\cD - \bP_{\cI}\}_{\cD \in \ccA_{\widehat{s}}}$ is significantly smaller compared to the case when they are weakly correlated. As an example, let us consider the equi-correlated Gaussian design, i.e., the row vectors $\{x_i\}_{i\in [n]}$ of $\bX$ in \eqref{eq: base model} follows i.i.d. mean-zero Gaussian distribution with covariance matrix

\begin{equation*}
\bSigma = (1-r) \bbI_p + r \mathbf{1}_p\mathbf{1}_p^\top.
\end{equation*}
For the sake of simplicity, we also assume that the true model is a singleton set. In particular, we consider $\cS = \{1\}$ and set $\widehat{s} = 1$. Also, note that in this case $\ccA_{\widehat{s}} = \{j\in [p] : j \neq 1\}$ and $\cI = \emptyset$. Under this setup, we have $\cG_\emptyset^{(1)} = \{\bX_j \bX_j^\top/\norm{\bX_j}_2^2: j \notin \cS\}$ and
$
n^{-1}\norm{\bX_{j} - \bX_k}_2^2 \approx  2(1 - r), 
$
for all $j,k\neq 1$.
If $r$ is very close to 1 in the above display, then it follows that the vectors $\{\bX_j/\sqrt{n}\}_{j \neq 1}$ are extremely clustered towards each other, and as a result, the spurious projection operators are also very close to each other in operator norm. Due to this, the complexity measure $\ccE_{\cG_\emptyset^{(1)}}$ becomes extremely small and the subspaces become almost indistinguishable. In contrast, when the features are approximately uncorrelated, i.e., $r \approx 0$, the scaled features $\{\bX_j/\sqrt{n}\}_{j\neq 1}$ are roughly orthogonal. In that case the 
\[
n^{-1}\norm{\bX_{j} - \bX_k}_2^2 \approx  2, \quad \text{for all $j,k\neq 1$}.
\]
As an example, for $p=6$ and $s=1$,
Figure~\ref{fig: angle spurious subspaces} illustrates a similar phenomenon in 3-dimension. Figure \ref{fig: angle spurious subspaces}(a) clearly shows that for the case $r=0$ the angle is larger compared to the $r=0.9$ case in Figure \ref{fig: angle spurious subspaces}(b).
\begin{figure}
    \centering
    \subfloat[\centering ]{{\includegraphics[width=0.5\textwidth]{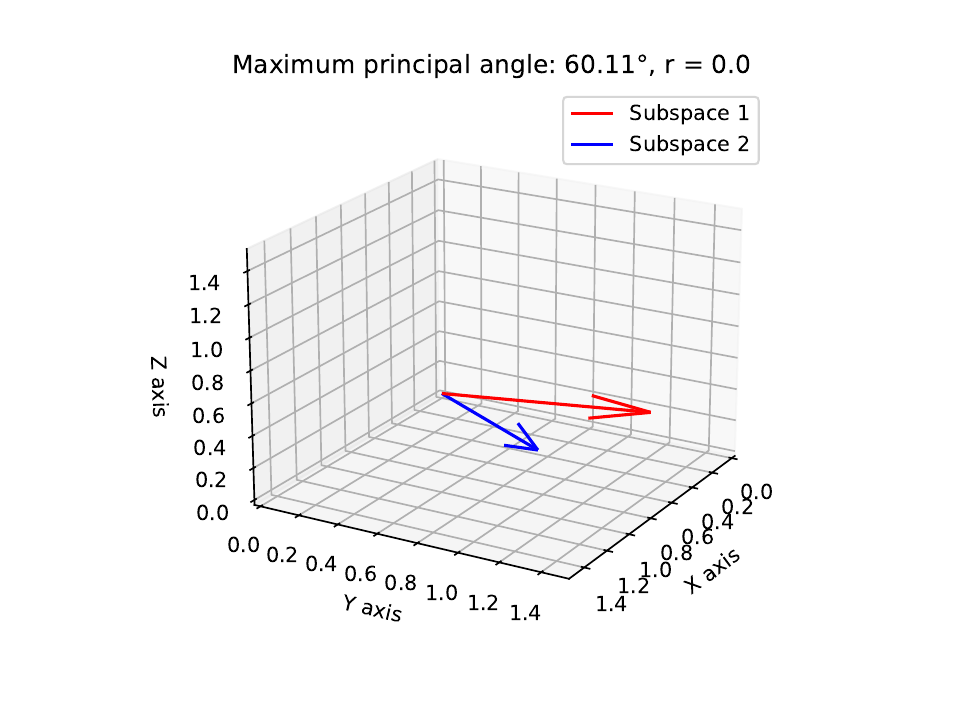} }}%
    \subfloat[\centering  ]{{\includegraphics[width=0.5\textwidth]{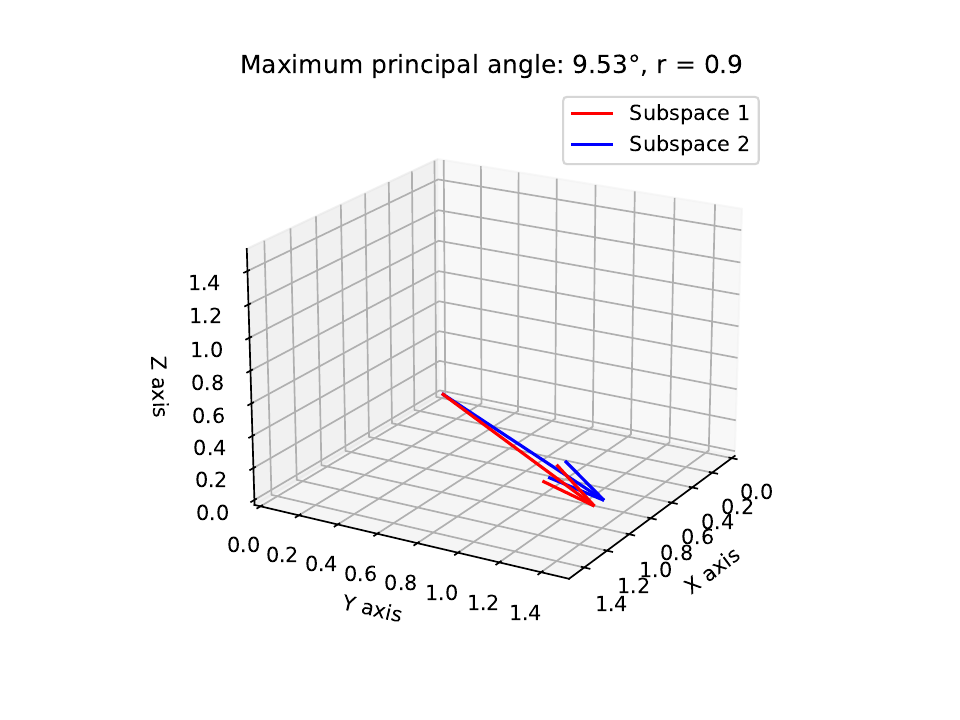} }}%
    \caption{(a) shows the angle between spurious features in $\cG^{(1)}_\emptyset$ for $r = 0$, (b) shows the angle between spurious features in $\cG^{(1)}_\emptyset$ for $r = 0.9$.}
    \label{fig: angle spurious subspaces}
\end{figure}
This suggests that the linear spans generated by each of the set of features $\{n^{-1/2}\bX_j\}_{j\neq 1}$ are well separated and $\ccE_{\cG_\emptyset^{(1)}}$ is well bounded away from zero. Thus, it follows that the features are well spread out in $\bbR^n$. This phenomenon indicates that a higher correlation may aid the model recovery chance for BSS by reducing the search space over the features. As we will see in our subsequent discussion in Section \ref{sec: illustartive examples}, the correlation between noise variables can significantly help BSS to identify the correct model. Specifically, we construct an example where the true variables are uncorrelated with the noise variables and show that a high correlation among noise variables helps BSS to identify the correct model.
The intuition is that under the presence of correlation, the diversity of the elements in $\cG_{\cI}^{(\widehat{s})}$ gets reduced as $\ccE_{\cG_{\cI}^{(\widehat{s})}}$ becomes small. Thus, BSS needs to search on a comparatively smaller feature space rather than searching over all possible $\binom{p-\widehat{s}}{\widehat{s}- \abs{\cI}}$ models, which in turn aids the probability of finding the correct model out of the other candidate ones. Thus, the smaller complexity of $\cG_{\cI}^{(\widehat{s})}$ counteracts the adverse effect of correlation to some degree, and it may improve the model recovery performance of BSS. 

However, it is not necessary that the complexity will be small for a highly correlated structure. For example, in the above case, if $\cS = \{1,2\}$, then $\cG^{(1)}_\emptyset$ might be closer to 1, i.e., the maximum principal angles between the subspaces are closer to $90^\circ$ 
 as shown in Figure \ref{fig:angle btw 2d subspaces}. Therefore, the complexity solely depends on the distribution and correlation structure of the design matrix and may behave differently on a case-by-case basis.
\begin{figure}
    \centering
    \includegraphics[width=0.8\linewidth]{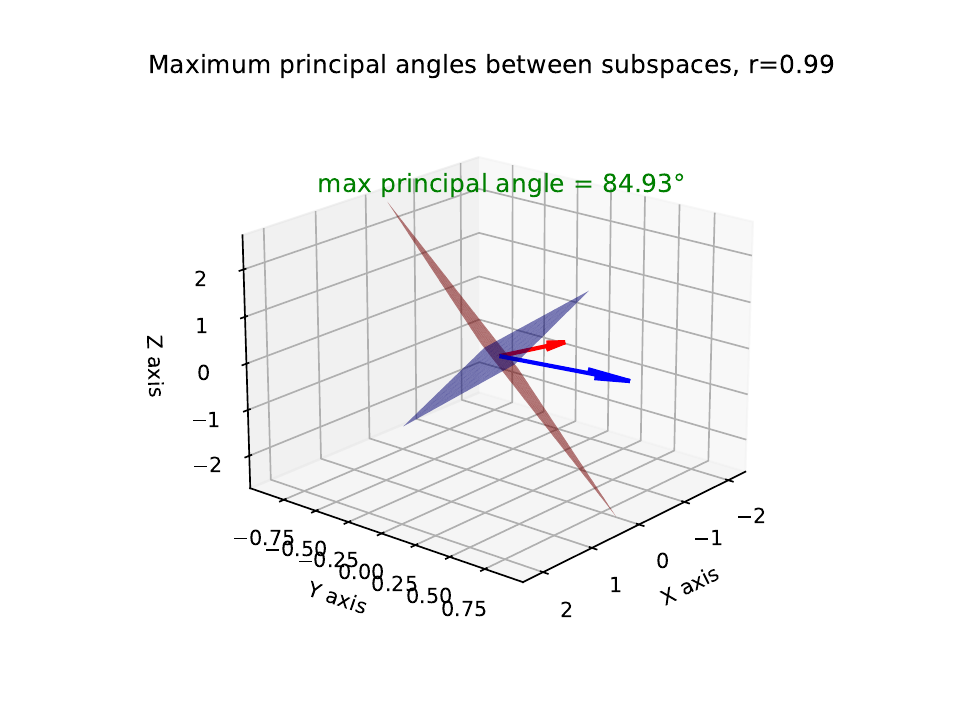}
    \caption{Maximum principal angle between the subspaces $U:=\text{span}\{\bX_1, \bX_2\}$ and $V:= \text{span}\{\bX_3, \bX_4\}$ for $r = 0.99$ under equi-correlated Gaussian design.}
    \label{fig:angle btw 2d subspaces}
\end{figure}

\paragraph{Correlation and residualized signals:}
Now we shift our focus to understanding the behavior of the set of normalized residualized signals denoted by $\cT_{\cI}^{(\widehat{s})}$. Recall that for a fixed $\cI$, the set $\cT_{\cI}^{(\widehat{s})}$ denotes the collection of all the unit vectors $\widehat{\bgamma}_\cD$ (defined in Section \ref{sec: space of unexplained signals}) such that $\cD\cap \cS = \cI$. Similar to $\cG_{\cI}^{(\widehat{s})}$, the complexity of the set $\cT_{\cI}^{(\widehat{s})}$ also depends on the correlation structure among the features. To elaborate more on this, we revisit the example of equi-correlated Gaussian design with correlation parameter $r$ and $\cS = \{1\}$. We denote by $\bP_j$ the orthogonal projection operator onto the span of $\bX_j$, i.e.,
$
\bP_j = \bX_j \bX_j^\top/\norm{\bX_j}_2^2.
$
Similar to the previous section, in this case also the set $\cT_\emptyset^{(1)}$ consists of the scaled residualized signals that take the following form for large $n$ with high probability:
\[
\widehat{\bgamma}_j = \frac{(\bbI_n - \bP_j)\bX_1}{\norm{(\bbI_n - \bP_j)\bX_1}_2} \approx \frac{\bX_1 - r \bX_j}{\norm{\bX_1 - r \bX_j}_2},\quad \text{for all $j\neq 1$.}
\]
Also, note that
\[
\widehat{\gamma}_j^\top \widehat{\gamma}_k \approx \frac{1 - 2 r^2 + r^3}{1 - r^2} =: f(r).
\]
Since, $f(r)$ is a strictly decreasing function on $[0,1)$, and $\norm{\widehat{\bgamma}_j - \widehat{\bgamma}_k}_2^2 = 2(1 - \widehat{\bgamma}_j^\top \widehat{\bgamma}_k)$, it follows that $\sfd_{\cT_\emptyset}\geq 1/2$ when $r$ is very close to 1. On the contrary, when $r\approx 0$, the above display suggests that $\sfD_{\cT_\emptyset^{(1)}} \approx 0$, i.e., for uncorrelated design, the complexity $ \ccE_{\cT_\emptyset^{(1)}}$ of the set $\cT_\emptyset^{(1)}$ is smaller compared to the highly correlated case which is in sharp contrast with the behavior of $\ccE_{\cG_\emptyset^{(1)}}$. 

However, it is worth pointing out that the above property of $\ccE_{\cT_\emptyset^{(1)}}$ is very specific to the above considered model. There may exist a correlated structure where higher correlation among noise variables does not increase $\ccE_{\cT_\emptyset^{(1)}}$ (see Section \ref{sec: illustartive examples}), and improves the chance of identifying the correct model via BSS. However, understanding such a phenomenon for a more general design could be significantly more challenging.

\section{Theoretical properties of BSS}
\label{sec: main results}
\subsection{Model selection consistency of BSS under known sparsity}
\label{sec: sufficient condition}
This section illustrates the interaction between the identifiability margin \eqref{eq: identifiability margin} and the two complexities that characterize the sufficient condition for the exact model recovery. From here on, we assume that the true sparsity is known, i.e., we set $\widehat{s} = s$ in \eqref{eq: BSS optimization}, and BSS searches the best model out of all possible models of size $s$. We now introduce a technical assumption that essentially prevents the noisy features from becoming highly correlated with the true features:

\begin{assumption}
\label{assumption: noisy features are not too correlated}
The design matrix $\bX$ enjoys the following property:
\[
\min_{\cI \subset \cS} \ccE_{\cG_{\cI}^{(s)}} > \{\log(ep)\}^{-1/2}.
\]
\end{assumption}
The above assumption ensures that the noisy features are distinguishable enough from the active features in order for BSS to identify the active features. To see this, consider the case when the noise variables are highly correlated with the true features $\{\bX_j\}_{j \in \cS}$. In this case, the projection operator $\bP_\cD - \bP_{\cI}$ can be written as $(\bbI_n - \bP_{\cI}) \bP_\cD$ for all $\cD \in \cG_\cI^{(s)}$, whenever $\cI \neq \emptyset$. As the features in $\{\bX_j : j \in \cD \setminus \cS\}$ are highly correlated with $\bX_\cI$, it follows that $\norm{\bP_\cD - \bP_\cI}_{\op} \approx 0$ and by triangle inequality it follows that $\norm{\bP_\cD - \bP_{\cD^\prime}}_{\op} \approx 0$ for any two candidate models $\cD$ and $\cD^\prime$ such that $\cD \cap \cS = \cD^\prime \cap \cS = \cI$. Thus, Assumption \ref{assumption: noisy features are not too correlated} gets rid of such cases by indirectly controlling the correlation between the active features and noisy features. Secondly, the assumption also enforces diversity among the noise variables in the following sense: If the features $\{\bX_j : j \notin \cS^c\}$ are too similar to each other, then also $\ccE_{\cG_\cI^{(s)}}$ shrinks towards 0. Thus, Assumption \ref{assumption: noisy features are not too correlated} prevents the noise variables from becoming extremely correlated with each other.

Assumptions with similar spirits are fairly common in the literature on high-dimensional statistics. For example, the well-known Sparse Riesz Condition (SRC) \citep{zhang2008sparsity} assumes that there exist positive numbers $\kappa_-, \kappa_+$ and $\Psi \geq 1$ such that
\begin{equation}
\label{eq: SRC condition}
\kappa_- \leq  \frac{\norm{\bX \bv}_2^2}{n}\leq \kappa_+, \quad \text{for all $\bv \in  \{\bu \in \bbR^p: \norm{\bu}_2 = 1, \norm{\bu}_0\leq \Psi s\}$.}
\end{equation}
The above SRC condition controls the maximum and minimum eigenvalues of all the models of size $s$, which essentially prevents the features to become extremely correlated with each other. In comparison, Assumption \ref{assumption: noisy features are not too correlated} is much weaker than SRC condition 
in two aspects. First,
 unlike the SRC,  Assumption \ref{assumption: noisy features are not too correlated} imposes conditions only over $(2^s - 2)$ models, whereas SRC imposes conditions on $\Omega((p/s)^{\floor{\Psi s}})$ many models.
 Second, the lower bound requirement in Assumption \ref{assumption: noisy features are not too correlated} is rather weak as the bound decays with increasing ambient dimension and allows a higher degree of correlation among the features. In other words, SRC condition \eqref{eq: SRC condition} implies the condition in Assumption \ref{assumption: noisy features are not too correlated}, and we formalize this claim in the following proposition.

 \begin{proposition}
     Let the columns of $\bX$ be normalized, i.e., $\norm{\bX_j}_2 = \sqrt{n}$. Also, assume that there exist positive constants $\kappa_-, \kappa_+$ such that the SRC condition \eqref{eq: SRC condition} holds with $\Psi = 2$. Then the condition in Assumption \ref{assumption: noisy features are not too correlated} also holds for large enough $p$, i.e., $\min_{\cI \subset \cS }\ccE_{\cG_\cI^{(s)}} \geq \kappa_-/\kappa_+ \gg \{\log(ep)\}^{-1/2}$. Furthermore, the implication in the other direction is not true in general.
    \end{proposition}

Next, we assume that the noise in model \eqref{eq: base model} is sub-Gaussian. In particular, we assume the following:
\begin{assumption}
\label{assumption: noise}
    We assume that the noise $\{\varepsilon_i\}_{i \in [n]}$  in model \eqref{eq: base model}
 are i.i.d. \textit{mean-zero} $\sigma$-sub-Gaussian noise, i.e., $\max_{i \in [n]}\bbE \exp(t \varepsilon_i) \le \exp(\sigma^2 t^2/2)$ for all $t \in \bbR$.
\end{assumption}

Now we are ready to state our main sufficiency result.
\begin{theorem}[Sufficiency]
    \label{thm: sufficiency of BSS}
    Under Assumption \ref{assumption: noisy features are not too correlated} and Assumption \ref{assumption: noise}, there exists a positive universal constant $C_0$ such that for any $0\leq \eta <1$, whenever the identifiability 
    margin $\tau_*(s)$ satisfies 
    \begin{equation}
    \label{eq: margin cond}
    \begin{aligned}
&\frac{\tau_*(s)}{\sigma^2} \geq\\ & \frac{C_0}{(1- \eta)^2} \left[\max\left\{
\max_{\cI \subset \cS} \ccE_{\cT_\cI^{(s)}}^2, \max_{\cI \subset \cS} \ccE_{\cG_\cI^{(s)}}^2 
\right\}+ \sqrt{\frac{\log(es)\vee \log\log(ep)}{\log(ep)}}\right] \frac{\log(ep)}{n},
\end{aligned}
    \end{equation}
    we have 
 \[
\{\widehat{\cS}_{\rm best}(s)\} \subseteq \left\{
 \widehat{\cS}: \vert\widehat{\cS}\vert = s, R_{\widehat{\cS}} \leq  \min_{\cD
 \in \ccA_s } R_{\cD}+ n \eta \tau_*(s) 
 \right\} = \{\cS\},
 \]
 with probability at least $1 - O(\{s \vee \log p\}^{-1})$. 
 In particular, we have $\cS = \argmin_{\cD\in \ccA_s} R_{\cD}$ with high probability.
\end{theorem}


The proof of the above theorem is present  {\color{black}Section S1.3} of the supplementary material.
The above theorem gives a sufficient condition for BSS to achieve model consistency. The above theorem states that under the margin condition \eqref{eq: margin cond} the true model $\cS$ is the optimizer of the BSS problem. Furthermore, the parameter $\eta$ quantifies the magnitude of the sub-optimality gap. For $\eta>0$, the above theorem shows that $R_\cD - R_\cS > n \eta \tau_*(s)$ for any $\cD \in \ccA_s$, i.e, the gap between the optimal RSS value $R_\cS$ and the next smallest RSS is more pronounced for larger values of $\eta$. However, this is more demanding than just the requirement for $R_S$ being the optimal value, and hence the margin condition \eqref{eq: margin cond} is more stringent for $\eta>0$ compared to $\eta = 0$ case.

Next, note that the margin condition 
\eqref{eq: margin cond} involves the identifiability margin $\tau_*(s)$ and the two complexities associated with the sets of residualized signals and spurious projection operators. This condition reveals an interesting interplay between the identifiability margin and the two complexities. To highlight this phenomenon, it is instructive to consider the case when the true model $\cS = \{1\}$ and $\bX_1$ is orthogonal to the spurious features $\{\bX_j\}_{j \neq 1}$. However, the spurious features are allowed to be extremely correlated to each other. As mentioned in the independent block design example in Section \ref{sec: illustartive examples}, in this case, both of the two complexities are small for higher correlation among the spurious features, whereas $\tau_*(s)$ remains roughly unaffected by the strength of correlation. Thus, the margin condition \eqref{eq: margin cond} becomes less stringent with increasing strength of correlation, and the performance of BSS should improve. To illustrate this phenomenon, we consider a simulation setup with $p = 2000, n = 500$, and $s = 1$. We generate $\bX$ from independent Gaussian block design mentioned in Section \ref{sec: illustartive examples} with the cross-correlation $c = 0$, and $r \in [0,1)$ being the correlation within the noise variables. Thus, $r=0$ corresponds to the independent Gaussian design. We set $\bbeta = (0.1, 0,\ldots, 0)^\top \in \bbR^p$, and the errors $\{\varepsilon_i\}_{i \in [n]}$ are generated in i.i.d. fashion from $\sfN(0,1)$. Finally, the response $\by$ is generated according to model \eqref{eq: base model}. Assuming $s$ is known, we use ABESS \citep{zhu2020polynomial} as a fast computational surrogate for BSS. The left panel of Figure \ref{fig: simul} shows that the mean model recovery rate of ABESS (across 20 independent runs) increases as the correlation between the noise variables increases to 1, which validates the findings in Theorem \ref{thm: sufficiency of BSS}. The right panel of Figure \ref{fig: simul} also shows that a similar phenomenon is true even for  $s> 1$.

\begin{figure}[h!]
     \centering
     \begin{subfigure}[b]{0.45\textwidth}
         \centering
         \includegraphics[scale = 0.57]{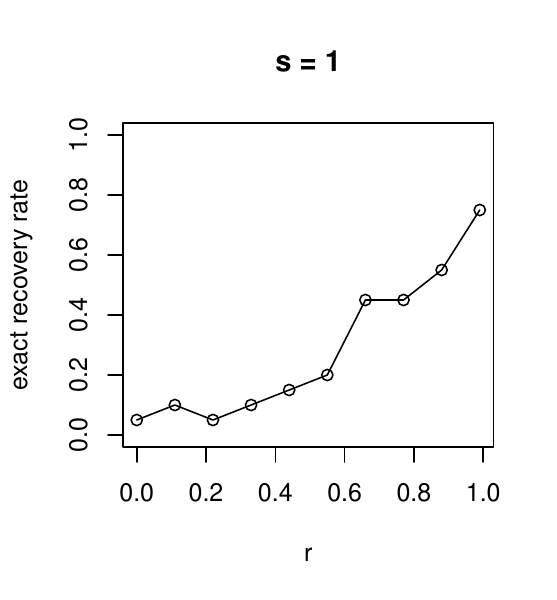}
         \label{fig: regions}
     \end{subfigure}
     \hfill
     \begin{subfigure}[b]{0.45\textwidth}
         \centering
         \includegraphics[scale = 0.57]{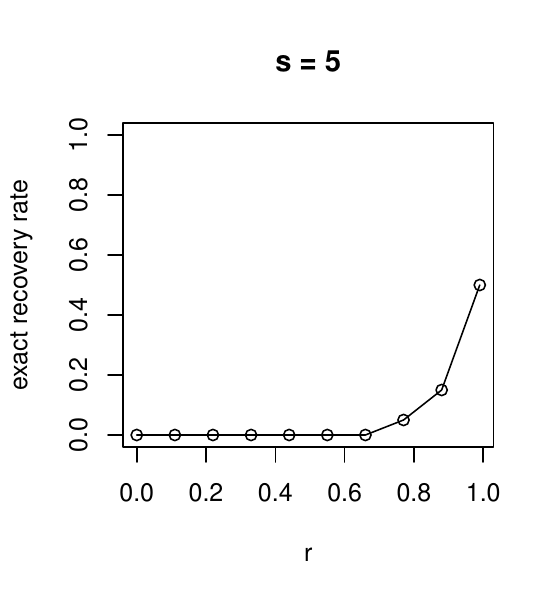}
     \end{subfigure}
        \caption{Model recovery rate of ABESS under independent block design.}
        \label{fig: simul}
\end{figure}

On the other hand, as mentioned in Remark \ref{remark: equi-cor}, under equicorrelated model with $\cS = \{1\}$ and high correlation, $\ccE_{\cT_\emptyset^{(1)}}$ remains strictly bounded away from 0 and it dominates $\ccE_{\cG_\emptyset^{(1)}}$.  However, the identifiability margin $\tau_*(s)$ becomes very small due to the high correlation between the true and noise variables. Hence, the margin condition \eqref{eq: margin cond} becomes harder to satisfy with increasing correlation. In the case of independent design, it turns out $\ccE_{\cG_\emptyset^{(1)}}$ is the dominating complexity measure. This is not surprising as under independent design, the features are more spread out in the feature space compared to correlated design, whereas the residualized signals are more concentrated towards a single unit direction, making $\ccE_{\cT_\emptyset^{(1)}}$ smaller compared to $\ccE_{\cG_\emptyset^{(1)}}$.

The above discussion shows that apart from the quantity $\tau_*(s)$, the complexity of residualized signals and the complexity of spurious projection operators also play a decisive role in the margin condition of the best subset selection problem. Specifically, the set with higher complexity characterizes the margin condition in Theorem \ref{thm: sufficiency of BSS}. 


We can further represent the condition \eqref{eq: margin cond} in terms of the diameter of the sets $\cT_\cI^{(s)}$ and $\cG_\cI^{(s)}$.
To see this, recall that $\ccE_{\cT_{\cI}^{(s)}} \leq \sfD_{\cT_{\cI}^{(s)}}$ and $\ccE_{\cG_{\cI}^{(s)}} \leq \sfD_{\cG_{\cI}^{(s)}}$ for all $\cI \subset \cS$. Under the light of this fact, we have the following corollary:
\begin{corollary}
\label{cor: sufficiency of BSS 2}
    Let the conditions in Assumption \ref{assumption: noisy features are not too correlated} and Assumption \ref{assumption: noise} hold. Then there exists a 
    positive universal constant $C_0$ such that for any $0\leq \eta <1$, whenever the identifiability 
    margin $\tau_*(s)$ satisfies 
    \begin{equation}
    \label{eq: margin cond 2}
    \begin{aligned}
&\frac{\tau_*(s)}{\sigma^2} \geq \\
& \frac{C_0}{(1- \eta)^2} \left[\max\left\{
\max_{\cI \subset \cS} \sfD_{\cT_\cI^{(s)}}^2, \max_{\cI \subset \cS} \sfD_{\cG_\cI^{(s)}}^2 
\right\}+ \sqrt{\frac{\log(es)\vee \log\log(ep)}{\log(ep)}}\right] \frac{\log(ep)}{n},
\end{aligned}
    \end{equation}
     we have 
 \[
\{\widehat{\cS}_{\rm best}(s)\} \subseteq\left\{
  \widehat{\cS}: \vert\widehat{\cS}\vert = s,  R_{\widehat{\cS}} \leq  \min_{\cD \in \ccA_s } R_{\cD}+ n \eta \tau_*(s) 
 \right\} = \{\cS\},
 \]
 with probability at least $1 - O(\{s \vee \log p\}^{-1})$.
 In particular, we have $\cS = \argmin_{\cD\in \ccA_s} R_{\cD}$ with high probability.
\end{corollary}

Corollary \ref{cor: sufficiency of BSS 2} essentially conveys the same message as Theorem \ref{thm: sufficiency of BSS}, only under a slightly stronger margin condition \eqref{eq: margin cond 2}. However, in some cases, it could be comparatively easier to give theoretical guarantees on the diameters $\sfD_{\cT_{\cI}^{(s)}}, \sfD_{\cG_{\cI}^{(s)}} $ rather than their corresponding complexity measures $\ccE_{\cT_{\cI}^{(
s)}}, \ccE_{\cG_{\cI}^{(s)}}$ respectively. In the next section, we will discuss a few illustrative examples to further elaborate on the effects of two complexities.

\subsection{Illustrative examples}
\label{sec: illustartive examples}
In this section, we will discuss a few illustrative examples to highlight the effect complexities of the two spaces described in Section \ref{sec: space of unexplained signals} and Section \ref{sec: space of projections}. 

\subsubsection*{Block design with a single active feature}
\label{sec: equi-corr design example}
Consider the model \eqref{eq: base model} where the rows of $\bX$ are independently generated from $p$-dimensional multivariate Gaussian distribution with mean-zero and variance-covariance matrix
\[
\bSigma = \begin{pmatrix}
1 & c \boldsymbol{1}_{p-1 }^\top\\
c\boldsymbol{1}_{p-1} & (1- r)\bbI_{p-1} + r \boldsymbol{1}_{p-1} \boldsymbol{1}_{p-1}^\top
\end{pmatrix},
\]
where $c \in [0, 0.997], r \in [0,1)$. We need to further impose a restriction 
$$c^2< r + \frac{1-r}{p-1}$$ 
to ensure positive definiteness of $\bSigma$.
In this case, we also set the true model $\cS = \{1\}$ and the noise variance $\sigma =1$. Recall that in this case the sets of residualized signals and spurious projection operators are denoted by $\cT_\emptyset^{(1)}$ and $\cG_\emptyset^{(1)}$ respectively. Under this setup, we have the following lemma:

\sapta{Finish the Lemmas}
\begin{lemma}
    \label{lemma: example independent}
    Assume that $\log p = o(n)$. Then under the above setup, there exist universal positive constants $C, L, M$ such that the followings are true with $\tol = C \{(\log p)/n\}^{1/2}$:
    \begin{enumerate}[label = (\alph*)]
        \item \label{lemma: margin equi-corr independent} For large enough $n,p$ we have 
       \begin{equation*}
    \begin{aligned}
    &\pr \left[ \left\{1 + \tol - \frac{(c - \tol)^2}{1 + \tol} \right\}\geq \frac{\tau_*(1)}{\beta_1^2} \geq  \left\{1 - \tol - \frac{(c+ \tol)^2}{1 - \tol} \right\} \right]\\
    & = 1 + o(1/p).
    \end{aligned}
\end{equation*}

        \item \label{lemma: cT complexity independent}
        For large enough $n,p$ we have
 \begin{align*}
&\pr \left[\max\left\{ \frac{2 c^2(1-r)}{1-c^2} - L \tol, 0\right\}\leq\sfd_{\cT_\emptyset^{(1)}}^2\leq\sfD_{\cT_\emptyset^{(1)}}^2 \leq  \frac{2 c^2(1-r)}{1-c^2} + L \tol \right]\\
& = 1+ o(1/p).
\end{align*}

\item \label{lemma: cG complexity independent}
        For large enough $n,p$ we have
        \begin{align*}
&\pr \left[ \max\left\{ (1-r^2) - M \tol, 0\right\}\leq\sfd_{\cG_\emptyset^{(1)}}^2\leq\sfD_{\cG_\emptyset^(1)}^2 \leq  (1-r^2) + M \tol \right] \\
& = 1+ o(1/p).
        \end{align*}
    \end{enumerate}
\end{lemma}




\sapta{Include the Lemmas}
From part \ref{lemma: cT complexity independent} and \ref{lemma: cG complexity independent} of the above lemma, it follows that
the complexity $\ccE_{\cT_\emptyset^{(1)}} \approx 0$ when $c =0$. For any fixed $c>0$ and $r \in [0,1)$, we have
\begin{equation}
\label{eq: two complexity trend example}
\ccE_{\cT_\emptyset^{(1)}}^2 \sim  \frac{2c^2(1-r)}{1-c^2},\quad \text{and} \quad \ccE_{\cG_\emptyset^{(1)}}^2 \sim (1-r^2) \quad \text{for large $n,p$}.
\end{equation}
A detailed derivation of the result is present in Section S1.5 of the supplementary material. Left panel of Figure \ref{fig: complexities} shows the partition of $c\mhyphen r$ plane based on the dominating complexity. It is worthwhile to note that a high value $r$, i.e., a high correlation among the noise variables results in a smaller value of the complexity terms in \eqref{eq: margin cond}. However, Lemma \ref{lemma: example independent}\ref{lemma: margin equi-corr independent} suggest that $\tau_*(1)/\beta_1^2 \sim (1-c^2)$, i.e.,  higher correlation between true and noise variables shrinks the margin quantity $\tau_*(1)$ towards 0. This suggests that a smaller value of $c$ and a higher value $r$ is more favorable to BSS than other possible choices of $(r,c)$. We now discuss these phenomena through some selected examples.
\begin{figure}[t!]
     \centering
     \begin{subfigure}[b]{0.45\textwidth}
         \centering         \includegraphics[width = \textwidth]{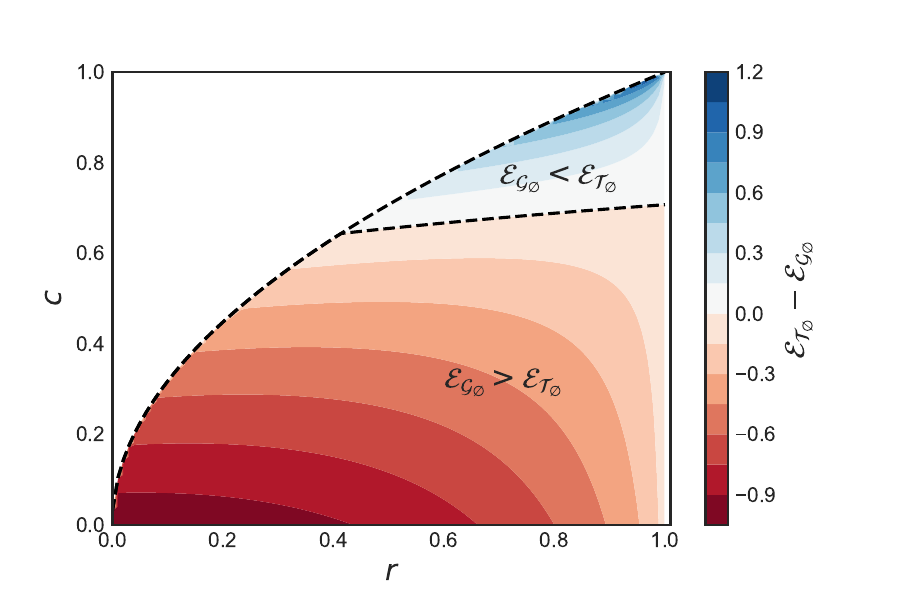}
     \end{subfigure}
     \hfill
     \begin{subfigure}[b]{0.45\textwidth}
         \centering         \includegraphics[width = \textwidth]{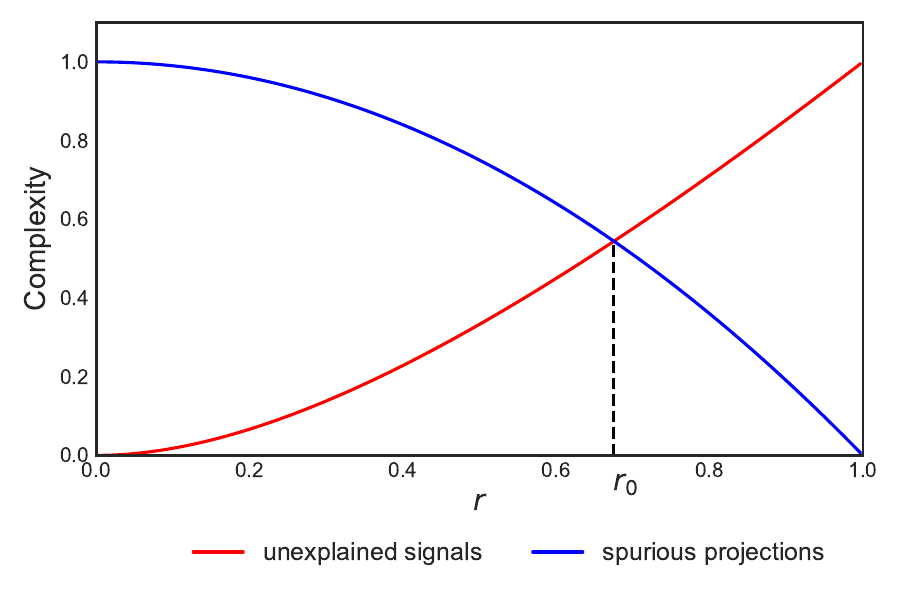}
     \end{subfigure}
        \caption{(left) Partition of $c\mhyphen r$ plane showing dominating regions for the two complexities. The color gradient indicates the value of $\ccE_{\cT_\emptyset^{(1)}} - \ccE_{\cG_\emptyset^{(1)}}$. (right) The plot of two complexities for varying $r$ under equicorrelated design.}
        \label{fig: complexities}
\end{figure}


\paragraph{Independent design: }
In this case $c = r= 0$. In this case \eqref{eq: two complexity trend example} suggest that $\ccE_{\cG_\emptyset^{(1)}}^2  \approx 1$. On the other hand, we see that $\ccE^2_{\cT_\emptyset^{(1)}} \approx 0$. Thus, the complexity of spurious projections is dominant in this case. Also, in this case, $\tau_*(1) \approx \beta_1^2$ which suggests that higher signal strength results in a better performance in terms of model selection.

\paragraph{Independent block design:} 
In this case, we set $c=0$ and we vary $r$ in $(0,1)$. Note that \eqref{eq: two complexity trend example} tells that $\ccE_{\cT_\emptyset^{(1)}}^2 
 \approx 0$, and  $\ccE_{\cG_\emptyset^{(1)}}^2$ has a decreasing trend with $r \in (0,1)$. This suggests that the independent block design with a high value of $r$ is more favorable for BSS to identify the true model compared to the independent random design model in the previous example. Finally, noting the fact that $\tau_*(1) \approx \beta_1^2$, we can conclude that for high values of $r$, the sufficient condition in Theorem \ref{thm: sufficiency of BSS} becomes less stringent. 

 \paragraph{Equicorrelated design:}
 Here we set $c = r $ and vary $r$ in the interval $[0,1)$. Let $r_0$ be the unique positive solution to the following equation:
 \[
 \frac{2 r^2}{1+r} - (1-r^2)=0.
 \]
 Calculation shows that $r_0 \approx 0.675$.
 Using \eqref{eq: two complexity trend example}, it follows that for $r \in [0,r_0)$ the complexity of spurious projection operators is dominating, i.e.,
 $
 \ccE_{\cG_{\emptyset}^{(1)}}^2 > \ccE_{\cT_\emptyset^{(1)}}^2.
 $
 In contrast, for $r\in( r_0,1)$, we have the complexity of the residualized signals to be dominating, i.e.,
 $
 \ccE_{\cT_{\emptyset}^{(1)}}^2 >  \ccE_{\cG_{\emptyset}^{(1)}}^2.
 $
 Right panel of Figure \ref{fig: complexities} indicates the phase transition between the two complexities. 
   Since the identifiability margin $\tau_*(1)$ roughly behaves like $ \beta_1^2 (1-r^2)$, the margin quantity becomes very small for a high value of $r$. Hence, for model consistency, we need a high value for $\beta_1^2$.

\begin{remark}
\label{remark: equi-cor}
In the example of equi-correlated design with $\cS = \{1\}$, the effect of correlation parameter $r$ on the complexities $\ccE_{\cT_\emptyset^{(1)}}$ and $\ccE_{\cG_\emptyset^{(1)}}$ are complementary to each other. In the case of the set of residualized signals, increasing correlation among the features increases the overall complexity of the set $\cT_\emptyset^{(1)}$ and vice versa. In contrast, higher correlation decreases the complexity $\ccE_{\cG_\emptyset^{(1)}}$, thus shrinking the effective size of $\cG_\emptyset^{(1)}$. Thus, in this case, the two complexities act as two opposing forces in the margin condition \eqref{eq: margin cond}. \sapta{update eq number}
\end{remark}

 
\subsection{Necessary condition}
\label{sec: necessary condition for BSS}
One question that arises from the preceding discussion is whether the margin condition in Theorem \ref{thm: sufficiency of BSS} is necessary for model consistency or not. Specifically, it is natural to ask whether the complexities of residualized signals and spurious projections also characterize the necessary margin condition.  In this section, we show that a condition very similar to \eqref{eq: margin cond} is necessary for model consistency of BSS, which is also governed by a similar margin quantity and complexity measures.  

For $j_0 \in \cS$, we define the set $\ccC_{j_0}:= \{ \cD: \cS\setminus \cD = \{j_0\}, \abs{\cD} = s\} \subset \ccA_{s,1}$. We consider the maximum \emph{leave-one-out} identifiability margin for $j_0\in \cS$ as 
\begin{equation}
  \widehat{\tau}(s) := \max_{j_0 \in \cS}\max_{\cD \in \ccC_{j_0}} \frac{\bbeta_{\cS\setminus \cD}^\top \Gamma(\cD) \bbeta_{\cS\setminus \cD}}{\abs{\cS\setminus \cD}} = \max_{j_0 \in \cS}\max_{\cD \in \ccC_{j_0}}\Gamma(\cD) \beta_{j_0}^2.
\label{eq: necessary margin cond}  
\end{equation}
Consider the set $\cI_0:= \cS\setminus \{j_0\}$ for a fixed index $j_0 \in \cS$. We capture the complexity of $\cT_{\cI_0}^{(s)}$ through the following quantity:
\begin{equation}
\label{eq: lower-complexity of cT}
\ccE^*_{\cT_{\cI_0}^{(s)}}:= \frac{\sup_{\delta>0} \frac{\delta}{2} \sqrt{\log \cM(\{\widehat{\bgamma}_{\cI_0 \cup \{j\}}\}_{j \in \cS^c}, \norm{\cdot}_2, \delta)}}{\sqrt{\log \vert \cT_{\cI_0}^{(s)}\vert}}.   
\end{equation}
The above display immediately shows that 
$
\ccE^*_{\cT_{\cI_0}} \geq \sfd_{\cT_{\cI_0}^{(s)}}/2.
$
Also, from the property of packing and covering number, it follows that
\[
\cM(\{\widehat{\bgamma}_{\cI_0 \cup \{j\}}\}_{j \in \cS^c}, \norm{\cdot}_2, \delta) \leq \cN(\{\widehat{\bgamma}_{\cI_0 \cup \{j\}}\}_{j \in \cS^c}, \norm{\cdot}_2, \delta/2).
\]
As $\cN(\{\widehat{\bgamma}_{\cI_0 \cup \{j\}}\}_{j \in \cS^c}, \norm{\cdot}_2, \delta/2)$ is a decreasing function over $\delta \in (0,\infty)$, we have the following inequality:
\[
\sup_{\delta>0} \frac{\delta}{2} \sqrt{\log \cN(\{\widehat{\bgamma}_{\cI_0 \cup \{j\}}\}_{j \in \cS^c}, \norm{\cdot}_2, \delta/2)} \leq \int_{0}^\infty \sqrt{\log \cN(\{\widehat{\bgamma}_{\cI_0 \cup \{j\}}\}_{j \in \cS^c}, \norm{\cdot}_2, \varepsilon)}\; d\varepsilon.
\]
The above inequality further shows that $\ccE^*_{\cT_{\cI_0}^{(s)}} \leq \ccE_{\cT_{\cI_0}^{(s)}} \leq \sfD_{\cT_{\cI_0}^{(s)}}$. Hence, similar to $\ccE_{\cT_{\cI_0}^{(s)}}$, the alternative complexity measure $\ccE^*_{\cT_{\cI_0}^{(s)}}$ also captures the average separation among the elements in $\cT_{\cI_0}^{(s)}$.

Next, we focus on the set $\cG_{\cI_0}^{(s)}$ which is the collection of all the spurious projection operators of the form $\bP_{\cD} - \bP_{\cI_0}$ for all $\cD\in \ccC_{j_0}$. 
If $\cD = \cI_0 \cup \{j\}$ for some $j \in \cS^c$, then the corresponding spurious projection operator takes the form
\begin{equation}
\label{eq: diff of proj}
\bP_\cD - \bP_{\cI_0} = \widehat{\bu}_j \widehat{\bu}_j^\top,
\end{equation}
where $\widehat{\bu}_j$ denotes the unit vector along the residualized feature vector $
\bu_j := (\bbI_n - \bP_{\cI_0}) \bX_j. 
$
Thus, the above display basically shows that the $\bP_\cD- \bP_{\cI_0}$ is the orthogonal projection operator onto the linear span generated by the residualized feature $\bu_j$. 
 Similar to \eqref{eq: lower-complexity of cT}, we define the complexity measure of $\cG_{\cI_0}^{(s)}$ as
 \begin{equation}
     \label{eq: lower-complexity of cG}
     \ccE^*_{\cG_{\cI_0}^{(s)}}:= \frac{\sup_{\delta>0} \frac{\delta}{2} \sqrt{\log \cM(\cG_{\cI_0}^{(s)}, \norm{\cdot}_\op, \delta)}}{\sqrt{\log \vert \cG_{\cI_0}^{(s)} \vert }}.
 \end{equation}
By a similar argument, it also follows that 
$
\sfd_{\cG_{\cI_0}^{(s)}}/2 \leq \ccE^*_{\cG_{\cI_0}^{(s)}} \leq \sfD_{\cG_{\cI_0}^{(s)}}.
$
Hence, combining the above observation with \eqref{eq: diff of proj}, it also follows that $\ccE^*_{\cG_{\cI_0}^{(s)}}$ captures the  angular separation among the elusive features $\{\bu_j\}_{j \in \cS^c}$.


Next, we introduce some technical assumptions that are crucial for our theoretical analysis of the necessity result.

\begin{assumption}
\label{assumption: anchor point}

The complexities of the $\cG_{\cI_0}^{(s)}$  and $\cT_{\cI_0}^{(s)}$ are not too small, i.e.,
   \[ \ccEstar_{\cG_{\cI_0}^{(s)}}^2 > 16 \{\log(ep)\}^{-1} , \quad \text{and} \quad  \ccEstar_{\cT_{\cI_0}^{(s)}}^2 > 16 \{\log(ep)\}^{-1}\]
   for all $\cI_0 \subset \cS$ and $\abs{\cI_0 } = s-1$.
\end{assumption}
\sapta{explain the assumptions}
 Assumption \ref{assumption: anchor point} combined with the observation \eqref{eq: diff of proj} essentially tells that the set of elusive features $\{\widehat{\bu}_j\}_{j \in \cS^c}$ and the scaled spurious signals $\{\widehat{\bgamma}_{\cD}\}_{\cD\in \ccC_{j_0}}$  are not too identical with each other, as $\ccE^*_{\cG_{\cI_0}^{(s)}}$ and $\ccE^*_{\cT_{\cI_0}^{(s)}}$ would be typically small otherwise. Thus, Assumption \ref{assumption: anchor point} induces diversity in $\cT_{\cI_0}^{(s)}$ and $\cG_{\cI_0}^{(s)}$. 

\begin{condition}
\label{cond: alpha-regularity}
There exists a constant $\alpha \in (0,1)$ such that $\ccEstar_{\cT_{\cI_0}^{(s)}}/ \ccE_{\cT_{\cI_0}^{(s)}} \in (\alpha, 1)$. 
\end{condition}
The condition
essentially tells that the set $\cT_{\cI_0}^{(s)}$ has a somewhat regular geometric shape in the sense that both the lower and upper complexity are of the same order. This essentially implies that minimal separation and maximal separation of the set $\cT_{\cI_0}^{(s)}$ are of the same order. \sapta{This is true for most of the random design models}.

Now we present our theorem on the necessary condition for model consistency of BSS.



\begin{theorem}[Necessity]\sapta{correct the statement and notations}
    \label{thm: necessary condition BSS}
    Assume $\bvarepsilon\sim \sfN(0, \sigma^2 \bbI_n)$, $p > 16 e^3$ and $s<p/2$. Also, let the Assumption \ref{assumption: anchor point} hold and write $\cJ = \{\cI \subset \cS : \abs{\cI} = s-1\}$.
Then the followings are true:
  \begin{enumerate}[label=(\alph*)]
    \item If $\ccEstar_{\cG
_{\cI_0}^{(s)}} \notin  (\ccEstar_{\cT_{\cI_0}^{(s)}}, \ccE_{\cT_{\cI_0}^{(s)}})$ for all $\cI_0 \in \cJ$, then there exists a universal constant $C_1>0$ such that
    
    \[
\widehat{\tau}(s) \leq C_1 \max\left\{\max_{\cI_0 \in \cJ}\ccEstar_{\cT_{\cI_0}^{(s)}}^2, \max_{\cI_0 \in \cJ}\ccEstar_{\cG_{\cI_0}^{(s)}}^2 \right\} \frac{\sigma^2 \log(ep)}{n}
\]
implies that 
 \[
\pr(\widehat{\cS}_{\rm best}(s) \neq \cS) \geq \frac{1}{10}.
    \]
    

\item If there exists $\cI_\# \in \cJ$ such that $\ccEstar_{\cG
_{\cI_\#}^{(s)}} \in  (\ccEstar_{\cT_{\cI_\#}^{(s)
}}, \ccE_{\cT_{\cI_\#}^{(s)}})$, then under Condition \ref{cond: alpha-regularity}, there exists a constant $C_\alpha$ depending on $\alpha$, such that 

\[
\widehat{\tau}(s) \leq C_\alpha \max\left\{\max_{\cI_0 \in \cJ}\ccEstar_{\cT_{\cI_0}^{(s)}}^2, \max_{\cI_0 \in \cJ}\ccEstar_{\cG_{\cI_0}^{(s)}}^2 \right\} \frac{\sigma^2 \log(ep)}{n}
\]
implies that
 \[
\pr(\widehat{\cS}_{\rm best}(s) \neq \cS) \geq \frac{1}{10}.
    \]
    \end{enumerate}   

\end{theorem}
The detailed proof can be found in Section S1.4 of the supplementary material.
The above theorem essentially says that if the maximum leave-one-out margin $\widehat{\tau}(s) \lesssim \sigma^2 (\log p)/n$ then the BSS fails to achieve model consistency with positive probability. However, the interesting part of the above theorem is to understand the effect of the term involving complexity measures. Similar to Theorem \ref{thm: sufficiency of BSS}, here also we see that the dominating complexity characterizes the necessary condition for model consistency.
However, we reiterate a few major differences between the above theorem and Theorem \ref{thm: sufficiency of BSS}. First, Theorem \ref{thm: sufficiency of BSS} needs $\tau_*(s)$ to be lower bounded, which is much stronger than the required condition on $\widehat{\tau}(s)$ in Theorem \ref{thm: necessary condition BSS}. Second, Theorem \ref{thm: necessary condition BSS} involves the alternative complexity measures $\ccEstar_{\cT_{\cI_0}^{(s)}}$ and $\ccEstar_{\cG_{\cI_0}^{(s)}}$, which are typically smaller than the complexity measures used in Theorem \ref{thm: sufficiency of BSS}. Third, the resulting complexity in Theorem \ref{thm: sufficiency of BSS} involves the maximum over all possible subsets of $\cS$, whereas Theorem \ref{thm: necessary condition BSS} involves the maximum only over the subsets of $\cS$ of size $s-1$. These three facts are the main reasons that the requirement in Theorem \ref{thm: necessary condition BSS} is weaker compared to the margin condition \eqref{eq: margin cond}. Nonetheless, Theorem \ref{thm: necessary condition BSS} is still interesting as it shows that the two types of complexities are indeed important quantities to understand the model selection performance of BSS.
Theorem \ref{thm: necessary condition BSS} can also be stated in terms of the diameter and minimum separability of the sets $\cT_{\cI_0}^{(s)}$ and $\cG_{\cI_0}^{(s)}$. Recall that $\ccEstar_{\cT_{\cI_0}^{(s)}} \gtrsim \sfd_{\cT_{\cI_0}^{(s)}}$ and $\ccEstar_{\cG_{\cI_0}^{(s)}} \gtrsim \sfd_{\cG_{\cI_0}^{(s)}}$. Hence, it follows that under the same conditions in Theorem \ref{thm: necessary condition BSS}, the margin condition
\[
\widehat{\tau}(s) \gtrsim \max \left\{\max_{\cI_0 \in \cJ}\sfd_{\cT_{\cI_0}^{(s)}}^2, \max_{\cI_0 \in \cJ}\sfd_{\cG_{\cI_0}^{(s)}}^2 \right\} \frac{\sigma^2 \log(ep)}{n}
\]
is necessary for model consistency of BSS.

\section{Experiments}
In this section, we will compare the performance of BSS with that of LASSO, one of the most popular tools for model selection. In these experiments, we set $p = 2000, s = 10$ and $n = 500$. We construct the design matrix $\bX$ by sampling each row $\bx_i \sim N(0, \bSigma)$ for $i \in [n]$, where
\[
\bSigma = \begin{bmatrix}
    (1-r_t) \bbI_s + r_t \bone_s \bone_s^\top & \boldsymbol{0}_{s \times (p-s)}\\
    \boldsymbol{0}_{(p-s)\times s} &  (1-r_s) \bbI_{p-s} + r_s \bone_{p-s} \bone_{p-s}^\top
\end{bmatrix}
\]
and $r_t, r_s \in [0,1]$. We set $\bbeta \in \bbR^p$ so that $\beta_j = \ind\{j \le s\}$ for all $j \in [p]$. The responses $\{y_i\}_{i \in [n]}$ are generated according to model \eqref{eq: base model} with $\varepsilon_i \sim \sfN(0,1)$. For the experiments, we vary $r_t\in \{0.0, 0.5, 0.9\}$ and $r_s \in \{0.0, 0.1, \ldots, 0.9\}$. 

We use ABESS \citep{zhu2022abess} as a computational surrogate for BSS, and we also provide the knowledge of $s$ to the algorithm. For LASSO, we choose the penalty parameter by five-fold cross-validation and then choose the $s$ coordinates with the highest absolute values of the estimated LASSO coefficient, i.e., we perform hard thresholding (HT) operation on the LASSO estimator. In Figure \ref{fig: bss vs lasso}, we plot the exact recovery rates of BSS and LASSO + HT for varying choices of $r_t$ and $r_s$.

\begin{figure}
     \centering
     \begin{subfigure}[b]{0.32\textwidth}
         \centering         \includegraphics[width = \textwidth]{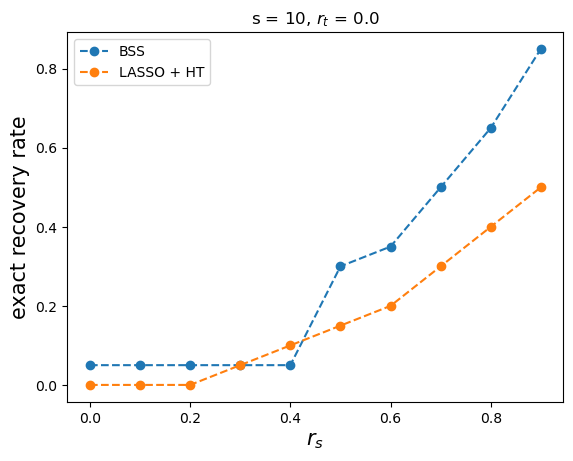}
         \caption{$r_t = 0.0$}
     \end{subfigure}
     \hfill
     \begin{subfigure}[b]{0.32\textwidth}
         \centering         \includegraphics[width = \textwidth]{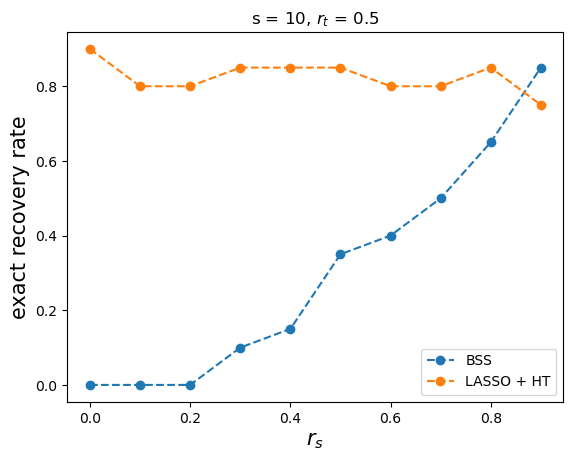}
         \caption{$r_t = 0.5$}
     \end{subfigure}
     \hfill
     \begin{subfigure}[b]{0.32\textwidth}
         \centering         \includegraphics[width = \textwidth]{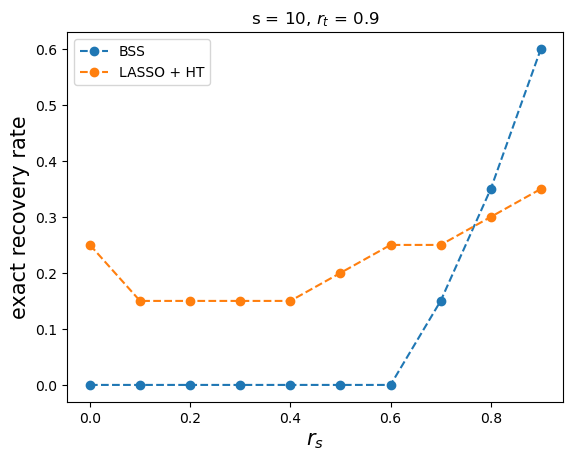}
         \caption{$r_t = 0.9$}
     \end{subfigure}
        \caption{Exact recovery rate of BSS and LASSO + HT for varying choices of $r_t$ and $r_s$.}
        \label{fig: bss vs lasso}
\end{figure}

In all the cases in Figure \ref{fig: bss vs lasso}, we see that the performance of BSS improves as $r_s$ increases to 1. This is in fact consistent with the theoretical results in Theorem \ref{thm: sufficiency of BSS} as the overall complexity of the spurious signals is likely smaller for high values of $r_s$, and thus the margin condition is easier to satisfy. In this case, complexity of residualized signals are somewhat unaffected as the cross-correlation between true and spurious signals is 0. However, the performance of LASSO+ HT does not seem to exhibit any particular behavior across different correlation structure. For $r_t = 0$, performance of LASSO + HT, although comparatively worse than BSS, is similar in terms of the trend. However, for $r_t = 0.5$, LASSO + HT is consistently better and somewhat stable. For $r_t = 0.9$, LASSO + HT is generally better than BSS and exact recovery rate increases for $r_s\ge 0.4$. However, its performance is much worse compared to the $r_t = 0.5$ case. These observations indicate that similar complexity theory for LASSO is in fact potentially challenging and evidently requires more future research. 
\section{Conclusion}
In this paper, we establish the sufficient and (nearly) necessary conditions for BSS to achieve model consistency in a high-dimensional linear regression setup. Apart from the identifiability margin, we show that the geometric complexity of the residualized signals and spurious projections based on the entropy number and packing numbers also play a crucial role in characterizing the margin condition for model consistency of BSS. In particular, we establish that the dominating complexity among the two plays a decisive role in the margin condition. We also highlight the variation in these complexity measures under different correlation strengths between the features through some simple illustrative examples. Moreover, in the supplementary material, we extend the results in Theorem \ref{thm: sufficiency of BSS} to the high-dimensional sparse generalized linear models (Section S2). However, it is an open problem to find the analogs of the two complexities in more general settings, e.g., the low-rank matrix regression problem or multi-tasking regression problem.


\paragraph{Supplement:}The supplementary material
contains the extension of Theorem \ref{thm: sufficiency of BSS} to the generalized linear model and the proofs of the main results.

\bibliographystyle{apalike} 
\bibliography{bj-ref}       

\setcounter{section}{0}
\setcounter{equation}{0}
\setcounter{theorem}{0}
\setcounter{lemma}{0}
\setcounter{assumption}{0}
\def\theequation{S\arabic{section}.\arabic{equation}}
\def\thesection{S\arabic{section}}\def\thetheorem{S\arabic{section}.\arabic{theorem}}
\def\thelemma{S\arabic{section}.\arabic{lemma}}
\def\theassumption{S\arabic{section}.\arabic{assumption}}


\section*{Supplementary sections}
\section{Proof of main results under linear model}

\subsection{Proof of Lemma 1}
First note that $\bbeta^\top_{\cS \setminus \cD} \Gamma(\cD) \bbeta_{\cS \setminus \cD} = 0 \Leftrightarrow (\bbI_n - \bP_\cD) \bX_{\cS \setminus \cD} \bbeta_{\cS \setminus \cD} = 0$. This shows that $\bX_{\cS \setminus \cD} \bbeta_{\cS \setminus \cD} \in \col(\bX_\cD)$. Thus, we have $\bX_\cS \bbeta_\cS = \bX_{\cS \setminus \cD} \bbeta_{\cS \setminus \cD} + \bX_{\cS \cap \cD} \bbeta_{\cS \cap \cD} \in \col(\bX_\cD)$. This finishes the proof.

\subsection{Proof of Proposition 1}
\label{sec: src is stronger assumption}
 In this section, we will show that the SRC condition (5) is strictly stronger than the condition in Assumption 1. Recall that the features are normalized, i.e., 
$\norm{\bX_j}_2 = \sqrt{n}$ for all $j \in [p]$. Now, we will prove the proposition.


    \begin{proof}
        \subsubsection*{SRC implies Assumption 1:}
        For a set $\cI \subset \cS$,
        define  $\ccA_\cI:= \{\cD \in \ccA_s: \cS \cap \cD = \cI\}$. Now recall that $\ccE_{\cG_\cI^{(s)}} \gtrsim \sfd_{\cG_\cI^{(s)}}$ for large $p$. Thus, it suffices to show that $\sfd_{\cG_\cI^{(s)}}$ is large for all choices of $\cI\subset \cS$.  Let $\cD_1, \cD_2\in \ccA_\cI$      and write $\cM = \cD_1 \cap \cD_2$. Let $m = \abs{\cM}$ and consider the two subspaces $L_1 = \col(\bX_{\cD_1}) \cap \col(\bX_{\cM})^\perp$ and $L_2 = \col(\bX_{\cD_2}) \cap \col(\bX_{\cM})^\perp$. 
        Let $\{\bxi_j\}_{j=1}^m$ be an orthonormal basis of $\cM$. 
        Let $\{\balpha_j\}_{j=1}^{s-m}$ be an orthonormal basis of $L_1$ and  $\{\bdelta_j\}_{j=1}^{s-m}$ be the orthonormal basis of $L_2$ such that 
        \[
        \theta_j := \angle(\balpha_j, \bdelta_j), \quad j \in [k],
        \]
        are the principal angles between $L_1$ and $L_2$ in decreasing order. Now, we construct the matrix $\bZ$ in the following way:
        \[
        \bZ = \left[\bX_{\cD_1\setminus \cD_2}\mid \bX_{\cM}\mid \bX_{\cD_2 \setminus \cD_1} \right].
        \]
        There exists matrix $\bu, \bv \in \bbR^{s-m}$ and $\bw_1, \bw_2 \in \bbR^{m}$ such that 
        \[\balpha_1= \bX_{\cD_1\setminus \cD_2}\bu + \bX_\cM \bw_1 \quad \text{and} \quad \bdelta_1= \bX_{\cD_2\setminus \cD_1}\bv + \bX_\cM \bw_2.
        \]
        As $\balpha_1 \perp \col(\bX_\cM)$, we have 
        \[
        1 = \balpha_1^\top \balpha_1 = \balpha_1^\top \bX_{\cD_1 \setminus \cD_2}\bu \leq \sqrt{n\kappa_+} \norm{\bu}_2 \Rightarrow \norm{\bu}_2^2 \geq 1/(n\kappa_+).
        \]
        By a similar argument, we have $\norm{\bv}_2^2 \geq 1/(n\kappa_+)$. Define the vectors $\bfeta := (\bu^\top, (\bw_1 - \bw_2)^\top, \bv^\top)^\top$. Hence, $\norm{\bfeta}_2^2 \geq \norm{\bu}_2^2 + \norm{\bv}_2^2 \geq 2/(n\kappa_+)$. Due to SRC condition (4.7), we have
        \begin{equation}
        \label{eq: norm lower bound}
        \norm{\bZ \bfeta}_2^2 \geq  n \norm{\bfeta}_2^2 \kappa_- \geq 2(\kappa_-/\kappa_+).
        \end{equation}

        \begin{equation}
        \label{eq: norm upper bound}
        \begin{aligned}
            \norm{\bZ \bfeta}_2^2 & = \norm{\balpha_1 - \bdelta_1}_2^2 \\
            & = 
            2(1 - \sqrt{1 - \sin^2\theta_1})\\
            & \leq 2 \sin 
            \theta_1 ,
        \end{aligned}
        \end{equation}
where the last inequality follows from the fact that $1-x \leq \sqrt{1 - x^2}$ for all $x \in [0,1]$.
Combining \eqref{eq: norm lower bound} and \eqref{eq: norm upper bound}, we have 
\[
\norm{\bP_{\cD_1} - \bP_{\cD_2}}_\op \geq \frac{\kappa_-}{\kappa_+}.
\]
 The above display shows that $\sfd_{\cG_\cI^{(s)}} \gtrsim (\kappa_-/\kappa_+)\gg \{\log(ep)\}^{-1/2}$ for all $\cI \subset \cS$.  Hence, the claim follows.      

        \subsubsection*{Assumption 1 does not imply SRC:}
        In this case, assume $\cS= \{1\}$ and $\be_j$ be the $j$th canonical basis in $\bbR^p$. Under this setup, Assumption 1 becomes 
        \begin{equation}
        \label{eq: condition}
        \ccE_{\cG_\emptyset^{(1)}} > \{\log(ep)\}^{-1/2}.
        \end{equation}
        Now assume that $$   \frac{2}{\log(ep)} \leq \frac{\norm{\bX_j - \bX_
{j^\prime}}_2^2}{n} \leq \frac{3}{\log(ep)}, \quad \text{for all $j, j^\prime \in [p]$}. $$
Then, for large $p$, the condition in \eqref{eq: condition} holds but SRC fails with the choice of $\bv = (\be_j - \be_{j^\prime})/\sqrt{2}$.
\end{proof}

\subsection{Proof of Theorem 1}
\label{sec: sufficiency of BSS}
Recall that $\bmu = \bX_\cS \bbeta_\cS$ , $\bgamma_\cD = n^{-1/2}(\bbI_n - \bP_\cD) \bmu$ and 
\begin{equation}
\label{eq: Schur comp}
    \Gamma(\cD) =  \widehat{\bSigma}_{ \cS \setminus \cD, \cS \setminus \cD} -  \widehat{\bSigma}_{ \cS \setminus \cD, \cD}\widehat{\bSigma}_{\cD, \cD}^{-1} \widehat{\bSigma}_{\cD, \cS \setminus \cD}.
\end{equation}
Note that for $\cD \in \ccA_{s,k}$ and $0 \leq \eta <1$ we have the following:
\begin{equation}
\label{eq: RSS diff}
    \begin{aligned}
    & n^{-1} (R_\cD - R_\cS) = n^{-1}\{ \by^\top (\bbI_n - \bP_\cD ) \by - \by^\top (\bbI_n - \bP_\cS) \by\}\\
    &= n^{-1} \left\{ (\bX_{\cS \setminus \cD} \bbeta_{\cS \setminus \cD} + \bvarepsilon)^\top (\bbI_n - \bP_{\cD}) (\bX_{\cS \setminus \cD} \bbeta_{\cS \setminus \cD} + \bvarepsilon) -  \bvarepsilon^\top (\bbI_n - \bP_\cS) \bvarepsilon \right\}\\
    & = \eta \bbeta_{\cS\setminus \cD}^\top \Gamma(\cD) \bbeta_{\cS\setminus \cD} +  2^{-1}(1- \eta)  \bbeta_{\cS\setminus \cD}^\top \Gamma(\cD) \bbeta_{\cS\setminus \cD} - 2 \left\{n^{-1} (\bbI_n - \bP_{\cD})\bX_{\cS\setminus \cD}\bbeta_{\cS\setminus \cD}\right\}^\top (-\bvarepsilon)\\
    &\quad  + 2^{-1}(1-\eta) \bbeta_{\cS\setminus \cD}^\top \Gamma(\cD) \bbeta_{\cS\setminus \cD} - n^{-1} \bvarepsilon^\top(\bP_\cD - \bP_\cS) \bvarepsilon.
    \end{aligned}
\end{equation}
Also, let $\widetilde{\bvarepsilon}:= (-\bvarepsilon)$. Now, $\cE$ be an event under which the following happens:

\[
\left\{
 \widehat{\cS}: \vert\widehat{\cS}\vert = s, \min_{\cS \in \ccA_s } R_{\widehat{\cS}} \leq R_\cS + n \eta \tau_*(s) 
 \right\} = \{\cS\}.
\]
Define the set $\ccA_{\cI}: = \{\cD \in \ccA_s: \cS \cap \cD = \cI\}$. We also set $\abs{\cI} = s - k$ for $k\in [s]$. Then we have $\ccA_{\cI} \subset \ccA_{s,k}$.
By union bound we have the following:

\begin{equation}
\label{eq: error prob BSS}
 \pr (\cE^c) \leq \sum_{k =1}^s \sum_{\cI \subset \cS : \abs{\cI} = s-k} \pr \left\{\min_{\cD \in \ccA_{\cI}} n^{-1} (R_\cD - R_\cS) < \eta \tau_*(s) \right\}. 
\end{equation}

Thus,  under the light of equation \eqref{eq: RSS diff} it is sufficient to show the following with high probability:

\begin{equation}
\label{eq: quant 1}
\max_{\cD \in \ccA_{\cI}} \widehat{\bgamma}_\cD^\top \widetilde{\bvarepsilon} < \frac{n^{1/2}(1-\eta)}{4} \min_{\cD\in \ccA_{\cI}} \norm{\bgamma_\cD}_2,
\end{equation}

\begin{equation}
    \label{eq: quant 2}
    \max_{\cD\in \ccA_{\cI}} n^{-1} \left\{ \bvarepsilon^\top (\bP_\cD - \bP_\cS) \bvarepsilon \right\} < \frac{1- \eta }{2} \min_{\cD\in \ccA_{\cI}}\norm{\bgamma_\cD}_2^2,
\end{equation}
for every $k \in [s]$. We will analyze the above events separately. We recall the two important sets below:
\[
\cT_{\cI}^{(s)}:= \{ \widehat{\bgamma}_\cD  : \cD \in \ccA_{s} , \cD \cap \cS = \cI \}, \; \text{and} \;\; \cG_{\cI}^{(s)} := \{ \bP_\cD - \bP_{\cI}  : \cD \in \ccA_{s}, \cD \cap \cS = \cI \}.
\]

To reduce notational cluttering, we will drop the $s$ in the superscript, and use $\cT_\cI$ and $\cG_\cI$ to denote the above sets.



\paragraph{Linear term:} 
Let $f_\cD : = \widehat{\bgamma}_{\cD}^\top \widetilde{\bvarepsilon} $ and $\norm{f} := \max_{\cD\in \ccA_{\cI}} f_\cD$. Since $\sfD_{\cT_\cI} \leq \sqrt{2}$,
 Theorem 5.36 of \cite{wainwright2019high}tells that there exists a constant $A_1>0$ such that  
\begin{equation}
\label{eq: borell-TIS.com}
\pr \left\{ \norm{f} \geq A_1\sigma (\ccE_{\cT_\cI}\sqrt{k \log(ep)}+ u) \right\} \leq 3\exp \left(- \frac{u^2}{2 }\right),
\end{equation}
for all $u>0$. 
Setting $u = 2c_\cT\sqrt{ k  \log(ep)}$ in Equation \eqref{eq: borell-TIS.com} we get 
\begin{equation}
\label{eq: linear deviation.com}
 \pr (\norm{f} \geq A_1 \sigma \ccE_{\cT_{\cI}} \sqrt{k\log(ep)}+  2 A_1 c_\cT \sigma \sqrt{ k \log(ep)}) \leq 3(ep)^{-2c_\cT^2 k}.
\end{equation}

Writing $A_1$ as $c_1$, we get

\begin{equation}
\label{eq: linear deviation 2.com}
    \pr \left\{\max_{\cD \in \ccA_{\cI}} \widehat{\bgamma}_{\cD}^\top \widetilde{\bvarepsilon} \geq c_1 (\ccE_{\cT_{\cI}} +
    2 c_\cT) \sigma \sqrt{k \log(ep) }  \right\} \leq  3(ep)^{-2 c_\cT^2 k}.
\end{equation}

\paragraph{Quadratic term:}
Here we study the quadratic supremum process in Equation \eqref{eq: quant 2}. First, define the two projection operators $\bP_{\cD \mid \cI} = \bP_\cD - \bP_{\cI}$ and $\bP_{\cS \mid \cI}: = \bP_\cS - \bP_{\cI}$. For any number $c_\cG\in(\{\log(ep)\}^{-1},1)$, by union bound we have,
\begin{equation}
    \label{eq: quad decomposition}
    \begin{aligned}
    &\pr \left\{ n^{-1}\max_{\cD \in \ccA_{\cI}} \bvarepsilon^\top (\bP_\cD - \bP_\cS)\bvarepsilon > \sigma^2 u + \sigma^2 c_\cG u_0\right\}\\
    & \pr \left\{ n^{-1}\max_{\cD \in \ccA_{\cI}} \bvarepsilon^\top (\bP_{\cD \mid \cI} - \bP_{\cS\mid \cI})\bvarepsilon > \sigma^2 u + \sigma^2 c_\cG u_0\right\}\\
    & \leq \pr \left\{ n^{-1} (k\sigma^2 - \bvarepsilon^\top  \bP_{\cS\mid \cI} \bvarepsilon) > \sigma^2 u_0 c_\cG \right\} +  \pr \left\{ n^{-1} \max_{\cD \in \ccA_{\cI}}(\bvarepsilon^\top \bP_{\cD\mid \cI} \bvarepsilon - k\sigma^2) > \sigma^2 u \right\}.
    \end{aligned}
\end{equation}
Also, note that $\bbE(\bvarepsilon^\top \bP_{\cD \mid \cI}\bvarepsilon) \leq k\sigma^2$ and recall that $\ccE_{\cG_{\cI}} >  \{\log(ep)\}^{-1/2}$ for all $\cI \subset \cS$. This shows that $\sqrt{k}\leq \ccE_{\cG_\cI} \sqrt{k \log(ep)}$. Furthermore, 
 by the properties of projection matrices, we have $d_\op (\cG_\cI) =1 $ and $d_F(\cG_\cI) = \sqrt{k}$ (defined in Section \ref{appendix: quadratic chaos}). Also, it follows that the quantities $M, V$ and $U$ (defined in Theorem \ref{thm: order-2 chaos deviation bound}) have the following properties:
\[
M \leq 2 \ccE_{\cG_\cI}^2 k \log(ep), \quad V \leq 2 \sqrt{k \log (ep)}, \quad \text{and} \quad U =1.
\]
Using these facts and Theorem \ref{thm: order-2 chaos deviation bound}, we get that there exists a universal positive constants $A_2, A_3$, such that  for $t = A_3 c_\cG k \log(ep)$, we get
\begin{equation}
\label{eq: quad-deviation.com}
\pr\left\{\max_{\cD \in \ccA_\cI} \bvarepsilon^\top \bP_{\cD \mid \cI}\bvarepsilon \geq  A_2 \sigma^2( \ccE_{\cG_\cI}^2 + c_\cG)  k \log(ep)  \right\} \leq (ep)^{-2 c_\cG^2 k}.
\end{equation}
Due to Theorem 1.1 of \cite{rudelson2013hanson}, setting $u_0 =  k\log (ep)/(2n)$ we can show that there exists an absolute constant $A_4>0$ such that
\begin{equation}
    \label{eq: quad deviation 4.com}
    \begin{aligned}
     \pr \left\{ n^{-1} \abs{\bvarepsilon^\top \bP_{\cS\mid \cI} \bvarepsilon - k\sigma^2} > \frac{ c_\cG \sigma^2 k \log(ep)}{2n}\right\}
    & \leq 2 \exp \left\{-  A_4 c_\cG k \log(ep) \right\}\\
    & = 2 (ep)^{- A_4 c_\cG k},
    \end{aligned}
\end{equation}
 Combining Equation \eqref{eq: quad decomposition}, \eqref{eq: quad-deviation.com}  and Equation \eqref{eq: quad deviation 4.com} yields

\begin{equation}
    \label{eq: quad deviation 5}
    \pr \left\{ n^{-1}\max_{\cD \in \ccA_{\cI}} \bvarepsilon^\top (\bP_\cD - \bP_{\cS})\bvarepsilon > c_2 ( \ccE_{\cG_{\cI}}^2 + c_\cG) \sigma^2\frac{ k \log(ep)}{n}\right\} \leq (ep)^{-2c_\cG^2 k} + 2 (ep)^{- A_4 c_\cG k},
\end{equation}
where $c_2$ is a universal constant.
Now, if we have 
\begin{equation}
\begin{aligned}
 \tau_*(s) & \triangleq \min_{\cD\neq \cS}\frac{\bbeta_{\cS\setminus \cD}^\top \Gamma(\cD) \bbeta_{\cS\setminus \cD}}{\abs{\cS \setminus \cD}}\\
 & \geq \frac{64}{(1-\eta)^2} \max\left\{  c_1\max_{ \cI \subset \cS} (\ccE_{\cT_{\cI}} + 2 c_{\cT})^2, c_2\max_{ \cI \subset \cS} ( \ccE_{\cG_{\cI}}^2 +c_\cG) \right\} \frac{\sigma^2 \log(ep)}{n},
 \end{aligned}
 \label{eq: sufficiency condition for BSS}
\end{equation}
it will ensure that \eqref{eq: quant 1} and \eqref{eq: quant 2} hold with high probability. Finally, using \eqref{eq: linear deviation 2.com} and \eqref{eq: quad deviation 5}, we have
\begin{align*}
&\pr (\cE^c)\\
& \leq \sum_{k =1}^s \sum_{\cI \subset \cS : \abs{\cI} = s-k} \pr \left\{\min_{\cD \in \ccA_{\cI}} n^{-1} (R_\cD - R_\cS) < \eta \tau_*(s)\right\}\\
& \lesssim \sum_{k = 1}^s \sum_{\cI \subset \cS: \abs{\cI} = s-k} \left\{(ep)^{-2 c_\cT^2 k} + (ep)^{-2c_\cG^2 k} + (ep)^{- A_4 c_\cG k} \right\}\\
& \lesssim \sum_{k = 1}^s \binom{s}{k} \left\{(ep)^{-2 c_\cT^2 k} + (ep)^{-2c_\cG^2 k} + (ep)^{- A_4 c_\cG k}\right\}\\
& \lesssim  \sum_{k=1}^s (es)^k\left\{(ep)^{-2 c_\cT^2 k} + (ep)^{-2c_\cG^2 k} + (ep)^{- A_4 c_\cG k}\right\}\\
& \lesssim \sum_{k=1}^s \exp \left[-k \{2 c_\cT^2 \log(ep) - \log(es)\} \right] + \exp [- k \{2c_\cG^2 \log(ep) - \log(es)\}]\\
& \quad \quad  + \exp\left[-k \{A_4 c_\cG \log(ep) - \log(es)\}\right].
\end{align*}
Now, setting $c_\cT =\sqrt{\{\log(es) \vee \log\log(ep)\}/ \log(ep)}$ and $c_\cG = (2 \vee A_4^{-1}) c_\cT$ in the above display, and using the identity $(a +b)^2\leq 2(a^2 + b^2)$ we can conclude that the following is sufficient to hold \eqref{eq: sufficiency condition for BSS}:
\[
\tau_*(s)\geq 
\frac{64}{(1-\eta)^2} \max\{8c_1, c_2(2\vee A_4^{-1})\} \left[\max\left\{
\max_{\cI \subset \cS} \ccE_{\cT_\cI}^2, \max_{\cI \subset \cS} \ccE_{\cG_\cI}^2 
\right\}+ c_\cT\right] \frac{\log(ep)}{n}.
\]
and renaming the absolute constant $64  \max\{8c_1, c_2(2\vee A_4^{-1})\}$ as $C_0$.

\subsection{Proof of Theorem 2}
\label{sec: ncessary condition BSS}

\begin{proof}
First, we will show that $\norm{\bgamma_\cD}_2$ has to be well bounded away from 0 for every $\cD \in \cup_{j_0 \in \cS}\ccC_{j_0}$. Again, to reduce notational cluttering, we drop the $s$ in the superscript and use $\cT_{\cI_0}$ and $\cG_{\cI_0}$ to denote the sets of scaled residualized signals and spurious projections respectively. 

\subsubsection*{\underline{Ruling out the case $\min_{\cD \in \cup_{j_0 \in \cS}\ccC_{j_0}}\norm{\bgamma_\cD}_2 \leq  \sigma/\sqrt{n} $ :} }
Let $\min_{\cD \in \cup_{j_0 \in \cS}\ccC_{j_0}}\norm{\bgamma_\cD}_2 \leq  \sigma/\sqrt{n} $, i.e, there exists $\cD \in \ccC_{j_0}$ for some $j_0 \in \cS$ such that $$\norm{\bgamma_\cD}_2 \leq  \sigma/\sqrt{n}.$$
Recall that $\by \sim \sfN(\bX_\cS \bbeta_\cS^*, \sigma^2 \bbI_n)$ and define $\bw := \by - \bX_\cS \bbeta_\cS^*$. Hence, we have 
\begin{align*}
    \pr \left(R_\cD - R_\cS < 0\right)& = \pr \left(\norm{\by - \bP_\cD \bX_\cS \bbeta_\cS^*}_2^2 < \norm{\by - \bX_\cS \bbeta_\cS^*}_2^2 \right)\\
    & = \pr \left( \norm{\bw + n^{1/2}\bgamma_\cD}_2^2 < \norm{\bw}_2^2\right)\\
    & \geq \frac{1}{2} \frac{e^{-0.5}}{\sqrt{2 \pi}} > 0.1 \quad (\text{By Lemma \ref{lemma: similar Gaussians are hard to distinguish}}),
\end{align*}
i.e., $\pr \left(\cS \notin \argmin_{\cD: \abs{\cD} = s}R_\cD\right) $ is strictly bounded away from 0. Hence, BSS can not recover the true model. Hence, we rule out this case.

\subsubsection*{\underline{Under the case $\min_{\cD \in \cup_{j_0 \in \cS}\ccC_{j_0}}\norm{\bgamma_\cD}_2 >  \sigma/\sqrt{n} $ :} }
We fix a $j_0 \in \cS$.
\subsubsection*{\textbf{First decomposition:}}

\begin{equation}
    \label{eq: RSS diff 2nd decomp}
    \begin{aligned}
     \min_{\cD \in \ccC_{j_0}} n^{-1}(R_\cD - R_\cS) &\leq \min_{\cD \in \ccC_{j_0}} \{ \bbeta_{\cS\setminus \cD}^\top \Gamma(\cD) \bbeta_{\cS\setminus \cD} - 2 n^{-1/2} \bgamma_\cD^\top \widetilde{\bvarepsilon} - n^{-1} \bvarepsilon^\top (\bP_\cD - \bP_\cS) \bvarepsilon\}\\
     & \leq  \min_{\cD \in \ccC_{j_0}} \left[\norm{\bgamma_\cD}_2 \left\{ \norm{\bgamma_\cD}_2 - 2n^{-1/2} \widehat{\bgamma}_{\cD}^\top \widetilde{\bvarepsilon}\right\} -n^{-1} \bvarepsilon^\top (\bP_\cD - \bP_{\cI_0}) \bvarepsilon\right]\\
     & \quad + n^{-1} \bvarepsilon^\top (\bP_\cS - \bP_{\cI_0}) \bvarepsilon\\
     & \leq  \min_{\cD \in \ccC_{j_0}} \left[\norm{\bgamma_\cD}_2 \left\{ \norm{\bgamma_\cD}_2 - 2n^{-1/2} \widehat{\bgamma}_{\cD}^\top \widetilde{\bvarepsilon}\right\}\right] \quad + n^{-1} \bvarepsilon^\top (\bP_\cS - \bP_{\cI_0}) \bvarepsilon
    \end{aligned}
\end{equation}

    

We start with the quadratic term in the right hand side of the above display. First note that $\bvarepsilon^\top (\bP_\cS - \bP_{\cI_0}) \bvarepsilon/\sigma^2 = (\widehat{\bu}_{j_0}^\top \bvarepsilon)^2/\sigma^2$ follows a chi-squared distribution with degrees of freedom 1. Hence, Markov's inequality shows that 
\begin{equation}
\label{eq: true proj deviation}
   \pr \left\{ n^{-1} \bvarepsilon^\top (\bP_\cS - \bP_{\cI_0}) \bvarepsilon >\frac{2 \sigma^2}{n} \right\} \leq \frac{1}{2}. 
\end{equation}

Recall that
\[
\norm{\bgamma_\cD}_2 \geq \frac{ \sigma}{\sqrt{n}}, \quad \text{for all $\cD\in \ccC_{j_0}$}.
\]

Next, by Sudakov's lower bound, we have 
\[
\bbE \left(\max_{\cD \in \ccC_{j_0}} \widehat{\bgamma}_\cD^\top \widetilde{\bvarepsilon}\right) \geq \sigma \ccEstar_{\cT_{\cI_0}} \sqrt{\log(p-s)} \geq \frac{\ccEstar_{\cT_{\cI_0}}}{\sqrt{2}} \sqrt{\log(ep)}.
\]
The last inequality uses the fact that $s<p/2$ and $p >16 e^3$.
Again, an application of Borell-TIS inequality yields

\begin{equation}
    \label{eq: linear term lower bound}
    \pr \left\{ \max_{\cD \in \ccC_{j_0}} \widehat{\bgamma}_{\cD}^\top \widetilde{\bvarepsilon} \geq \sigma \ccEstar_{\cT_{\cI_0}} \sqrt{\log(ep)}  - c_\cT  \sigma \sqrt{\log(ep)}\right\} \geq 1- (e p)^{- c_\cT^2 /2},
\end{equation}
for any $c_\cT>0$. Choosing $c_\cT = \ccEstar_{\cT_{\cI_0}}(2^{-1/2} - 2^{-1})$, and using the fact that $\ccEstar_{\cT_{\cI_0}}^2 \geq 16 \{\log(ep)\}^{-1}$, we get

\begin{equation}
    \label{eq: linear term lower bound 2}
    \pr \left\{ \max_{\cD \in \ccC_{j_0}} \widehat{\bgamma}_{\cD}^\top \widetilde{\bvarepsilon} \geq \sigma \ccEstar_{\cT_{\cI_0}} \sqrt{\log(ep)}/2  \right\} \geq 1- e^{-1}.
\end{equation}
Let us define $\widehat{\tau}_{j_0} = \max_{\cD \in \ccC_{j_0}} \Delta(\cD) \beta_{j_0}^2.$, and by construction we have $\widehat{\tau}(s) = \max_{j_0 \in \cS}\widehat \tau_{j_0}$.
If $\widehat{\tau}_{j_0}^{1/2} \leq \sigma \ccEstar_{\cT_{\cI_0}} \sqrt{\log(ep)}/(2n^{1/2})$, then we have

\[
\min_{\cD \in \ccC_{j_0}} \left[\norm{\bgamma_\cD}_2 \left\{ \norm{\bgamma_\cD}_2 - 2n^{-1/2} \widehat{\bgamma}_{\cD}^\top \widetilde{\bvarepsilon}\right\}\right] \leq - \frac{ \sigma^2 \ccEstar_{\cT_{\cI_0}} \sqrt{\log(ep)}}{2n} 
\]
Thus, using \eqref{eq: true proj deviation} and the above display we have 
\[
\pr (\cS \notin \argmin_{\cD: \abs{\cD} = s} R_\cD) = \pr \left(\min_{\cD \in \ccC_{j_0}} n^{-1}(R_\cD - R_\cS) <0 \right) \geq 1 - \frac{1}{e} - \frac{1}{2} \geq \frac{1}{10},
\]
as we have $\ccEstar_{\cT_{\cI_0}}\sqrt{\log(ep)} >4 $ (Assumption 2). Thus the necessary condition turns out to be 
\begin{equation}
\label{eq: necessary condition 1st decomposition}
\widehat{\tau}_{j_0} \geq \frac{\ccEstar_{\cT_{\cI_0}}^2}{4} \frac{\sigma^2 \log(ep)}{n}.
\end{equation}

\subsubsection*{\textbf{Second decomposition:}}
We again start with the difference of RSS between a candidate model $\cD \in \ccA_{s,k}$ and the true model $\cS$:
\begin{equation}
\label{eq: RSS diff 2}
    \begin{aligned}
    & n^{-1} (R_\cD - R_\cS)\\
    & = n^{-1}\{ \by^\top (\bbI_n - \bP_\cD ) \by - \by^\top (\bbI_n - \bP_\cS) \by\}\\
    &= n^{-1} \left\{ (\bX_{\cS \setminus \cD} \bbeta_{\cS \setminus \cD} + \bvarepsilon)^\top (\bbI_n - \bP_{\cD}) (\bX_{\cS \setminus \cD} \bbeta_{\cS \setminus \cD} + \bvarepsilon) -  \bvarepsilon^\top (\bbI_n - \bP_\cS) \bvarepsilon \right\}\\
    & =   \bbeta_{\cS\setminus \cD}^\top \Gamma(\cD) \bbeta_{\cS\setminus \cD} - 2 \left\{n^{-1} (\bbI_n - \bP_{\cD}) \bX_{\cS\setminus \cD}\bbeta_{\cS\setminus \cD}\right\}^\top \widetilde{\bvarepsilon} - n^{-1} \bvarepsilon^\top(\bP_\cD - \bP_\cS) \bvarepsilon.
    \end{aligned}
\end{equation}

First of all, in order achieve model consistency, the following is necessary for any $k \in [s]$:

\begin{equation}
\label{eq: RSS diff necessary}
\min_{\cD\in \ccC_{j_0}} n^{-1}(R_\cD - R_{\cS}) >0. 
\end{equation}
Recall that $\widehat{\tau}_{j_0} = \max_{\cD \in \ccC_{j_0}} \Gamma(\cD) \beta_{j_0}^2$.
Next we note that

\begin{equation}
    \label{eq: RSS diff 3}
    \begin{aligned}
     \min_{\cD \in \cD_{j_0}} n^{-1}(R_\cD - R_\cS) &\leq \min_{\cD \in \ccC_{j_0}} \{ \bbeta_{\cS\setminus \cD}^\top \Gamma(\cD) \bbeta_{\cS\setminus \cD} - 2 n^{-1/2} \bgamma_\cD^\top \widetilde{\bvarepsilon} - n^{-1} \bvarepsilon^\top (\bP_\cD - \bP_\cS) \bvarepsilon\}\\
     & \leq  \widehat{\tau}_{j_0}  +2 n^{-1/2}\max_{\cD \in \ccC_{j_0}} \abs{\bgamma_\cD^\top \widetilde{\bvarepsilon}} - n^{-1}\max_{\cD \in \ccC_{j_0}} \bvarepsilon^\top (\bP_\cD - \bP_\cS) \bvarepsilon\\
     & \leq \widehat{\tau}_{j_0} + 2 (\widehat{\tau}_{j_0}/n)^{1/2} \max_{\cD \in \ccC_{j_0}} \abs{\widehat{\bgamma}_\cD^\top \widetilde{\bvarepsilon}} - n^{-1}\max_{\cD \in \ccC_{j_0}} \bvarepsilon^\top (\bP_\cD - \bP_\cS) \bvarepsilon.
    \end{aligned}
\end{equation}

Similar to the proof of Theorem 1, we define $f_\cD : = \widehat{\bgamma}_{\cD}^\top \widetilde{\bvarepsilon} $ and $\norm{f} := \max_{\cD\in \ccC_{j_0}} f_\cD$. Hence, we have 
\begin{equation}
\label{eq: max T}
\max_{\cD\in \ccC_{j_0}} \abs{\widehat{\bgamma}_\cD^\top \widetilde{\bvarepsilon}} = \max_{\cD \in \ccC_{j_0}} f_\cD \vee (-f_\cD)
\end{equation}
By Borell-TIS inequality \cite[Theorem 2.1.1]{adler2007random}, we have 
\begin{equation*}
\label{eq: borell-TIS 2}
\pr \left\{ \norm{f} - \bbE(\norm{f}) \geq \sigma u \right\} \leq \exp \left(- \frac{u^2}{2 }\right),
\end{equation*}
for all $u>0$. Setting $u = c_\cT\sqrt{ \log (ep)}$ we get 

\[
\pr \left\{ \norm{f} - \bbE(\norm{f}) \geq c_\cT\sigma \sqrt{  \log (ep)}\right\} \leq (ep)^{-c_\cT^2 /2}.
\]

\[
\bbE(\norm{f}) \leq 4\sqrt{2} \ccE_{\cT_{\cI_0}} \sigma \sqrt{ \log(ep)},
\]
which ultimately yields 
\[
\pr \left\{ \norm{f} \geq (4\sqrt{2} \ccE_{\cT_{\cI_0}} + c_\cT) \sigma \sqrt{ \log(ep)}\right\} \leq (ep)^{-c_\cT^2 /2}.
\]
Finally. using \eqref{eq: max T} we have the following for any $c_\cT>0$:
\begin{equation}
\label{eq: lin deviation}
    \pr \left\{ \max_{\cD\in \ccC_{j_0}} \abs{\widehat{\bgamma}_\cD^\top \widetilde{E}} \geq (4\sqrt{2} \ccE_{\cT_{\cI_0}} + c_\cT) \sigma \sqrt{ \log(ep)}\right\} \leq 2 (ep)^{-c_\cT^2 /2}.
\end{equation}
Next, we will lower bound the quadratic term in Equation \eqref{eq: RSS diff 3} with high probability.
similar to the proof of Theorem 1, we consider the decomposition
\[
\max_{\cD \in \ccC_{j_0}}n^{-1} \bvarepsilon^\top (\bP_\cD - \bP_\cS)\bvarepsilon = \max_{j \notin \cS} n^{-1} \bvarepsilon^\top (\widehat{\bu}_j \widehat{\bu}_j^\top - \widehat{\bu}_{j_0}\widehat{\bu}_{j_0}^\top)\bvarepsilon. 
\]



For the maximal process we will use Theorem 2.10 of \cite{adamczak2015note}. We begin with the definition of concentration property.
\begin{definition}[\cite{adamczak2015note}]
\label{def: concentration property}
Let $Z$ be random vector in $\bbR^n$. We say that $Z$ has concentration property with constant $K$ if for every 1-Lipschitz function $\varphi : \bbR^n \to \bbR$, we have $\bbE \abs{\varphi(Z)} < \infty$ and for every $u>0$,
\[
\pr \left( \abs{\varphi(Z) - \bbE(\varphi(Z))} \geq u\right) \leq 2 \exp(- u^2/K^2).
\]
\end{definition}
Note that the Gaussian vector $\bvarepsilon/\sigma$ enjoys concentration property with $K = \sqrt{2}$ \cite[Theorem 5.6]{boucheron2013concentration}. Let $Q_1 : =  \max_{j \notin \cS} \bvarepsilon^\top (\widehat{\bu}_j \widehat{\bu}_j^\top - \widehat{\bu}_{j_0}\widehat{\bu}_{j_0}^\top)\bvarepsilon$.
By Theorem 2.10 of \cite{adamczak2015note} we conclude that 
\[
\pr \left\{ n^{-1}\abs{Q_1 - \bbE(Q_1)} \geq t \sigma^2 \right\} \leq 2 \exp \left\{ - \frac{1}{2} \min \left( \frac{n^2 t^2}{16}, \frac{n t}{2}\right)\right\}.
\]

Setting $t =  2 \delta \log(ep)/n$ in the above equation we get
\begin{equation}
    \label{eq: thm2 : quad deviation 2}
    \pr \left\{ n^{-1}\abs{Q_1 - \bbE(Q_1)} \geq  2 \delta  \sigma^2  \log(ep)/n \right\} \leq 2 (ep)^{-\frac{ \delta}{2 }  }.
\end{equation}

Next, we will lower bound the expected value of $Q_1$. First, note the following:
\begin{equation}
\label{eq: expected quadratic}
\begin{aligned}
    \bbE(Q_1) & = \bbE \left\{\max_{j \notin \cS} \bvarepsilon^\top (\widehat{\bu}_j \widehat{\bu}_j^\top - \widehat{\bu}_{j_0} \widehat{\bu}_{j_0}^\top) \bvarepsilon\right\}  \\
    & = \bbE \left\{\max_{j \notin \cS} (\widehat{\bu}_{j}^\top \bvarepsilon)^2 \right\}  - \sigma^2
    \\
    & \geq  \left\{\bbE \max_{j \in \cS^c}\; (\widehat{\bu}_{j}^\top \bvarepsilon) \vee (-\widehat{\bu}_{j}^\top \bvarepsilon)\right\}^2  - \sigma^2
\end{aligned}
\end{equation}
Define the set 
\[
\cU_{\rm sym}:= \{\widehat{\bu}_j \; : j \in \cS^c\} \cup \{-\widehat{\bu}_j \; : j \in \cS^c\}.
\]
We denote by $\widetilde{\bu}_j$ a generic element of $\cU_{\rm sym}$. Thus for any two elements $\widetilde{\bu}_j, \widetilde{\bu}_k$, we have
\[
\norm{\widetilde{\bu}_j - \widetilde{\bu}_k}_2 \geq \min \left\{ \norm{\widehat{\bu}_j - \widehat{\bu}_k}_2, \norm{\widehat{\bu}_j + \widehat{\bu}_k}_2\right\} \geq \norm{\widehat{\bu}_j \widehat{\bu}_j^\top - \widehat{\bu}_k \widehat{\bu}_k^\top}_\op.
\]
By Sudakov's lower bound, we have
\begin{align*}
\bbE \max_{j \in \cS^c} \; (\widehat{\bu}_j^\top \bvarepsilon) \vee (- \widehat{\bu}_j^\top \bvarepsilon) 
&\geq \sup_{\delta>0} \frac{\sigma \delta}{2} \sqrt{\log \cM( \cU_{\rm sym}, \norm{\cdot}_2, \delta)}\\
& \geq \sup_{\delta>0} \frac{\sigma \delta}{2} \sqrt{\log \cM( \{\widehat{\bu}_j \widehat{\bu}_j^\top\}_{j \in \cS^c}, \norm{\cdot}_\op,  \delta)}\\
& \ge \sigma \frac{\ccEstar_{\cG_{\cI_0}}}{\sqrt{4/3}} \sqrt{ \log(ep)}.
\end{align*}
The last inequality uses the fact that $p>16e^3$.
Finally, \eqref{eq: expected quadratic} yields
\[
\bbE(Q_1) \geq \frac{\sigma^2 \ccEstar_{\cG_{\cI_0}}^2   \log (ep)}{4/3} - \sigma^2.
\]

Thus, combined with \eqref{eq: thm2 : quad deviation 2} we finally get

\[
\pr \left[ Q_1 \geq (3/4)\sigma^2 \ccEstar_{\cG_{\cI_0}}^2 \log (ep) - \sigma^2  - 2 \delta \sigma^2 \log(ep)  \right] \geq 1 - 2 (ep)^{-\delta/2 }.
\]
Setting $\delta = \frac{\ccEstar_{\cG_{\cI_0}}^2}{8}$, we get 
\[
\pr \left[ Q_1 \geq \frac{\sigma^2 \ccEstar_{\cG_{\cI_0}}^2}{2} \log (ep) - \sigma^2  \right] \geq 1 - 2 (ep)^{-\frac{\ccEstar_{\cG_{\cI_0}}^2}{16} }.
\]
Thus, finally combining the above with \eqref{eq: RSS diff 3} and \eqref{eq: lin deviation} we get
\begin{equation}
    \label{eq: RSS diff 4}
    \begin{aligned}
     &\min_{\cD \in \ccC_{j_0}} n^{-1}(R_\cD - R_\cS)\\
     & \leq \widehat{\tau}_{j_0} + 2 \widehat{\tau}_{j_0}^{1/2} (4\sqrt{2} \ccE_{\cT_{\cI_0}} + c_\cT) \sigma \sqrt{\frac{\log(ep)}{n}} - n^{-1}\left( \frac{\sigma^2 \ccEstar_{\cG_{\cI_0}}^2}{2} \log (ep) - \sigma^2    \right)
    \end{aligned}
\end{equation}
with probability at least $1 - 2 (ep)^{-c_{\cT}^2/2} - 2(ep)^{- \frac{\ccEstar_{\cG_{\cI_0}}^2}{16}}$. Thus, for large $p$ we have
\begin{equation*}
    \begin{aligned}
     &\min_{\cD \in \ccC_{j_0}} n^{-1}(R_\cD - R_\cD)\\
     & \leq \widehat{\tau}_{j_0} + 2 \widehat{\tau}_{j_0}^{1/2} (4\sqrt{2} \ccE_{\cT_{\cI_0}} + c_\cT) \sigma \sqrt{\frac{\log(ep)}{n}} - \frac{\sigma^2 \ccEstar_{\cG_{\cI_0}}^2}{4} \frac{\log (ep)}{n} 
    \end{aligned}
\end{equation*}
with probability at least $1 - 2 (ep)^{c_{\cT}^2/2} - 2(ep)^{- \frac{\ccEstar_{\cG_{\cI_0}}^2}{16}}$. Now choose $c_\cT = 4/\sqrt{\log(ep)}$ and use Assumption 2 to get 
\begin{equation*}
    \begin{aligned}
     &\min_{\cD \in \ccC_{j_0}} n^{-1}(R_\cD - R_\cD)\\
     & \leq \widehat{\tau}_{j_0} + 2 \widehat{\tau}_{j_0}^{1/2} \left(4\sqrt{2} \ccE_{\cT_{\cI_0}} + \frac{4}{\sqrt{\log(ep)}}\right) \sigma \sqrt{\frac{\log(ep)}{n}} - \frac{\sigma^2 \ccEstar_{\cG_{\cI_0}}^2}{4} \frac{\log (ep)}{n} \\
     & \leq \widehat{\tau}_{j_0} + 8\sqrt{2} \widehat{\tau}_{j_0}^{1/2} \left( \ccE_{\cT_{\cI_0}} + \frac{1}{\sqrt{2\log(ep)}}\right) \sigma \sqrt{\frac{\log(ep)}{n}} - \frac{\sigma^2 \ccEstar_{\cG_{\cI_0}}^2}{4} \frac{\log (ep)}{n}\\
     & \leq \widehat{\tau}_{j_0} + 10\sqrt{2} \widehat{\tau}_{j_0}^{1/2}  \ccE_{\cT_{\cI_0}} \sigma \sqrt{\frac{\log(ep)}{n}} - \frac{\sigma^2 \ccEstar_{\cG_{\cI_0}}^2}{4} \frac{\log (ep)}{n}
    \end{aligned}
\end{equation*}
with probability at least $1/5$.

Thus, in light of \eqref{eq: RSS diff necessary}, the following is necessary:
\begin{equation}
\label{eq: neccessary condition 2nd decomposition}
\widehat{\tau}_{j_0}\geq \left\{\frac{\sqrt{200 \ccE_{\cT_{\cI_0}}^2 + \ccEstar_{\cG_{\cI_0}}^2} - 10\sqrt{2} \ccE_{\cT_{\cI_0}}}{2}\right\}^2 \frac{\sigma^2 \log(ep)}{n}.
\end{equation}

\paragraph{\textbf{Case 1:}}
If $\ccE_{\cT_{\cI_0}} \leq \ccEstar_{\cG_{\cI_0}} $, then the right hand side of \eqref{eq: neccessary condition 2nd decomposition} is lower bounded by
\[
\frac{\ccEstar_{\cG_{\cI_0}}^2}{(\sqrt{201} + 10\sqrt{2})^2} \frac{\sigma^2\log(ep)}{n}.
\]
Thus \eqref{eq: neccessary condition 2nd decomposition} yields the necessary condition
\[
\widehat{\tau}_{j_0} \geq \frac{\ccEstar_{\cG_{\cI_0}}^2}{(\sqrt{201} + 10\sqrt{2})^2} \frac{\sigma^2\log(ep)}{n}.
\]
Combining this with \eqref{eq: necessary condition 1st decomposition} we have the necessary condition to be 
\[
\widehat{\tau}_{j_0} \geq \tilde{C}_1 \max\{\ccEstar_{\cT_{\cI_0}}^2, \ccEstar_{\cG_{\cI_0}}^2 \} \frac{\sigma^2 \log(ep)}{n},
\]
for a universal constant $\tilde{C}_1$.

\paragraph{\textbf{Case 2:}} If $\ccEstar_{\cG_{\cI_0}} \leq \ccEstar_{\cT_{\cI_0}}$, then using the inequality $\sqrt{1+t} - \sqrt{t}<1$ for all $t>0$, we can conclude that the right hand side of \eqref{eq: neccessary condition 2nd decomposition} is always smaller than
\[
\frac{\ccEstar_{\cG_{\cI_0}}^2}{4} \frac{\sigma^2\log(ep)}{n},
\]
which is further smaller than $$\frac{\ccEstar_{\cT_{\cI_0}}^2}{4} \frac{\sigma^2\log(ep)}{n}.$$
Thus, combining this with \eqref{eq: necessary condition 1st decomposition} we have the necessary condition to be 
\[
\widehat{\tau}_{j_0} \geq \frac{\ccEstar_{\cT_{\cI_0}}^2}{4} \frac{\sigma^2\log(ep)}{n} = \max\{\ccEstar^2_{\cT_{\cI_0} }, \ccEstar_{\cG_{\cI_0}}^2\} \frac{\sigma^2 \log(ep)}{4 n}.
\]

Combining all these cases we finally have the following necessary condition for consistent model selection with $C_1, C_2>0$ being some absolute constants:

\begin{equation}
\label{eq: necessary margin 1}
\widehat{\tau}_{j_0} \geq C_1 \max\{\ccEstar_{\cT_{\cI_0}}^2, \ccEstar_{\cG_{\cI_0}}^2 \} \frac{\sigma^2 \log(ep)}{n}, \quad \text{if $\ccEstar_{\cG
_{\cI_0}} \notin (\ccEstar_{\cT_{\cI_0}}, \ccE_{\cT_{\cI_0}})$},
\end{equation}
or,
\[
\widehat{\tau}_{j_0} \geq C_2 \max\left\{ \ccEstar_{\cT_{\cI_0}}^2, \left( \sqrt{200 \ccE_{\cT_{\cI_0}}^2 + \ccEstar_{\cG_{\cI_0}}^2} - 10\sqrt{2} \ccE_{\cT_{\cI_0}}\right)^2 \right\}\frac{\sigma^2 \log(ep)}{n}, 
\]
if 
$\ccEstar_{\cG
_{\cI_0}} \in (\ccEstar_{\cT_{\cI_0}}, \ccE_{\cT_{\cI_0}})$.

If there exists $\cI_0 \in \cJ$ such that $\ccEstar_{\cT_{\cI_0}}/\ccE_{\cT_{\cI_0}} \in (\alpha, 1)$, then in the preceding case it follows that the last display can be simplified in the following form:
\[
\widehat{\tau}_{j_0} \geq  C_2 \max\left\{ \ccEstar_{\cT_{\cI_0}}^2, A_\alpha \ccEstar_{\cG_{\cI_0}}^2 \right\}\frac{\sigma^2 \log(ep)}{n}, \quad \text{if $\ccEstar_{\cG
_{\cI_0}} \in (\ccEstar_{\cT_{\cI_0}}, \ccE_{\cT_{\cI_0}})$},
\]
where 
\[
A_\alpha =
\left(\sqrt{\frac{200}{\alpha^2} + 1} + \frac{10\sqrt{2}}{\alpha}\right)^{-1}.
\]
Thus, using the fact that $ A_\alpha <1$ and combining the previous three displays we have the necessary condition to be
\begin{equation}
\label{eq: necessary margin 2}
\widehat{\tau}_{j_0} \geq C_2 A_\alpha \max\left\{ \ccEstar_{\cT_{\cI_0}}^2, \ccEstar_{\cG_{\cI_0}}^2 \right\}\frac{\sigma^2 \log(ep)}{n}.
\end{equation}
Since, \eqref{eq: necessary margin 1} and \eqref{eq: necessary margin 2} hold for all choices of $j_0$ and $\cI_0$ depending on whether the $\ccEstar_{\cG
_{\cI_0}} \in (\ccEstar_{\cT_{\cI_0}}, \ccE_{\cT_{\cI_0}})$ is satisfied or not, the claim follows.
\end{proof}
\subsection{Correlated random feature model example (Proof of Lemma 2)}
\label{sec: random feature model example}
Consider the model (1). We assume that the rows of $\bX$ are independently generated from $p$-dimensional multivariate Gaussian distribution with mean-zero and variance-covariance matrix

\[
\bSigma = \begin{pmatrix}
1 & c \bone_{p-1 }^\top\\
c\bone_{p-1} & (1- r)\bbI_{p-1} + r \bone_{p-1} \bone_{p-1}^\top
\end{pmatrix},
\]
where $c \in [0, 0.997], r \in [0,1)$ and true model is $\cS = \{1\}$. We need to further impose the restriction 
$$c^2< r + \frac{1-r}{p-1}$$ 
to ensure positive definiteness of $\Sigma$. In this case

\[
\widehat{\tau} = \beta_{1}^2 \min_{j \neq 1} \left\{ \frac{\norm{\bX_{1}}_2^2}{n} - \frac{(\bX_{1}^\top \bX_{j}/n)^2}{\norm{\bX_{j}}_2^2/n} \right\}  = \frac{\beta_1^2 \norm{\bX_1}_2^2}{n} - \beta_1^2 \max_{j \neq 1} \frac{(\bX_1^\top \bX_j/n)^2}{\norm{\bX_j}_2^2/n}.
\]
We start with providing an upper bound on the margin quantity $\widehat{\tau}$. Using Equation \eqref{eq: LM1-simple} and \eqref{eq: LM2-simple}, for any $\tol\in (0,1)$ we get

\begin{equation}
    \label{eq: norm concentration 1}
    \pr \left( \frac{\norm{\bX_j}_2^2}{n} \geq 1 + \tol\right) \leq \exp(-n \tol^2/16), \quad \forall j \in [p].
\end{equation}

\begin{equation}
    \label{eq: norm concentration 2}
    \pr \left( \frac{\norm{\bX_j}_2^2}{n} \leq 1 - \tol\right) \leq \exp(-n \tol^2/4), \quad \forall j \in [p].
\end{equation}

Using Equation \eqref{eq: norm concentration 1} and Equation \eqref{eq: norm concentration 2}, we also get the following:
\begin{equation}
\label{eq: norm concentration 3}
 \begin{aligned}
    \pr \left(\max_{j \neq 1} \abs{\frac{\norm{\bX_j}_2^2}{n} - 1} \geq \tol \right) \leq 2 p \exp(-n \tol^2/16).
 \end{aligned}
\end{equation}

Let $\bX_j = (x_{1,j}, \ldots, x_{n,j})^\top$ for all $j \in [p]$. To this end we recall the \textit{sub-Gaussian norm} $\norm{\cdot}_{\psi_2}$ \cite[Definiton 2.5.6]{vershynin_2018} and the \textit{sub-exponential norm} $\norm{\cdot}_{\psi_1}$ \cite[Definition 2.7.5]{vershynin_2018}.
Now due to \cite[Lemma 2.7.7]{vershynin_2018}, we have that $\norm{x_{u,1} x_{u,j}}_{\psi_1} \leq \norm{x_{u,1}}_{\psi_2} \norm{x_{u,2}}_{\psi_2} \leq 4.$
Thus, by Berstein's inequality, we have
\[
\pr \left( \abs{\frac{\bX_1^\top \bX_j}{n} - c} > \tol \right) \leq \exp \left( - C n \min\{\tol,\tol^2\} \right),
\]
where $C>0$ is a universal constant. Thus, we have
\begin{equation}
    \label{eq: cross term 1}
    \pr \left( \max_{j \neq 1}\abs{\frac{\bX_1^\top \bX_j}{n} - c} > \tol \right) \leq p \exp(- C n \min\{\tol,\tol^2\}).
\end{equation}
Combining \eqref{eq: norm concentration 2}, \eqref{eq: norm concentration 3} and \eqref{eq: cross term 1} we have
\begin{equation}
    \label{eq: prob margin}
    \begin{aligned}
    &\pr \left[ \left\{1 + \tol - \frac{(c - \tol)^2}{1 + \tol} \right\}\geq \frac{\widehat{\tau}}{\beta_1^2} \geq  \left\{1 - \tol - \frac{(c+ \tol)^2}{1 - \tol} \right\} \right]\\
    & \geq 1 - \exp(-n \tol^2/16) - 2p\exp(-n \tol^2/4) - p \exp(-C \tol^2 n)\\
    & = 1 + o(1/p),
    \end{aligned}
\end{equation}
if $\tol \asymp   \{(\log p)/n\}^{1/2}$ and $(\log p)/n$ is small enough. Similarly, due to Bernstein's inequality,
it can also be shown that 

\begin{equation}
    \label{eq: cross term 2}
    \pr \left( \max_{j,k \neq 1}\abs{\frac{\bX_k^\top \bX_j}{n} - r} \leq \tol\right) \geq 1 - p^2 \exp(- C n \tol^2) = 1 + o(1/p),
\end{equation}
where $r \in [0,1)$ and with the same conditions on $\tol$.

Next, we will analyze the geometric quantities.
In this case, we have 
\[
\widehat{\bgamma}_j = \frac{\bX_{1} - \frac{\bX_j^\top \bX_{1}}{\norm{\bX_j}_2^2}.\bX_{j}}{\sqrt{\norm{\bX_{1}}_2^2 - \frac{(\bX_{1}^\top \bX_{j})^2}{\norm{\bX_{j}}_2^2}}}.
\] 

Note that \begin{align*}
    \norm{\widehat{\bgamma}_j - \widehat{\bgamma}_k}_2^2 & = 2 (1 - \widehat{\bgamma}_j^\top \widehat{\bgamma}_k)
\end{align*}
and 
\begin{align*}
    \widehat{\bgamma}_j^\top \widehat{\bgamma}_k &= \dfrac{\norm{\bX_{1}}_2^2/n - \frac{
    (\bX_j^\top \bX_{1}/n)^2 }{\norm{\bX_j}_2^2/n} - \frac{
    (\bX_k^\top \bX_{1}/n)^2 }{\norm{\bX_k}_2^2/n} +\frac{
    (\bX_j^\top \bX_{1}/n) (\bX_k^\top \bX_{1}/n) (\bX_j^\top \bX_{k}/n)}{(\norm{\bX_j}_2^2/n) (\norm{\bX_k}_2^2/n)}}{\sqrt{\norm{\bX_{1}}_2^2/n - \frac{(\bX_{1}^\top \bX_{j}/n)^2}{\norm{\bX_{j}}_2^2/n}} \sqrt{\norm{\bX_{1}}_2^2/n - \frac{(\bX_{1}^\top \bX_{k}/n)^2}{\norm{\bX_{k}}_2^2/n}}}.
\end{align*}
Next, we consider the event
\[
\cG_n:= \left\{
\max_{j \in [p]} \abs{\frac{\norm{\bX_j}_2^2}{n} - 1} \leq \tol, \max_{j \neq 1} \abs{\frac{\bX_1^\top \bX_j}{n} - c} \leq \tol, \max_{j,k \neq 1}\abs{\frac{\bX_k^\top \bX_j}{n} - r} \leq \tol
\right\}.
\]
Due to \eqref{eq: norm concentration 1}, \eqref{eq: norm concentration 2}, \eqref{eq: cross term 1} and \eqref{eq: cross term 2} we have $\pr(\cG_n) = 1+ o(1/p)$. Also, for large $n,p$, the value of $\tol$ can be chosen such that $\tol <0.001$ so that $c+ \tol < 0.998$ for all $c \in [0, 0.997]$.

\paragraph{Complexity of unexplained signals:}
Let $\bu := (u_1, u_2,u_3)\in \bbR^3$ and $\bt := (t_1, t_2, t_2)\in \bbR^3$. 
Define the function
\[
\Phi(\bu, \bt) := \frac{u_1- (t_1^2/u_2) - (t_2^2/u_3) + (t_1 t_2 t_3)/(u_2 u_3)}{\sqrt{u_1 - t_1^2/u_2} \sqrt{u_1 - t_2^2/u_3}},
\]
where 
$$(u_1,u_2,u_3, t_1, t_2, t_3) \in \underbrace{[0.999, 1.001]\times[0.999, 1.001] \times [0.999, 1.001]\times [0,0.998]\times [0,0.998]\times [0,1]}_{:= \cK}.$$
It is easy to see that the function $\Phi$ is continuously differentiable on the compact set $\cK$. Hence, there exists a universal constant $L>0$ such that 
\[
\abs{\Phi(\bu, \bt) - \Phi(\bu^\prime, \bt^\prime)} \leq L (\norm{\bu - \bu^\prime}_1 + \norm{\bt - \bt^\prime}_1). 
\]
Noting the fact that 
\[
\widehat{\bgamma}_j^\top \bgamma_{k} = \Phi\left(\frac{\norm{\bX_1}_2^2}{n}, \frac{\norm{\bX_j}_2^2}{n}, \frac{\norm{\bX_K}_2^2}{n}, \frac{\bX_1^\top \bX_j}{n}, \frac{\bX_1^\top \bX_k}{n}, \frac{\bX_j^\top \bX_k}{n}\right),
\]
it follows that on the event $\cG_n$, the following holds for all $j,k\in [p]\setminus \{1\}:$
\begin{align*}
&\abs{\norm{\widehat{\bgamma}_j - \widehat{\bgamma}_k}_2^2 - \frac{2c^2(1-r)}{1-c^2}}  \leq 12 L \tol, \\
&\Rightarrow \sqrt{\max\left\{
\frac{2c^2(1-r)}{1-c^2} - 12L \tol, 0
\right\}} \leq \norm{\widehat{\bgamma}_j - \widehat{\bgamma}_k}_2 \leq \sqrt{\frac{2c^2(1-r)}{1-c^2} + 12L \tol}.\\ 
\end{align*}
Hence we have
{\small
\begin{align*}
&\ccE_{\cT_\emptyset}\\
& \hspace{-.3in}= \{\log(ep)\}^{-1/2} \Big[\int_0^{\sqrt{\left(\frac{2c^2(1-r)}{1-c^2} - 12L \tol\right)\vee 0}} \sqrt{\log \cN( \cT_\emptyset, \norm{\cdot}_2, \varepsilon)}\; d\varepsilon\\
& +
\int_{\sqrt{\left(\frac{2c^2(1-r)}{1-c^2} - 12L \tol\right)\vee 0}}^{\sqrt{\frac{2c^2(1-r)}{1-c^2} + 12L \tol}} \sqrt{\log \cN( \cT_\emptyset, \norm{\cdot}_2, \varepsilon)}\; d\varepsilon\Big].
\end{align*}
}
Applying Lemma \ref{lemma: sqrt lipschitz} on the second integral, it follows that
{
\[
\hspace{-0.2in} \omega_{n,p} \sqrt{\frac{\log p}{\log(ep)}} \leq \ccE_{\cT_\emptyset} \leq \left(\omega_{n,p} + \sqrt{24 L\tol}\right) \sqrt{\frac{\log p}{\log(ep)}},
\]
}
where $\omega_{n,p} = \sqrt{\left(\frac{2c^2(1-r)}{1-c^2} - 12L \tol\right)\vee 0}$.
Thus, for $c =0$ we have $0 \leq \ccE_{\cT_\emptyset} \leq \sqrt{24L \tol}$. For any fixed $c>0$ and $r \in [0,1)$ we have
\begin{equation}
\label{eq: prob_example_linear_complexity}
\ccE_{\cT_\emptyset} \sim \left\{ \frac{2c^2(1-r)}{1-c^2}\right\}^{1/2} \quad \text{for large $n,p$}.
\end{equation}


\paragraph{Complexity of spurious projections:}
For $j, k \neq 1$, let $\theta_{j,k}$ denote the
angle between $\bX_j$ and $\bX_k$.
\[
\norm{\bP_j- \bP_k}_{\op} = \sin(\theta_{j,k}) = \sqrt{1 - \cos^2(\theta_{j,k})} =  \sqrt{1 - \left( \frac{\bX_j^\top \bX_k}{\norm{\bX_j}_2 \norm{\bX_k}_2}\right)^2}.
\]
By a similar argument as above, we can conclude that there exists a universal constant $M>0$ such that on the event $\cG_n$ we have,
\[
\abs{\norm{\bP_j - \bP_k}_\op^2 - (1 - r^2)} \leq M \tol, \quad \text{for all $j,k \in [p]\setminus \{1\}$}.
\]
Thus, for any fixed $r \in [0,1)$ we have
\begin{equation}
\label{eq: prob_example_proj_complexity}
\ccE_{\cG_\emptyset} \sim (1-r^2)^{1/2}.
\end{equation}

\section{Generalized linear model}
\label{supp: GLM}
In this section, we will focus on the best subset selection problem under generalized linear models (GLM). Similar to the linear regression setup, we will also adopt the fixed design setup in this case. In particular, given the data matrix $\bX := (\bx_1, \ldots, \bx_n)^\top \in \bbR^{n \times p}$ we observe the responses $\by:= (y_1, \ldots, y_n)^\top$ coming from the distribution
\begin{equation}
\label{eq: glm pdf}
f_{\bx, \bbeta^*}(y):= h(y) \exp \left\{\frac{y (\bx^\top \bbeta^*) - b(\bx^\top \bbeta^*)}{\phi}\right\} =  h(y) \exp \left\{\frac{y \eta - b(\eta)}{\phi}\right\}.
\end{equation}
Here $\eta = \bx^\top \bbeta^*$ is linear predictor and $\bbeta^*$ is true parameter with $\norm{\bbeta^*}_0 = s$ and support $\cS$. The functions $b: \bbR \to \bbR$ and $h: \bbR\to \bbR$ are known and specific to modeling assumptions. Examples include several well-known models  such as
\begin{enumerate}
    \item Linear regression: Consider the linear regression model $y = \bx^\top \bbeta^* + \varepsilon$, where $\varepsilon\sim \sfN(0, \sigma^2)$. In this case $h(y) = \exp\{-y^2/(2\sigma^2)\}$ and $b(u) =  u^2/2$.
    \item Logistic regression: In this model $y \sim \Ber(1/(1 + \exp(- \bx^\top \bbeta^*)))$. Standard calculations show that $h(y) =1$ and $b(u) = \log(1+e^u)$.
\end{enumerate}

For the purpose of model selection, we choose the loss function to be the scaled negative log-likelihood function 
\[\cL(\bbeta;\{(\bx_i,y_i)\}_{i \in [n]}) = \frac{2}{n}\sum_{i \in [n]} \ell(\bbeta; (\bx_i, y_i)),\]
where  
$
\ell(\bbeta; (\bx, y)) = -y(\bx^\top \bbeta) + b(\bx^\top \bbeta).
$
Furthermore, for a candidate model $\cD\in \ccA_s$ and $\tilde{\bbeta}\in \bbR^s$, define the restricted version of the scaled negative log-likelihood function as $\cL_{\cD}(\tilde{\bbeta};\{(\bx_{i, \cD},y_i)\}_{i \in [n]}) := (n/2)^{-1}\sum_{i \in [n]} \ell_\cD(\tilde{\bbeta}; (\bx_{i,\cD}, y_i))$ with
$\ell_\cD(\Tilde{\bbeta}; (\bx_\cD,y)):= -y(\bx_\cD^\top \tilde{\bbeta}) + b(\bx_{\cD}^\top \tilde{\bbeta}) $. Let $\widehat{\bbeta}_{\cD}$ be the unique minimizer of $\cL_{\cD}(\tilde{\bbeta};\{(\bx_{i, \cD},y_i)\}_{i \in [n]})$.
Under the oracle knowledge of sparsity $s$, BSS solves for 
\[
\widehat{\cS}_{\rm best} = n^{-1}\argmin_{\cD: \abs{\cD}=s} \cL_{\cD}(\widehat{\bbeta}_\cD; \{(\bx_{i,\cD},y_i)\}_{i \in [n]}).
\]
Next, we will introduce the quantities that capture the degree of separation between the true model $\cS$ and a candidate model $\cD \in \ccA_s$ and characterize the identifiability margin for model selection consistency. Let $\cP_{\cD, \tilde{\bbeta}}$ be probability measure corresponding to the joint density $\prod_{i \in [n]} f_{\bx_{i,\cD}, \tilde{\bbeta}}(y_i)$ and define 
\begin{equation*}
\begin{aligned}
\Delta_\kl(\cD) & := \frac{2 \phi}{n}\min_{\tilde{\bbeta} \in \bbR^s}\KL\left( \cP_{\cS, \bbeta^*_\cS} \; \big\Vert \; \cP_{\cD, \tilde{\bbeta}}\right)\\
& = \frac{2 }{n} \sum_{i=1}^n \left\{(\bx_{i,\cS}^\top\bbeta^*_\cS) b^\prime(\bx_{i,\cS}^\top\bbeta^*_\cS) - b(\bx_{i,\cS}^\top\bbeta^*_\cS)\right\}\\
& \quad - \max_{\tilde{\bbeta} \in \bbR^s}\frac{2}{n} \sum_{i=1}^n \left\{(\bx_{i,\cD}^\top \tilde{\bbeta}) b^\prime(\bx_{i,\cS}^\top\bbeta^*_\cS) - b(\bx_{i,\cD}^\top \tilde{\bbeta})\right\},  
\end{aligned}
\end{equation*}
where $\KL(\cdot \Vert \cdot)$ denotes the Kullback-Leibler (KL) divergence.
The above quantity can be thought of as the degree of model separation as it measures the minimum KL-divergence between the likelihood generated by the data under $(\cS, \bbeta^*_\cS)$ and the likelihood generated by $\cD$ and all possible choices of $\bbeta \in \bbR^p$ with the support in $\cD$. 
Let $\bar{\bbeta}_\cD$ be the minimizer of the optimization problem in the above display, i.e.,
\[
\bar{\bbeta}_\cD:= \argmin_{\tilde{\bbeta} \in \bbR^s}\KL\left( \cP_{\cS, \bbeta^*_\cS} \; \big\Vert \; \cP_{\cD, \tilde{\bbeta}}\right).
\]
By definition it follows that $\bar{\bbeta}_\cS = \bbeta^*_\cS$.
Also, note that for $\cP_{\cD, \bar{\bbeta}_\cD}$, the \textit{natural parameter} of the density function is $\bX_\cD \bar{\bbeta}_\cD$. Thus, one can also measure the separation between two models through the mutual distance between the corresponding natural parameters. This motivates the definition of the second measure of separability between the true model $\cS$ and candidate model $\cD$:
\[
\Delta_\param(\cD) := \frac{\norm{\bX_\cS \bbeta_\cS^* - \bX_\cD \bar{\bbeta}_\cD}_2^2
}{n}.
\]
Note that, under the linear regression model with isotropic Gaussian error, both $\Delta_\kl(\cD)$ and $\Delta_\param(\cD)$ becomes equal to  the quantity $\bbeta_{\cS \setminus \cD}^\top \Gamma(\cD) \bbeta_{\cS \setminus \cD}$. To see this, recall that for linear regression model $b(u) = u^2/2$ and the KL-divergence $\KL(\cP_{\cS, \bbeta^*_\cS} \; \big \Vert\; \cP_{\cD, \bar{\bbeta}_\cD}) = \norm{\bX_\cS \bbeta_\cS^* - \bX_\cD \bar{\bbeta}_\cD}_2^2/(2 \sigma^2)$. Thus, from the definition of $\bar{\bbeta}_\cD$, it immediately follows that $\bX_\cD \bar{\bbeta}_\cD = \bP_\cD \bX_\cS \bbeta^*_\cS $. Later, we will see that these two notions of distances are equivalent under certain regularity conditions on the link function $b(\cdot)$.

\subsection{Identifiability margin and two complexities}
In this section we will introduce the identifiability margin and the two complexities similar the  case of linear model. We consider the following identifiability margin:
\[
\tilde{\tau}_{*}(s):= \min_{\cD \in \ccA_s} \frac{\Delta_\kl(\cD)}{\abs{\cD \setminus \cS}}.
\]
We assume that $\tilde{\tau}_*(s)>0$ to avoid non-identifiability issue.
Next, we consider the transformed features as follows:
\[
\widetilde{\bX}_\cD = \bLambda_\cD^{1/2} \bX_{\cD},
\]
where $\bLambda_\cD = \mathbf{diag}(b^\dprime(\bx_{1,\cD}^\top \bar{\bbeta}_\cD), \ldots, b^\dprime(\bx_{n,\cD}^\top \bar{\bbeta}_\cD))$. Let $\widetilde{\bP}_\cD$ be orthogonal projection matrices onto the columnspace of $\widetilde{\bX}_\cD$. Let $\widetilde{\bP}_{\cI \mid \cD}$ be the orthogonal projector onto the columnspace of $[\widetilde{\bX}_\cD]_\cI$. 
Now we define the following sets of residualized signals and spurious projections:
\[
\ctT_\cI^{(s)} = \left\{ \frac{\bX_\cD\bar{\bbeta}_\cD - \bX_\cS \bbeta_\cS^*}{\norm{\bX_\cD\bar{\bbeta}_\cD - \bX_\cS \bbeta_\cS^*}_2}  : \cD \in \ccA_\cI\right\},
\]

\[
\ctG_\cI^{(s)} = \left\{ \widetilde{\bP}_\cD - \widetilde{\bP}_{\cI\mid \cD} : \cD \in \ccA_\cI\right\}.
\]
The complexity measures for these two sets are $\ccE_{\ctT_\cI^{(s)}}$ and $\ccE_{\ctG_{\cI^{(s)}}}$ respectively, which are defined in the same way as the complexity measures in Section 3.2 and Section 3.3 of the main paper.


\subsection{Main results}
In this section we will state the main result analogous to the Theorem 1 in the main paper. We begin with some standard assumption necessary for the theoretical analysis for GLM models.
\begin{assumption}[Features and parameters]
\label{assumption: glm- features and params }
We assume the following conditions:
\begin{enumerate}[label = (\alph*)]
    \item \label{item: feature-param bound} There exists positive constants  $x_0$ and $R_0$ such that $\max_{i \in [n]}\norm{\bx_i}_\infty \leq x_0$ and $\norm{\bbeta^*}_1 \leq R_0$.
    \item \label{item: sparse eigenvalue} There exists a constant $\kappa_0>0$ such that 
    \[
    \min_{\cD\subset [p]: \abs{\cD} = s} \lambda_{\min}\left(\bX_\cD^\top \bX_\cD/n\right) \geq \kappa_0.\]

    \item \label{item: third moment condition} There exists constant $M>0$ such that 
    \[
    \max_{\cD\subset [p]: \abs{\cD} = s} \norm{\frac{1}{n} \sum_{i \in [n]} \bx_{i,\cD} \otimes \bx_{i,\cD} \otimes \bx_{i,\cD} }_{\op} \leq M.
    \]

     \item \label{item: minimum KL param bound} There exists a constant $R>0$ such that $\max_{\cD \in \ccA_s} \max_{i \in [n]} \abs{\bx_{i, \cD}^\top \bar{\bbeta}_\cD} \leq x_0 R$.

    \item \label{item: glm noisy features not too correlated} The design matrix $\bX$ enjoys the following property:
    \[
    \min_{\cI \subset \cS} \ccE_{\tilde{\cG}_\cI^{(s)}}^2 > \{\log(ep)\}^{-1}.
    \]
\end{enumerate}
\end{assumption}
Assumption \ref{assumption: glm- features and params }\ref{item: feature-param bound} is very common in high-dimensional literature. Assumption \ref{assumption: glm- features and params }\ref{item: sparse eigenvalue} basically tells that the sparse-eigenvalues of $\bX$ are strictly bounded away from 0. Assumption \ref{assumption: glm- features and params }\ref{item: third moment condition} tells that the third order empirical moment of $\bX_\cD$ is bounded. A stronger version of Assumption \ref{assumption: glm- features and params }\ref{item: minimum KL param bound} is present in \cite{pijyan2020consistent, zheng2020building}, where the authors assume that $\norm{\bar{\bbeta}_\cD}_1$ is bounded uniformly over all $\cD \in \ccA_s$. Finally, Assumption \ref{assumption: glm- features and params }\ref{item: glm noisy features not too correlated} allows diversity among the spurious features. Next, we will assume some technical assumptions on the link function $b(\cdot)$.

\begin{assumption}[$b(\cdot)$ function]
\label{assumption: glm-b function}
    We assume the following conditions on $b(\cdot)$ function:
    \begin{enumerate}[label = (\alph*)]
        \item \label{item: strong convexity} There exists a function $\psi : \bbR_+ \to \bbR_+$ such that for any $\eta \in \bbR $ and any $\omega>0$, $b^\dprime(\eta) \geq \psi(\omega)$ whenever $\abs{\eta} \leq \omega$.

        \item \label{item: b-boundedness} There exists constants $B >0$ and  $\tilde{B}\ge 0$ such that $\norm{b^\dprime}_\infty \leq B$ and $ \norm{b^{\tprime}}_\infty\leq \tilde{B}$.
    \end{enumerate}
\end{assumption}
These assumptions on the link function are pretty common to analyze high-dimensional generalized models. Assumption \ref{assumption: glm-b function}\ref{item: strong convexity} basically assumes that $b(\cdot)$ is strongly convex within a compact neighborhood of 0. It is straightforward to check that this assumption is satisfied by standard GLM setups like linear regression and logistic regression. In particular,  one can choose $\psi(\omega) = 1$ for linear regression, and $\psi(\omega) = (3 + e^\omega)^{-1}$ in the case of logistic regression.
Furthermore, from \eqref{eq: glm pdf} it follows that $\bbE(y) = b^\prime(\bx^\top \bbeta^*)$ and $\var(y) = \phi b^\dprime(\bx^\top \bbeta^*)\geq \phi \psi(\omega)$, whenever $\abs{\bx^\top \bbeta^*} \le \omega$.

Finally, Assumption \ref{assumption: glm-b function}\ref{item: b-boundedness} tells that the second and third derivatives of $b(\cdot)$ are bounded. This guarantees the first convergence rates of the maximum likelihood estimator. Moreover, this assumption guarantees sub-Gaussianity of $y$ as 

\begin{equation}
    \label{eq: MGF glm}
    \begin{aligned}
        &\bbE(\exp\{t (y - b^\prime(\eta)\})\\
        &= e^{-t b^\prime(\eta)}\int_{-\infty}^\infty h(y) \exp\left\{\frac{(\eta + \phi t)y - b(\eta)}{\phi}\right\} \; dy\\
        & = \exp
        \left(\frac{b(\eta+ \phi t) - b(\eta) - t \phi b^\prime(\eta)}{\phi}\right) \int_{-\infty}^\infty h(y) \exp\left\{\frac{(\eta + \phi t)y - b(\eta + \phi t)}{\phi}\right\} \; dy\\
        & = \exp[\phi^{-1}\{b(\eta+ \phi t) - b(\eta) - t \phi b^\prime(\eta)\}] \leq \exp \left( \frac{\phi B t^2}{2}\right).
    \end{aligned}
\end{equation}

Under Assumption \ref{assumption: glm-b function}, we can compare between the margin quantities $\Delta_\kl(\cD)$ and $\Delta_\param(\cD)$. To see this, we first 
focus on $\Delta_{\kl}(\cD)$
. Recall that

\begin{align*}
&\Delta_\kl(\cD)\\
& = \frac{2 }{n} \sum_{i=1}^n \left\{(\bx_{i,\cS}^\top\bbeta^*_\cS) b^\prime(\bx_{i,\cS}^\top\bbeta^*_\cS) - b(\bx_{i,\cS}^\top\bbeta^*_\cS)\right\} - \frac{2}{n} \sum_{i=1}^n \left\{(\bx_{i,\cD}^\top \bar{\bbeta}_\cD) b^\prime(\bx_{i,\cS}^\top\bbeta^*_\cS) - b(\bx_{i,\cD}^\top \bar{\bbeta}_\cD)\right\}\\
& =\frac{2}{n} \sum_{i=1}^n \left\{
b(\bx_{i,\cD}^\top \bar{\bbeta}_\cD) - b(\bx_{i,\cS}^\top \bbeta_\cS^*) - (\bx_{i,\cD}^\top \bar{\bbeta}_\cD - \bx_{i,\cS}^\top \bbeta_\cS^*)  b^\prime(\bx_{i,\cS}^\top \bbeta_\cS^*)
\right\}\\
& = \frac{1}{n}\sum_{i=1}^n b^\dprime\left(\bx_{i,\cS}^\top \bbeta_\cS^* + t(\bx_{i,\cD}^\top \bar{\bbeta}_\cD - \bx_{i,\cS}^\top \bbeta_\cS^*) \right) \{\bx_{i,\cD}^\top \bar{\bbeta}_\cD - \bx_{i,\cS}^\top \bbeta_\cS^*\}^2 ,
\end{align*}
for some $t \in (0,1)$.
Due to Assumption \ref{assumption: glm- features and params }\ref{item: feature-param bound} and Assumption \ref{assumption: glm- features and params }\ref{item: minimum KL param bound}, we get \[\abs{\bx_{i,\cS}^\top \bbeta_\cS^* + t(\bx_{i,\cD}^\top \bar{\bbeta}_\cD - \bx_{i,\cS}^\top \bbeta_\cS^*)}\leq x_0(R_0 + R).\]
Finally, strong convexity and smoothness of $b(\cdot)$ (Assumption \ref{assumption: glm-b function}\ref{item: strong convexity}, \ref{assumption: glm-b function}\ref{item: b-boundedness}), we have
\begin{equation}
\label{eq: delta_kl and delta_para comparison}
 B \Delta_\param(\cD) \geq \Delta_\kl(\cD) \geq \psi(x_0 R_0 + x_0 R) \Delta_\param(\cD).
\end{equation}
This established the equivalence between $\Delta_\kl(\cD)$ and $\Delta_\param(\cD)$. Now, we present the main below.

\begin{theorem}[Sufficiency]
    \label{thm: sufficiency of BSS-GLM}
    Under Assumption 1, there exists a positive constant $C$ depending on $\phi,B, \tilde{B}, x_0, R , R_0, \kappa_0, M$ and the function $\psi(\cdot)$ such that for any $0\leq \eta <1$, whenever the identifiability 
    margin $\widetilde{\tau}_*(s)$ satisfies 
    \begin{equation}
    \label{eq: margin cond-glm}
    \begin{aligned}
&\frac{\widetilde{\tau}_*(s)}{\phi B} \geq\\ & \frac{C}{(1- \eta)^2} \max\left\{\max\left\{
\max_{\cI \subset \cS} \ccE_{\ctT_\cI^{(s)}}^2, \max_{\cI \subset \cS} \ccE_{\ctG_\cI^{(s)}}^2 
\right\}+ \sqrt{\frac{\log(es)\vee \log\log(ep)}{\log(ep)}}, t^{(1)}_{s,n,p}, t^{(2)}_{s,n,p}\right\} \frac{\log(ep)}{n}
\end{aligned}
    \end{equation}
    for a specified $t^{(1)}_{s,n,p} = O(s \{\log s \vee \log \log p\}/\log p)$ and $t^{(2)}_{s,n,p} = O(\frac{s^2(\log n)^2}{n \log p} + \frac{s^{3/2} (\log n)^{3/2}}{\sqrt{n}\log p})$,
    we have 
 \[
\left\{
 \widehat{\cS}: \vert\widehat{\cS}\vert = s, \min_{\cS \in \ccA_s } \cL_{\widehat{\cS}}(\widehat{\bbeta}_{\widehat{\cS}})  \leq \cL_{\cS}(\widehat{\bbeta}_{\cS})+ n \eta \tilde{\tau}_*(s) 
 \right\} = \{\cS\},
 \]
 with probability at least $1 - O(\{s \vee \log p\}^{-1} + n^{-7} s \log p)$.
 In particular, setting $\eta = 0$, we have $\cS = \argmin_{\widehat{\cS}\in \ccA_s}  \cL_{\widehat{\cS}}(\widehat{\bbeta}_{\widehat{\cS}})$ with high probability.
\end{theorem}
The proof of the above theorem is deferred to Section \ref{sec: proof GLM}. The above theorem is the generalization of Theorem 1, and \eqref{eq: margin cond-glm} also involves the two complexities related to the sets of residualized signals and spurious projection operators. However, condition \eqref{eq: margin cond-glm} also involves two extra terms $t^{(1)}_{s,n,p}$ and $t^{(2)}_{s,n,p}$, the exact forms of which can be found in Section \ref{sec: proof GLM}. It can be shown that both of these terms are exactly 0 for linear models as $\psi \equiv 1, B = 1$ and $\tilde{B}=0$. 
\begin{remark}
If $p = \Omega(e^{c_0 n})$ for some universal constant $c_0>0$ and $s (\log n)/n \to 0$ as $n \to \infty$, then both  $t^{(1)}_{s,n,p}$ and $t^{(2)}_{s,n,p}$ are negligible compared to the complexity term in \eqref{eq: margin cond-glm}. Hence, in this case, we witness roughly a similar phenomenon involving the two complexities as in the linear model.
\end{remark}

\section{Proof of main results under GLM model}
\label{sec: proof GLM}
Let $\cD \in \ccA_{s}$ such that $\cS \cap \cD = \cI$.
\subsubsection*{Strong convexity}
We will start by showing the strong convexity of $\cL_\cD(\Tilde{\bbeta}; \{\bx_i, y_i\}_{i \in [n]})$. For ease of presentation we will just write $\cL_\cD(\Tilde{\bbeta})$ instead of $\cL_\cD(\Tilde{\bbeta}; \{\bx_i, y_i\}_{i \in [n]})$. Given any $r\in(0,R_0 \wedge R]$ and $\bDelta \in \bbB_1(\mathbf{0}, r)$ define the function

\begin{align*}
\delta \cL_\cD(\bar{\bbeta}_\cD + \bDelta; \bar{\bbeta}_\cD) & := \cL_\cD(\bar{\bbeta}_\cD + \bDelta) - \cL_\cD(\bar{\bbeta}_\cD ) - \nabla \cL_\cD(\bar{\bbeta}_\cD)^\top \bDelta\\
& = \frac{1}{2}\bDelta^\top \nabla^2\cL_\cD(\bar{\bbeta}_\cD + t \bDelta) \bDelta \quad \text{(for some $t \in (0,1)$)}\\
& = \frac{1}{n} \sum_{i=1}^n b^\dprime (\bx_{i,\cD}^\top (\bar{\bbeta}_\cD  + t \bDelta)) (\bx_{i,\cD}^\top \bDelta)^2\\
& \geq  \psi(x_0 R + x_0 r) \kappa_0 \norm{\bDelta}_2^2 \quad \text{(Using Assumption \ref{assumption: glm- features and params }\ref{item: feature-param bound}, \ref{assumption: glm- features and params }\ref{item: sparse eigenvalue}, \ref{assumption: glm- features and params }\ref{item: minimum KL param bound})}\\
& \geq \psi(x_0 R + x_0 R_0)\kappa_0 \norm{\bDelta}_2^2 
\end{align*}

\subsubsection*{Rate of convergence}
Construct an intermediate estimator $\widehat{\bbeta}_{\cD, \alpha} = \bar{\bbeta}_\cD + \alpha (\widehat{\bbeta}_\cD - \bar{\bbeta}_\cD)$ where 
$$\alpha = \min\left\{1, \frac{r}{\Vert\widehat{\bbeta}_\cD - \bar{\bbeta}_\cD\Vert_2}\right\},$$
where $r$ will be chosen later.

Write $\widehat{\bbeta}_{\cD, \alpha} - \bar{\bbeta}_\cD$ as $\bDelta_\alpha$ and note that 
\[
\psi(x_0 R + x_0 R_0) \norm{\bDelta_\alpha}_2^2 \leq \delta\cL_\cD(\widehat{\bbeta}_{\cD, \alpha}, \bar{\bbeta}_\cD) \leq - \nabla \cL_\cD(\bar{\bbeta}_\cD)^\top \bDelta_\alpha \leq \norm{\nabla\cL_\cD(\bar{\bbeta}_\cD)}_2
\norm{\bDelta_\alpha}_2 . 
\]
Hence we have 
\begin{equation}
    \label{eq: alpha-est bound}
    \norm{\bDelta_\alpha}_2 \leq \frac{\norm{\nabla\cL_\cD(\bar{\bbeta}_\cD)}_2}{\psi(x_0 R + x_0 R_0)} \leq \frac{\sqrt{s} \norm{\nabla\cL_\cD(\bar{\bbeta}_\cD)}_\infty}{\psi(x_0 R + x_0 R_0)}.
\end{equation}
Now, note that
\begin{align*}
    \nabla\cL_\cD(\bar{\bbeta}_\cD) &:= -\frac{2}{n} \sum_{i \in [n]}\{ y_i  - b^\prime(\bx_{i,\cD}^\top \bar{\bbeta}_{\cD}) \} \bx_{i, \cD}\\
    & = -\frac{2}{n} \sum_{i \in [n]}\{ y_i  - b^\prime(\bx_{i,\cS}^\top \bbeta_{\cS}^*) \} \bx_{i, \cD}
     -\frac{2}{n} \underbrace{\sum_{i \in [n]}\{ b^\prime(\bx_{i,\cS}^\top \bbeta_{\cS}^*)  - b^\prime(\bx_{i,\cD}^\top \bar{\bbeta}_{\cD}) \} \bx_{i, \cD}}_{=0}\\
     & = -\frac{2 (\phi B)^{1/2}}{n} \sum_{i \in [n]} \underbrace{\frac{\{ y_i  - b^\prime(\bx_{i,\cS}^\top \bbeta_{\cS}^*) \}}{(\phi B)^{1/2}}}_{:=\epsilon_i} \bx_{i, \cD}
\end{align*}
Note that $\bbE\{\exp(\lambda \epsilon_i [\bx_{i,\cD}]_j) \leq \exp(\lambda^2 x_0^2/2)\}$, i.e., $\epsilon_i [\bx_{i,\cD}]_j$ is sub-Gaussian with parameter $x_0$. Hence, by an application of union bound and Hoeffding's inequality we have
\begin{equation}
    \label{eq: bound on gradient}
    \pr \left( \norm{\nabla \cL_\cD(\bar{\bbeta}_\cD)}_\infty \geq 2 t x_0 (\phi B)^{1/2}  \right) \leq 2 s \exp\left( -\frac{n t^2}{2}\right). 
\end{equation}

Setting $t = 4 (\log n/n)^{1/2}$ in \eqref{eq: bound on gradient} we get 
\begin{equation}
    \label{eq: bound on gradient 2}
    \pr \left( \norm{\nabla \cL_\cD(\bar{\bbeta}_\cD)}_\infty \geq 8  x_0 (\phi B)^{1/2} \sqrt{\frac{\log n}{n}} \right) \leq \frac{2}{n^7}. 
\end{equation}
Using the above fact and \eqref{eq: alpha-est bound} we finally get that with probability at least $1-2n^{-7}$ the following holds:
\[
\norm{\bDelta_\alpha}_2 \leq \frac{8  x_0 (\phi B)^{1/2}}{\psi(x_0 R + x_0 R_0)}  \sqrt{\frac{s\log n}{n}}.
\]
Now we set $r = \frac{9  x_0 (\phi B)^{1/2}}{\psi(x_0 R)}  \sqrt{\frac{s \log n}{n}} $. Hence, we have $\norm{\bDelta_\alpha}_2 <r$, i.e., $\norm{\bDelta}_2 <r$. This shows that 
\begin{equation}
\label{eq: estimation bound for beta_D}
    \pr \left( \max_{\cD\in \ccA_s} \norm{\widehat{\bbeta}_\cD - \bar{\bbeta}_\cD}_2 > \frac{9  x_0 (\phi B)^{1/2}}{\psi(x_0 R + x_0 R_0)}  \sqrt{\frac{s \log n}{n}} \right) \leq \frac{4 s \log p}{n^7}. 
\end{equation}

By a similar argument, it can be shown that
\begin{equation}
    \label{eq: estimation bound for beta_S}
    \pr \left(\norm{\widehat{\bbeta}_\cS - {\bbeta}_\cS^*}_2 > \frac{9  x_0 (\phi B)^{1/2}}{\psi( x_0 R + x_0 R_0 )}  \sqrt{\frac{s \log n}{n}} \right) \leq \frac{2}{n^7}. 
\end{equation}

\subsection*{Expansion of likelihood estimate}
Now that we have determined the rate of estimation, we can now write $\widehat{\bbeta}_\cD$ in terms of $\bar{\bbeta}_\cD$. To see this, note that 
\begin{equation*}
    \begin{aligned}
        \mathbf0 = \nabla\cL_\cD(\widehat{\bbeta}_\cD) = \nabla\cL_\cD(\bar{\bbeta}_\cD) +  \nabla^2 \cL_\cD(\bar{\bbeta}_\cD) (\widehat{\bbeta}_\cD - \bar{\bbeta}_\cD) + \bR_\cD (\widehat{\bbeta}_\cD - \bar{\bbeta}_\cD)^{\otimes 2},
    \end{aligned}
\end{equation*}
where $\bR_\cD = (1/2)  \nabla^3\cL_{\cD}(\bar{\bbeta}_{\cD} + t_\cD(\widehat{\bbeta}_\cD - \bar{\bbeta}_\cD))$ for some $t_\cD \in (0,1)$.
Thus, we have
\begin{equation}
    \label{eq: beta_D hat expansion}
    \widehat{\bbeta}_\cD = \bar{\bbeta}_\cD -  [\nabla^2 \cL_\cD(\bar{\bbeta}_\cD)]^{-1} \left( \nabla \cL_\cD(\bar{\bbeta}_\cD) + \bR_\cD (\widehat{\bbeta}_\cD - \bar{\bbeta}_\cD)^{\otimes 2} \right)
\end{equation}
\subsubsection*{Higher order Taylor's expansion of loss function}
Now we are ready to analyze the loss functions. We do so by expanding the Taylor series of the loss function. Write $\widehat{\bbeta}_\cD - \bar{\bbeta}_\cD$ ad $\widehat{\bDelta}_\cD$. Then, using \eqref{eq: beta_D hat expansion} we have

\begin{align*}
&\cL_\cD(\widehat{\bbeta}_\cD) \\
&= \cL_\cD(\bar{\bbeta}_\cD) + \nabla \cL_\cD(\bar{\bbeta}_\cD)^\top \widehat{\bDelta}_\cD + \frac{1}{2} \widehat{\bDelta}_\cD^\top \nabla^2 \cL_\cD(\bar{\bbeta}_\cD) \widehat{\bDelta}_\cD + (1/3) \widehat{\bDelta}_\cD^\top \tilde{\bR}_\cD (\widehat{\bDelta}_\cD \otimes \widehat{\bDelta}_\cD)\\
& = \cL_\cD(\bar{\bbeta}_\cD) - \nabla \cL_\cD(\bar{\bbeta}_\cD)^\top [\nabla^2 \cL_\cD(\bar{\bbeta}_\cD)]^{-1} \left(\nabla \cL_\cD(\bar{\bbeta}_\cD) + \tilde{\bR}_\cD (\widehat{\bbeta}_\cD  - \bar{\bbeta}_\cD)^{\otimes 2} \right)\\
& \quad + (1/2)\left(\nabla \cL_\cD(\bar{\bbeta}_\cD) + \tilde{\bR}_\cD (\widehat{\bbeta}_\cD  - \bar{\bbeta}_\cD)^{\otimes 2} \right)^\top [\nabla^2\cL_\cD(\bar{\bbeta}_\cD)]^{-1}  \left(\nabla \cL_\cD(\bar{\bbeta}_\cD) + \tilde{\bR}_\cD (\widehat{\bbeta}_\cD  - \bar{\bbeta}_\cD)^{\otimes 2} \right)\\
& \quad + (1/3) \widehat{\bDelta}_\cD^\top \tilde{\bR}_\cD (\widehat{\bDelta}_\cD \otimes \widehat{\bDelta}_\cD)\\
& = \cL_\cD(\bar{\bbeta}_\cD) - \frac{1}{2}\nabla \cL_\cD(\bar{\bbeta}_\cD)^\top [\nabla^2 \cL_\cD(\bar{\bbeta}_\cD)]^{-1} \nabla \cL_\cD(\bar{\bbeta}_\cD) + \frac{1}{2} (\tilde{\bR}_\cD \widehat{\bDelta}_\cD^{\otimes 2})^\top [\nabla^2 \cL_\cD(\bar{\bbeta}_\cD)]^{-1} (\tilde{\bR}_\cD \widehat{\bDelta}_\cD^{\otimes 2})\\
& \quad + (1/3) \widehat{\bDelta}_\cD^\top \tilde{\bR}_\cD (\widehat{\bDelta}_\cD \otimes \widehat{\bDelta}_\cD)\\
& = \frac{2}{n}\sum_{i \in [n]} \{-y_i (\bx_{i,\cD}^\top \bar{\bbeta}_\cD) + b(\bx_{i,\cD}^\top \bar{\bbeta}_\cD)\} - \frac{1}{n}(\by - \brho(\bX_\cD \bar{\bbeta}_\cD))^\top \bX_\cD (\widetilde{\bX}_\cD^\top \widetilde{\bX}_{\cD})^{-1} \bX_\cD^\top (\by- \brho(\bX_\cD \bar{\bbeta}_\cD))\\
& \quad + \frac{1}{2} (\tilde{\bR}_\cD \widehat{\bDelta}_\cD^{\otimes 2})^\top [\nabla^2 \cL_\cD(\bar{\bbeta}_\cD)]^{-1} (\tilde{\bR}_\cD \widehat{\bDelta}_\cD^{\otimes 2}) + (1/3) \widehat{\bDelta}_\cD^\top \tilde{\bR}_\cD (\widehat{\bDelta}_\cD \otimes \widehat{\bDelta}_\cD)\\
& = - \frac{2}{n} \brho(\bX_\cS \bbeta_\cS^*)^\top \bX_\cD \bar{\bbeta}_\cD + \frac{2}{n}\sum_{i \in [n]} b(\bx_{i,\cD}^\top \bar{\bbeta}_\cD)  -\frac{2}{n}(\by - \brho(\bX_\cS \bbeta_\cS^*))^\top \bX_\cD \bar{\bbeta}_\cD\\
& \quad - \frac{1}{n}(\by - \brho(\bX_\cS \bbeta_\cS^*))^\top \bX_\cD (\widetilde{\bX}_\cD^\top \widetilde{\bX}_{\cD})^{-1} \bX_\cD^\top (\by- \brho(\bX_\cS \bbeta_\cS^*)) \\
& \quad + \frac{1}{2} (\tilde{\bR}_\cD \widehat{\bDelta}_\cD^{\otimes 2})^\top [\nabla^2 \cL_\cD(\bar{\bbeta}_\cD)]^{-1} (\tilde{\bR}_\cD \widehat{\bDelta}_\cD^{\otimes 2}) + (1/3) \widehat{\bDelta}_\cD^\top \tilde{\bR}_\cD (\widehat{\bDelta}_\cD \otimes \widehat{\bDelta}_\cD),
\end{align*}
where $\tilde{\bR}_\cD = (1/2)  \nabla^3\cL_{\cD}(\bar{\bbeta}_{\cD} + \tilde{t}_\cD(\widehat{\bbeta}_\cD - \bar{\bbeta}_\cD))$ for some $\tilde{t}_\cD \in (0,1)$.

Thus, we have the following:
\begin{align*}
& \cL_\cD(\widehat{\bbeta}_\cD) - \cL_\cS(\widehat{\bbeta}_\cS)\\
& = \Delta_\kl(\cD) - \underbrace{\frac{2}{n}(\by - \brho(\bX_\cS \bbeta^*_\cS))^\top (\bX_\cD\bar{\bbeta}_\cD - \bX_\cS \bbeta_\cS^*)}_{\text{linear term}}\\
& \quad - \underbrace{\frac{1}{n} (\by - \brho(\bX_\cS \bbeta_\cS^*))^\top \left\{\bX_\cD (\widetilde{\bX}_\cD^\top \widetilde{\bX}_{\cD})^{-1} \bX_\cD^\top - \bX_\cS (\widetilde{\bX}_\cS^\top \widetilde{\bX}_{\cS})^{-1} \bX_\cS^\top \right\} (\by- \brho(\bX_\cS \bbeta_\cS^*))}_{\text{quadratic term}}\\
& \quad + \frac{1}{2} (\tilde{\bR}_\cD \widehat{\bDelta}_\cD^{\otimes 2})^\top [\nabla^2 \cL_\cD(\bar{\bbeta}_\cD)]^{-1} (\tilde{\bR}_\cD \widehat{\bDelta}_\cD^{\otimes 2}) + (1/3)\widehat{\bDelta}_\cD^\top \tilde{\bR}_\cD (\widehat{\bDelta}_\cD \otimes \widehat{\bDelta}_\cD) \\
& \quad  - \frac{1}{2} (\tilde{\bR}_\cS \widehat{\bDelta}_\cS^{\otimes 2})^\top [\nabla^2 \cL_\cS(\bbeta_\cS^*)]^{-1} (\tilde{\bR}_\cS \widehat{\bDelta}_\cS^{\otimes 2}) - (1/3) \widehat{\bDelta}_\cD^\top \tilde{\bR}_\cD (\widehat{\bDelta}_\cD \otimes \widehat{\bDelta}_\cD). 
\end{align*}
Write $\psi_* = \min\{\psi(x_0 R), \psi(x_0 R_0)\}$. By definition of $\widetilde{\bX}_\cD$, we have 
\[
\widetilde{\bX}_\cD^\top \widetilde{\bX}_\cD = \bX_\cD^\top \bLambda_\cD\bX_\cD \succeq \psi_* \bX_\cD^\top \bX_\cD, \quad \text{where}\quad \bLambda_\cD = \mathbf{diag}(b^\dprime(\bx_{1,\cD}^\top \bar{\bbeta}_\cD), \ldots, b^\dprime(\bx_{n,\cD}^\top \bar{\bbeta}_\cD)).
\]
Hence, we have $(\widetilde{\bX}_\cD^\top \widetilde{\bX}_\cD)^{-1} \preceq \frac{1}{\psi_* } (\bX_\cD^\top \bX_\cD)^{-1} \Rightarrow \bX_\cD (\widetilde{\bX}_\cD^\top \widetilde{\bX}_{\cD})^{-1}\bX_\cD ^\top\preceq  \frac{1}{\psi_*} \bP_\cD$. Similarly, $\bX_\cS (\widetilde{\bX}_\cS^\top \widetilde{\bX}_{\cS})^{-1}\bX_\cS \succeq \frac{1}{B}\bP_\cS$. Also, recall that $\widetilde{\bP}_\cD = \widetilde{\bX}_\cD(\widetilde{\bX}_\cD^\top \widetilde{\bX}_{\cD})^{-1}\widetilde{\bX}_\cD^\top$ for any $\cD \in \ccA_s \cup \{\cS\}$. Let $\widetilde{\bP}_{\cI\mid \cD}$ be the orthogonal projector onto the $\col([\widetilde{\bX}_\cD]_\cI)$.
Using these facts, for any $\eta \in [0,1)$, the difference between the two losses can be lower bounded as follows:

\begin{equation}
    \label{eq: loss-diff lower bound}
    \begin{aligned}
    & \cL_\cD(\widehat{\bbeta}_\cD) - \cL_\cS(\widehat{\bbeta}_\cS)\\
    & \geq \eta \Delta_\kl(\cD)\\
    & \quad + 2^{-1}(1-\eta)\Delta_\kl(\cD) - \frac{2}{n}(\by - \brho(\bX_\cS \bbeta^*_\cS))^\top (\bX_\cD\bar{\bbeta}_\cD - \bX_\cS \bbeta_\cS^*)\\
    & \quad + 2^{-1}(1-\eta) - \frac{1}{n} (\by - \brho(\bX_\cS \bbeta^*_\cS))^\top \bLambda_{\cD}^{-1/2} (\widetilde{\bP}_\cD - \widetilde{\bP}_{\cI \mid \cD}) \bLambda_{\cD}^{-1/2}(\by - \brho(\bX_\cS \bbeta^*_\cS))\\
    & \quad + \underbrace{\frac{1}{n} (\by - \brho(\bX_\cS \bbeta^*_\cS))^\top \bLambda_\cS^{-1/2}(\widetilde{\bP}_\cS - \widetilde{\bP}_{\cI \mid \cS})\bLambda_\cS^{-1/2} (\by - \brho(\bX_\cS \bbeta^*_\cS))}_{\geq 0}\\
    & \quad - \frac{1}{n} (\by - \brho(\bX_\cS \bbeta^*_\cS))^\top \bLambda_\cD^{-1/2} \widetilde{\bP}_{\cI\mid \cD} \bLambda_\cD^{-1/2}(\by - \brho(\bX_\cS \bbeta^*_\cS))\\
    & \quad + \frac{1}{n} (\by - \brho(\bX_\cS \bbeta^*_\cS))^\top \bLambda_\cS^{-1/2} \widetilde{\bP}_{\cI \mid \cS} \bLambda_\cS^{-1/2}(\by - \brho(\bX_\cS \bbeta^*_\cS))\\
    & \quad - \frac{\Tilde{B}^2 M^2}{4 \kappa_0 \psi_*} \norm{\widehat{\bDelta}_\cD}_2^4 - \frac{\Tilde{B}^2 M^2}{4 \kappa_0 \psi_*} \norm{\widehat{\bDelta}_\cS}_2^4 - \frac{\Tilde{B} M}{6} \norm{\widehat{\bDelta}_\cD}_2^3
    - \frac{\Tilde{B} M}{6} \norm{\widehat{\bDelta}_\cS}_2^3.   
    \end{aligned}
\end{equation}

Now recall that
\[
\ctT_\cI^{(s)} = \left\{ \frac{\bX_\cD\bar{\bbeta}_\cD - \bX_\cS \bbeta_\cS^*}{\norm{\bX_\cD\bar{\bbeta}_\cD - \bX_\cS \bbeta_\cS^*}_2}  : \cD \in \ccA_\cI\right\}
\]

\[
\ctG_\cI^{(s)} = \left\{ \widetilde{\bP}_\cD - \widetilde{\bP}_{\cI\mid \cD} : \cD \in \ccA_\cI\right\}
\]


Now we will handle the linear and the quadratic terms separately. We assume that $\abs{\cI} = s-k$, where $1\leq k\leq s$. 

\subsection*{Analysis of likelihood lower bound}
\subsubsection*{Analysis of linear term}
To analyze the linear term we will use the deviation bound for the supremum of the sub-Gaussian process. In particular, we will use Theorem 5.36 of \cite{wainwright2019high}. First, note that $\diam(\ctT_\cI^{(s)}) \leq \sqrt{2}$ and recall $\bepsilon = (\epsilon_1, \ldots, \epsilon_n)^\top$, where $\epsilon_i = \frac{\{ y_i  - b^\prime(\bx_{i,\cS}^\top \bbeta_{\cS}^*) \}}{(\phi B)^{1/2}}$. Due to 1-sub-Gaussianity, we have $\max_{i \in [n]} \var(\epsilon_i) \leq \sigma^2_\epsilon$ for some universal constant $\sigma_\epsilon>0$. Also, recall that 
\[
\widehat{\bsr}_\cD = \frac{\bX_\cD\bar{\bbeta}_\cD - \bX_\cS \bbeta_\cS^*}{\norm{\bX_\cD\bar{\bbeta}_\cD - \bX_\cS \bbeta_\cS^*}_2}.
\]
Thus, using the aforementioned theorem we get
\begin{equation}
\label{eq: glm lin proc bound}
\begin{aligned}
&\pr \left\{\max_{\cD \in \ccA_\cI} \widehat{\bsr}_\cD^\top\bepsilon \geq A_1 ( \ccE_{\ctT_\cI^{(s)}} \sqrt{k \log(ep)} + \sqrt{2 k \{\log(es)\vee\log \log (ep)\}}  )\right\}\\
& \leq 3 \{(es) \vee \log(ep)\}^{-2k},
\end{aligned}
\end{equation}
for some universal constant $A_1>0$.

\subsubsection*{Analysis of quadratic terms}

Now, we focus on the quadratic terms.
Define the random vector $\bxi_\cD := (\xi_{1,\cD}, \ldots, \xi_{n,\cD})$, where 
\[
\xi_i = \frac{\{y_i - b^\prime(\bx_{i,\cS}^\top \bbeta_\cS^*)\} }{(\phi B \psi_*^{-1})^{1/2} \sqrt{b^\dprime(\bx_{i,\cD}^\top \bar{\bbeta}_\cD)}}, \quad i \in [n].
\]
Note that $\{\xi_{i,\cD}\}_{i \in [n]}$ are independent 1-sub-Gaussian.
We will study the random quantity $Q_{\ccA_\cI}:=\max_{\cD \in \ccA_{\cI}} \bxi_\cD^\top (\widetilde{\bP}_\cD - \widetilde{\bP}_{\cI \mid \cD})\bxi_\cD$. Let us assume that $\max_{i \in [n]}\var(\xi_{i,\cD}) = \sigma_\cD^2$. First, we note that $\widetilde{\bP}_\cD - \widetilde{\bP}_{\cI\mid \cD}$ is a projection matrix of rank $k$ and hence it is idempotent. Also note that $\bbE\left\{ \bxi_\cD^\top (\widetilde{\bP}_\cD - \widetilde{\bP}_{\cI \mid \cD})\bxi_\cD \right\} = \tr\left\{ (\widetilde{\bP}_\cD - \widetilde{\bP}_{\cI \mid \cD}) \bbE(\bxi_\cD \bxi_\cD^\top)\right\} = k \sigma_\cD^2 \leq k \sigma^2_0$, where $\sigma_0^2$ is a universal constant. Now, to bound $Q_{\ccA_s}$ we will use Theorem \ref{thm: order-2 chaos deviation bound}. By the properties of projection matrices, we have $d_\op (\ctG_\cI^{(s)}) =1 $ and $d_F(\ctG_\cI^{(s)}) = \sqrt{k}$. Hence, equipped with Assumption \ref{assumption: glm- features and params }\ref{item: glm noisy features not too correlated}, the quantities $M, V$ and $U$ (defined in Theorem \ref{thm: order-2 chaos deviation bound}) has the following properties:
\[
M \leq 2 \ccE_{\ctG_\cI^{(s)}}^2 k \log(ep), \quad V \leq 2 \sqrt{k \log (ep)}, \quad \text{and} \quad U =1.
\]
Due to Theorem \ref{thm: order-2 chaos deviation bound}, there exists a universal positive constant $A_3$, such that  for $$t = A_3 k \sqrt{\log(ep)\{\log (es) \vee \log \log (ep)\}},$$ we get
\begin{align*}
&\pr\left( C_{\ccA_\cI}(\bxi_\cD) \geq  A_2 \ccE_{\ctG_\cI^{(s)}}^2 k \log(ep) + A_ 3 k \sqrt{\log(ep)\{\log (es) \vee \log \log (ep)\}} \right) \\
& \leq \{ (es)\vee \log (ep)\}^{-2 k},
\end{align*}
for a universal positive constant $A_2$.
As $\max_{\cD \in \ccA_{\cI}} \bbE\{\bxi^\top (\widetilde{\bP}_\cD - \widetilde{\bP}_{\cI \mid \cD})\bxi\} \leq  k \sigma_0^2 \leq k \sigma_0^2 \ccE_{\ctG_\cI^{(s)}}^2 \log (ep)$, we finally have 
\begin{equation}
    \label{eq: glm quad proc bound}
    \begin{aligned}
    &\pr \left( Q_{\ccA_s} \leq A_4  \ccE_{\ctG_\cI^{(s)}}^2 k \log(ep) + A_ 3 k \sqrt{\log(ep)\{\log (es) \vee \log \log (ep)\}}\right)\\
    & \leq \{ (es) \vee \log (ep)\}^{- 2k},
    \end{aligned}
\end{equation}
where $A_4$ is a universal positive constant.

Next, by construction, we have 

\begin{align*}
& n^{-1} (\by - \brho(\bX_\cS \bbeta^*_\cS))^\top \bLambda_\cD^{-1/2} \widetilde{\bP}_{\cI\mid \cD} \bLambda_\cD^{-1/2}(\by - \brho(\bX_\cS \bbeta^*_\cS))\\
& = n^{-1} (\by - \brho(\bX_\cS \bbeta^*_\cS))^\top \bX_{\cI}(\widetilde{\bX}_\cI^\top \widetilde{\bX}_\cI)^{-1}\bX_{\cI}^\top (\by - \brho(\bX_\cS \bbeta^*_\cS))\\
& \preceq n^{-1} \psi_*^{-1} (\by - \brho(\bX_\cS \bbeta^*_\cS))^\top \bP_\cI (\by - \brho(\bX_\cS \bbeta^*_\cS)).
\end{align*}

Similarly, 
\[
n^{-1} (\by - \brho(\bX_\cS \bbeta^*_\cS))^\top \bLambda_\cD^{-1/2} \widetilde{\bP}_{\cI\mid \cD} \bLambda_\cD^{-1/2}(\by - \brho(\bX_\cS \bbeta^*_\cS)) \succeq n^{-1} B^{-1} (\by - \brho(\bX_\cS \bbeta^*_\cS))^\top \bP_\cI (\by - \brho(\bX_\cS \bbeta^*_\cS)).
\]

By Theorem 1 of \cite{rudelson2013hanson}, there exists a universal constant $A_5>0$ such that 
\[
\pr \left\{
\abs{\bepsilon^\top \bP_{\cI} \bepsilon - (s-k)\sigma_\epsilon^2}\geq t
\right\} \leq 2 \exp \left[ - A_5 \min \left\{\frac{t^2}{s-k}, t \right\} \right].
\]
For $t = 2 A_5^{-1} s \{\log(es)\vee\log \log (ep)\}$ and a universal constant $A_6>0$,  we get
\begin{equation}
    \label{eq: glm single quad deviation bound}
    \pr \left[ \bepsilon^\top \bP_\cI \bepsilon \geq A_6 s \max\{\log(es), \log \log (ep)\}\right] \leq \{\log(ep)\}^{-2s}\quad \text{for all $\cD\in \ccA_s \cup \{\cS\}$}.
\end{equation}

\subsubsection*{Final 0-1 error bound}
Define the event
\[
\Omega  = \left\{\max_{\cD\in \ccA_s \cup \{\cS\}} \norm{\widehat{\bbeta}_\cD - \bar{\bbeta}_\cD}_2 \leq \frac{9  x_0 (\phi B)^{1/2}}{\psi(x_0 R + x_0 R_0)}  \sqrt{\frac{s \log n}{n}} \right\}
\].

By \eqref{eq: estimation bound for beta_D} and \eqref{eq: estimation bound for beta_S} we have $\pr (\Omega^c) \lesssim (s \log p)/n^7$. Also, let $\cE$ be an event inside of which the assertion of the theorem holds. It follows that 
\begin{align*}
    \pr (\cE^c) & = \pr (\cE^c\cap \Omega) + \pr (\cE^c \cap \Omega^c)\\
    & \lesssim \left[\sum_{k=1}^s \sum_{\cI \subset \cS:\abs{\cI} = s-k} \pr \left(\min_{\cD \in \ccA_\cI} \cL_{\cD}(\widehat{\bbeta}_\cD) - \cL_\cS(\widehat{\bbeta}_\cS) < \eta \widetilde{\tau}_*(s) \right)\right] + \frac{s \log p}{n^7}
\end{align*}
 Now, assume the following:
 \begin{align*}
     & \widetilde{\tau}_\glm(s)  :=  \min_{\cD\neq \cS, \abs{\cD}  = s} \min \left\{ \frac{\Delta_{\kl}^2(\cD)}{\Delta_\param(\cD) \abs{\cD \setminus \cS} }, \frac{\Delta_\kl(\cD)}{\abs{\cD \setminus \cS}}\right\}\\
     & \gtrsim \frac{(1\vee \psi_*^{-1})}{(1-\eta)^2}\max\left\{\mathsf{Comp}_1, \mathsf{Comp} _2, {\color{black} t_{s,n,p}^{(1)} } , t^{(2)}_{n,s,p}  \right\} \frac{(\phi B)\log (ep)}{n},
 \end{align*}
where 
\[
\mathsf{Comp}_1 = \left( \ccE_{\ctT_\cI^{(s)}} + \sqrt{\frac{\log(es) \vee \log \log (ep)}{\log(ep)}}\right)^2,
\]

\[
\mathsf{Comp}_2 = \left( \ccE_{\ctG_\cI^{(s)}}^2 + \sqrt{\frac{\log(es) \vee \log \log (ep)}{\log(ep)}}\right),
\]

 \[
 {\color{black} t_{s,n,p}^{(1)}:= (\psi_*^{-1} - B^{-1})\frac{s\{\log(es) \vee \log \log (ep)\}}{\log(ep)} },
 \]

 \[
 t^{(2)}_{s,n,p}:= \left(\frac{\Tilde{B}^2 M^2 x_0^4 \phi^2 B^2}{\kappa_0 \psi_* \psi_{**}^4} \right)\frac{ s^2 (\log n)^2}{n \log p} + \left(\frac{\Tilde{B} M x_0^3 \phi^{3/2}B^{3/2}}{6}\right)\frac{s^{3/2} (\log n)^{3/2}}{\sqrt{n} \log p},
 \]
 where $\psi_{**} = \psi(x_0 R + x_0 R_0)$.
Now, recall the property \eqref{eq: delta_kl and delta_para comparison} of $\Delta_\kl(\cD)$. Thus, for the aforementioned condition to hold for $\widetilde{\tau}_\glm(s)$, it is sufficient to have 
\begin{align*}
\widetilde{\tau}_*(s):= &\min_{\cD\neq \cS, \abs{\cD}  = s} \frac{\Delta_\kl(\cD)}{ \abs{\cD \setminus \cS} }\\
& \gtrsim \frac{(\phi B)(1 \vee \psi_*)^{-1}}{(\psi_{**} \wedge 1)(1-\eta)^2} \max\left\{\mathsf{Comp}_1, \mathsf{Comp} _2, {\color{black} t_{s,n,p}^{(1)} } , t^{(2)}_{n,s,p}  \right\} \frac{\log (ep)}{n}
\end{align*}
 Under the above inequality and due to \eqref{eq: glm lin proc bound}, \eqref{eq: glm quad proc bound}, and \eqref{eq: glm single quad deviation bound}, we can finally conclude

\begin{align*}
    \pr(\cE^c) & \lesssim \sum_{k=1}^s \binom{s}{k} \{(es) \vee \log(ep)\}^{-2 k} + \frac{s \log p}{n^7}\\
    & \lesssim \frac{1}{(s \vee \log p)} + \frac{s \log p}{n^7}.
\end{align*}



\section{Quadratic chaos process}
\label{appendix: quadratic chaos}
Let $\cA$ be a set of $m \times n$ matrices and $\bxi$ be a 1-sub-Gaussian random vector. The random variable of interest is 
\[
C_\cA(\bxi) := \sup_{A \in \cA} \abs{\norm{A \bxi}_2^2 - \bbE \norm{A \bxi}_2^2}.
\]
This quantity is studied by \cite{krahmer2014suprema} and \cite{banerjee2019random}. In the literature of empirical process, this is known as order-2 sub-Gaussian chaos. Before we present the main result for $C_\cA(\bxi)$, we introduce some useful definitions.

\begin{definition}
\label{def: talagrand complexity}
For a metric space $(T, d)$, an admissible sequence of $T$ is a collection of subsets of $T$, $\{T_r: r\geq 0\}$, such that for every $r\geq 0$, $\abs{T_r}\leq 2^{2^r}$ and $\abs{T_0} = 1$. For $\alpha\geq 1$, define the $\gamma_\alpha$ functional by 
\[
\gamma_\alpha(T, d):= \inf \sup_{t \in T} \sum_{r=0}^\infty 2^{r/\alpha} d(t, T_r).
\]
\end{definition}
The $\gamma_\alpha$ functional can be bounded in terms of the covering numbers $\cN( T, d, \epsilon)$ by the well-known Dudley's integral (see \cite{talagrand2005generic}). A more specific formulation for the 
$\gamma_2$ functional of a set of matrices $\cA$ endowed with the operator norm, the scenario which we will focus on in this article, is
\[
\gamma_2(\cA, \norm{\cdot}_\op) \leq \int_0^\infty \sqrt{\log \cN( \cA, \norm{\cdot}_\op, \epsilon)}\; d\epsilon.
\]
We also define the two quantities $d_\op(\cA) = \sup_{A\in \cA} \norm{A}_\op$ and $d_F(\cA):= \sup_{A\in \cA} \sqrt{\tr(A^\top A)}$.

Now, we present the main deviation bound for $C_\cA(\bxi)$. 

\begin{theorem}[Theorem 1 of \cite{banerjee2019random}]
\label{thm: order-2 chaos deviation bound}
Let $\cA$ be a set of $m \times n$ matrices and $\bxi := (\xi_1, \ldots, \xi_n)^\top$ be a random vector with independent 1-sub-Gaussian entries. Let 

\begin{align*}
& M = \gamma_2(\cA, \norm{\cdot})_\op \left\{\gamma_2(\cA, \norm{\cdot})_\op + d_F(\cA)\right\},\\
& V = d_\op(\cA) \left\{\gamma_2(\cA, \norm{\cdot})_\op + d_F(\cA)\right\},\\
& U = d_\op(\cA).
\end{align*}
Then, for $t>0$,
\[
\pr \left(C_\cA(\bxi) \geq c_1 M + t\right) \leq 2 \exp \left( -c_2 \min \left\{
\frac{t^2}{V^2}, \frac{t}{U}
\right\}\right),
\]
where $c_1,c_2$ are universal positive constants.
\end{theorem}

\section{Technical lemmas}

\begin{lemma}[Equation (9) in \cite{gordon1941values}]
\label{lemma: gaussian tail bounds}
    Let $\Phi(\cdot)$ denote the cumulative distribution function of standard Gaussian distribution. Then for all $x\geq 0$, the following inequalities are true:
    \[
    \left(\frac{x}{1+x^2}\right) \frac{e^{-x^2/2}}{\sqrt{2\pi}} \leq 1- \Phi(x) \leq \left(\frac{1}{x}\right)\frac{e^{-x^2/2}}{\sqrt{2\pi}}.
    \]
\end{lemma}

\begin{lemma}
    \label{lemma: similar Gaussians are hard to distinguish}
    Let $\bw\sim \sfN(\boldsymbol{0}, \sigma^2\bbI_n)$ and $\bmu\in \bbR^n \setminus \{0\}$ such that $\norm{\bmu}_2 \leq \sigma \delta $. Then 
    \[
    \pr (\norm{\bw + \bmu}_2^2 < \norm{\bw}_2^2) \geq \frac{\delta}{1+ \delta^2} \frac{e^{-\delta^2/2}}{\sqrt{2 \pi}}.
    \]
\end{lemma}

\begin{proof}
    By straightforward algebra, it follows that
    \[
    \sf{p}_0 := \pr (\norm{\bw + \bmu}_2^2 < \norm{\bw}_2^2) = \pr (\bmu^\top \bw + \norm{\bmu}_2^2< 0).
    \]
    Note that $\bmu^\top \bw \overset{\rm d}{=} \norm{\bmu}_2 w$, where $w \sim \sfN(0,\sigma^2)$. Hence, due to Lemma \ref{lemma: gaussian tail bounds} we have 
    \begin{align*}
        \sf{p}_0 & = \pr (w < - \norm{\bmu}_2)\\
        & = \pr (w >  \norm{\bmu}_2)\\
        & \geq \pr (w > \sigma \delta )\\
        & \geq \frac{\delta}{1+ \delta^2} \frac{e^{-\delta^2/2}}{\sqrt{2 \pi}}.
    \end{align*}
    
\end{proof}

\begin{lemma}[Lemma 1 in \cite{laurent2000adaptive}]\label{lemma: chi_tail_bound}
Let $W$ be chi-squared random variable with degrees of freedom $m$. Then, we have the following large-deviation inequalities for all $x>0$
\begin{equation}
\mathbb{P}(W-m>2\sqrt{mx}+2x)\leq \exp(-x),\quad and
\label{eq: LM1}
\end{equation}

\begin{equation}
\mathbb{P}(W-m<-2\sqrt{mx})\leq \exp(-x).
\label{eq: LM2}
\end{equation}


\end{lemma}
If we set $x = mu$ in Equation \eqref{eq: LM1} for $u>0$, then we get
\[
\pr \left(\frac{W}{m} -1 \geq 2 \sqrt{u} + 2 u \right) \leq \exp(-mu).
\]
Note that for $u <1$, we have $2\sqrt{u} + 2u< 4 \sqrt{u}$. Thus, setting $u = \delta^2/16$ for any $\delta <1$, we get
\begin{equation}
\label{eq: LM1-simple}
    \pr \left( \frac{W}{m} - 1 \geq \delta\right) \leq \exp(- m \delta^2/16).
\end{equation}
Similarly, setting $x = mu$ and $u = \delta^2/4$ in Equation \eqref{eq: LM2}, we get
\begin{equation}
    \label{eq: LM2-simple}
    \pr \left( \frac{W}{m} - 1 \leq -\delta\right) \leq \exp(- m \delta^2/4).
\end{equation}

\begin{lemma}
    \label{lemma: sqrt lipschitz}
    Let $\delta \in (0,\infty)$. Then for any $x>0$ the following inequality holds:
    \[
    0 <\sqrt{x + \delta} - \sqrt{(x-\delta) \vee 0} \leq \sqrt{2 \delta}.
    \]
\end{lemma}

\begin{proof}
It is obvious that $f_\delta(x):=\sqrt{x + \delta} - \sqrt{(x-\delta) \vee 0} >0$. Now for the other inequality, we will consider two cases:
\paragraph{Case 1: $x \leq \delta$}
In this case $f_\delta(x) = \sqrt{x+\delta} \leq \sqrt{2 \delta}$.

\paragraph{Case 2: $x>\delta$} In this case we have
\[
f_\delta^\prime(x) = \frac{1}{2}\left(\frac{1}{\sqrt{x+\delta} } - \frac{1}{\sqrt{x-\delta}}\right) <0 \quad \text{for all $x>\delta$.}
\]
Hence $f_\delta(x) \leq f_\delta(\delta) = \sqrt{2 \delta}$.
\end{proof}


\end{document}